\documentclass[11pt, reqno]{amsart}
\usepackage{amssymb}
\usepackage{amsmath}
\usepackage{amsthm}
\usepackage[dvipsnames]{xcolor}
\usepackage{braket}
\usepackage{hyperref}
\usepackage{pdfpages}
\usepackage{graphicx}
\usepackage{caption}
\usepackage{subcaption}
\usepackage{mathrsfs} 
\usepackage{tikz-cd} 
\usepackage{bm}
\usepackage{enumitem}
\usepackage{mathtools}
\usepackage{pgfplots}
\usepackage{import}
\usepackage{xifthen}
\usepackage{transparent}
\usepackage{tabularx}
\usepackage{hhline}
\usepackage{footnote}
\makesavenoteenv{tabular}
\usepackage{pgf}
\usepackage{tikz}
\usepackage[utf8]{inputenc}
\usetikzlibrary{arrows,automata}
\usetikzlibrary{positioning}
\usepackage{mathptmx}
\usetikzlibrary{calc}
\usepackage{scrextend}
\usepackage[colorinlistoftodos]{todonotes}

\newcolumntype{L}[1]{>{\raggedright\arraybackslash}p{#1}}
\newcolumntype{C}[1]{>{\centering\arraybackslash}p{#1}}
\newcolumntype{R}[1]{>{\raggedleft\arraybackslash}p{#1}}

\usepackage[utf8]{inputenc}
\usepackage[english]{babel}

\usepackage{hyperref}

\pgfplotsset{compat=1.18}

\tikzset{
    state/.style={
           rectangle,
           rounded corners,
           draw=black, very thick,
           minimum height=2em,
           inner sep=2pt,
           text centered,
           },
}

\usepackage{lmodern,babel,adjustbox,booktabs,multirow}

\numberwithin{equation}{section}
\theoremstyle{plain}
\newtheorem{theorem}{Theorem}[section]
\newtheorem*{theorem*}{Theorem}
\newtheorem{theoremx}{Theorem}
\newtheorem{corx}[theoremx]{Corollary}

\theoremstyle{plain} 
\newcommand{\thistheoremname}{}
\newtheorem{genericthm}[theorem]{\thistheoremname}

\newtheorem*{genericthm*}{\thistheoremname}
\newenvironment{namedthm*}[1]
  {\renewcommand{\thistheoremname}{#1}%
   \begin{genericthm*}}
  {\end{genericthm*}}

\theoremstyle{plain}
\newtheorem{prop}[theorem]{Proposition}
\newtheorem{lem}[theorem]{Lemma}
\newtheorem{cor}[theorem]{Corollary}

\newtheorem{question}[theorem]{Question}
\newtheorem*{question*}{Question}

\theoremstyle{definition}
\newtheorem{defn}[theorem]{Definition}
\newtheorem{rmk}[theorem]{Remark}
\newtheorem*{remark}{Remark}

\newcommand{\R}{\mathbb{R}}

\newcommand{\C}{\mathbb{C}}

\newcommand{\Z}{\mathbb{Z}}
\newcommand{\Hyp}{\mathbb{H}}

\newcommand{\N}{\mathbb{N}}

\DeclareMathOperator{\PSL}{PSL}
\DeclareMathOperator{\ext}{ext}

\DeclareMathOperator{\Isom}{Isom}
\DeclareMathOperator{\Conf}{Conf}

\DeclareMathOperator{\Mod}{Mod}
\DeclareMathOperator{\Int}{Int}
\DeclareMathOperator{\chull}{Cvx\, Hull}

\DeclareMathOperator{\Aut}{Aut}
\DeclareMathOperator{\Teich}{Teich}
\DeclareMathOperator{\core}{core}
\DeclareMathOperator{\Homeo}{Homeo}

\numberwithin{figure}{section}

\makeatletter
\DeclareFontFamily{U}{tipa}{}
\DeclareFontShape{U}{tipa}{m}{n}{<->tipa10}{}
\newcommand{\arc@char}{{\usefont{U}{tipa}{m}{n}\symbol{62}}}

\newcommand{\arc}[1]{\mathpalette\arc@arc{#1}}

\newcommand{\arc@arc}[2]{%
  \sbox0{$\m@th#1#2$}%
  \vbox{
    \hbox{\resizebox{\wd0}{\height}{\arc@char}}
    \nointerlineskip
    \box0
  }%
}
\makeatother

\title{Circle packings, renormalizations and subdivision rules}
\begin{author}[Y.~Luo]{Yusheng Luo}
\address{Department of Mathematics, Cornell University, 212 Garden Ave, Ithaca, NY 14853, USA}
\email{yusheng.s.luo@gmail.com}
\end{author}
\begin{author}[Y.~Zhang]{Yongquan Zhang}
\address{Institute for Mathematical Sciences, Stony Brook University, 100 Nicolls Rd, Stony Brook, NY 11794-3660, USA}
\email{yqzhangmath@gmail.com}
\end{author}
\thanks{The first-named author is partially supported by NSF Grant DMS-2349929}

\begin{document}

\begin{abstract}
    In this paper, we use iterations of skinning maps on Teichm\"uller spaces to study circle packings and develop a renormalization theory for circle packings whose nerves satisfy certain subdivision rules.
    We characterize when the skinning map has bounded image.
    Under the corresponding condition, we prove that the renormalization operator $\mathfrak{R}$ is uniformly contracting.
    This allows us to give complete answers for the existence and moduli problems for such circle packings. 
    The exponential contraction of $\mathfrak{R}^n$ means that despite the non-rigidity of such circle packings, they are geometrically inflexible.
    As an application, we show that any geometrically finite Kleinian circle packing is combinatorially rigid.
\end{abstract}

\maketitle

\tableofcontents

\section{Introduction}
A {\em circle packing} $\mathcal{P}$ of the Riemann sphere $\widehat\C$ is a collection of round disks of $\widehat\C$ with disjoint interiors.
Its {\em limit set} $\Lambda(\mathcal{P})$ is the closure of the union of all circles in $\mathcal{P}$.
Circle packings have been extensively studied from many different perspectives.
In addition to their intrinsic beauty, they have wide applications to combinatorics, geometry, probability theory, number theory and complex analysis.

Given a circle packing $\mathcal{P}$, we can associate to it a graph $\mathcal{G}$, called the {\em nerve} or the {\em tangency graph} of $\mathcal{P}$, as follows.
The nerve $\mathcal{G}$ assigns a vertex $v$ to each disk $D_v$ of $\mathcal{P}$ and an edge between $v, w$ if $\partial D_v \cap \partial D_w \neq \emptyset$.

It is easy to see that the nerve is a {\em simple} graph, i.e., it contains no multi-edges or self loops.
Since there is a natural embedding of $\mathcal{G}$ in $\widehat\C$, it is also a {\em plane graph}, i.e.\ a {planar graph together with a chosen embedding} into $\widehat\C$.
It is usually convenient to {\em mark} the circle packing $\mathcal{P}$, i.e., label each circle in $\mathcal{P}$ by the corresponding vertex in $\mathcal{G}$ (see \S \ref{sec:teich}).

Two marked circle packings $\mathcal{P}, \mathcal{P}'$ are (quasiconformally) homeomorphic if there exists a (quasiconformal) homeomorphism $\Psi:\widehat\C\longrightarrow \widehat\C$ that satisfies $\Psi(\mathcal{P})=\mathcal{P'}$ and preserves the marking.
The Teichm\"uller space $\Teich(\mathcal{G})$ of circle packings associated to $\mathcal{G}$ is the set of marked circle packings with nerve $\mathcal{G}$ up to M\"obius transformations (see Definition \ref{defn:teichCP}). If $\Teich(\mathcal{G})$ consists of a single point, we say the circle packing is \emph{combinatorially rigid}.

In this paper, we always assume that a circle packing has connected nerve.
For finite circle packings, the Koebe-Andreev-Thurston circle packing theorem {(see \cite[Theorem~1.4]{RHD07} or \cite[Chapter 12]{Thu22})} characterizes their existence and uniqueness.
\begin{namedthm*}{Finite Circle Packing Theorem}
Given a finite plane graph $\mathcal{G}$.
\begin{enumerate}[label=(\roman*)]
\item The graph $\mathcal{G}$ is isomorphic to the nerve of a finite circle packing $\mathcal{P}$ if and only if it is simple.
\item Suppose $\mathcal{G}$ is simple. Then 
\begin{itemize}
    \item any two circle packings with nerve $\mathcal{G}$ are quasiconformally homeomorphic, and hence $\Teich(\mathcal{G})$ is endowed with the Teichm\"uller metric;
    \item the Teichm\"uller space $\Teich(\mathcal{G})$ is isometrically homeomorphic to a product of Teichm\"uller spaces of ideal polygons
$$\prod_{F \text{ face of }\mathcal{G}} \Teich(\Pi_{F}).$$
In particular, the circle packing is combinatorially rigid if and only if its nerve gives a triangulation of the sphere.
\end{itemize}
\end{enumerate}
\end{namedthm*}
A detailed discussion of the Teichm\"uller spaces $\Teich(\Pi_F)$ can be found in \S\ref{sec:teich}. Briefly speaking, they record the conformal structures on the ideal polygons bounded by circles corresponding to vertices on $\partial F$.

It is both natural and important to understand the situation for infinite circle packings.
\begin{question}\label{q:m}
Given an infinite simple plane graph $\mathcal{G}$, when is $\mathcal{G}$ isomorphic to the nerve of an infinite circle packing $\mathcal{P}$?

If such an infinite circle packing exists, when is it combinatorially rigid? Or more generally, what is its Teichm\"uller space?
\end{question}

The answer is more complicated in the infinite setting.
It has been extensively studied in the literature for {\em locally finite triangulations}, and has generated numerous new tools and techniques (see \cite{RS87, He91, Sch91, HS93, HS94} or \S \ref{subsubsec:rt} for more discussions).

In this paper, we adapt new methods to study circle packings for {\em subdivision graphs} (cf.\ \cite{CFP01} and see Definition \ref{defn:fsr} or \S \ref{subsubsec:fsr} for more discussions).
Such circle packings arise naturally in conformal dynamics and low dimensional topology.
In particular, we give a complete answer to Question \ref{q:m} for graphs with finite subdivision rule (see {{Theorems} \ref{thm:LR} and \ref{thm:A}).

Our method uses iterations of {\em skinning maps} on Teichm\"uller spaces, and also establishes a natural connection to renormalization theory.
As an overview, we prove a relative version of Thurston's Bounded Image Theorem {(see \cite{Thu82} or \cite{Ken10})} in our setting, which provides precompactness and allows us to derive existence by taking limits of finite circle packings.
{Since the derivative of the skinning map is strictly smaller than $1$ (see \cite[Theorem 1.2]{McM90}, or a more precise version \cite[Theorem 5.3]{McM90} and Lemma~\ref{lem:nuc} in our setting)}, our Bounded Image Theorem implies that the skinning map is uniformly contracting.
By iterating this skinning map, we prove that 
\begin{itemize}
    \item (zooming out) finite circle packings for subdivision graphs stabilize exponentially fast (see Theorem \ref{thm:LB} and Theorem \ref{thm:B});
    \item (zooming in) renormalizations of infinite circle packings for subdivision graphs converge exponentially fast (see Theorem \ref{thm:EC}).
\end{itemize}
This allows us to establish many universality results, such as universal scaling and asymptotic conformality of local symmetries (see Theorem \ref{thm:LS}).

Our theory shares similarities with the renormalization theories of quadratic polynomials and Kleinian groups. 
To put our results in perspective, we provide a summary of these comparisons in Table~\ref{tab:renorm}. The connections between the second and third columns have been thoroughly studied and discussed in \cite{McM96}.

{The introduction below is organized as follows. In \S~\ref{ss:gfsr}, we introduce the notion of a finite subdivision rule and subdivision graph. In \S~\ref{ss:teich}, we state the existence and moduli theorem for infinite circle packings of a subdivision graph in Theorem~\ref{thm:LR}. We also state exponential convergence of finite circle packing approximations in Theorem~\ref{thm:LB}. In \S~\ref{ss:intro_renom}, we introduce renormalization theory on circle packings and state the geometric inflexibility Theorem~\ref{thm:EC}. This leads to universality results stated in Theorem~\ref{thm:LS}.
With applications in conformal dynamics in mind, we state the analogous theorems (Theorem~\ref{thm:A}, Theorem~\ref{thm:B}) for spherical subdivision graphs in \S~\ref{ss:sphericalsub}.
This leads to rigidity result for Kleinian circle packings (Theorem~\ref{thm:C}) in \S~\ref{ss:rigidityKleinian}.
Finally, we refer the reader to \S~\ref{sec:discussion} for background, motivation, and related work.}
\begin{table}[htp]
\centering
\begin{tabular}{|C{0.33\textwidth} | C{0.31\textwidth}|C{0.3\textwidth}|}
     \hline & & \\
\textbf{Circle packings} & \textbf{Quadratic maps} &\textbf{Kleinian groups} \\ & & \\ \hhline{|===|}
Circle packings $\mathcal{P}$    & Quadratic-like maps $f$ & Kleinian groups $\Gamma$ \\ \hline
Limit set $\Lambda(\mathcal{P})$    & Julia set $J(f)$ & Limit set $\Lambda(\Gamma)$ \\ \hline
Subdivision rule $\mathcal{R}$    & Kneading combinatorics & Mapping class $\psi:S \rightarrow S$ \\ \hline
Finite circle packing $\mathcal{P}^n$ for $\mathcal{R}$     & Finitely-renormalizable polynomial & Quasi-Fuchsian group \\ \hline
Circle packing $\mathcal{P}$ for subdivision graph     & $\infty$-renormalizable polynomial $f$ & Singly degenerate group $\Gamma$ \\ \hline
QC homeomorphism among $\mathcal{P}$ (Theorem \ref{thm:LR})     & QC conjugacy among $\infty$-renormalizable $f$ & QC conjugacy among singly degenerate $\Gamma$ \\ \hline
Circle packing $\mathcal{Q}$ for spherical subdivision     & Tower of renormalization & Doubly degenerate group \\ \hline
Rigidity of circle packings $\mathcal{Q}$   (Theorem \ref{thm:A})    & Rigidity of towers & Rigidity of double limits \\ \hline
Exponential convergence (Theorem \ref{thm:LB}, \ref{thm:EC} and \ref{thm:B})         &  Exponential convergence of $\mathfrak{R}^n(f)$ & Exponential convergence of $\psi^n(M)$ \\ \hline
Universality and $C^{1+\alpha}$ regularity (Theorem \ref{thm:LS})        & Universality and $C^{1+\alpha}$ regularity & Universality and $C^{1+\alpha}$ regularity
\\ \toprule
\end{tabular}
\caption{Parallel theories of renormalization}
\label{tab:renorm}
\end{table}

\subsection{Graphs with finite subdivision rule}\label{ss:gfsr}
We first recall that a CW complex $Y$ is a {\em subdivision} of a CW complex $X$ if $X = Y$ (as topological spaces) and every closed cell of $Y$ is contained in a closed cell of $X$.
We define a {\em polygon} as a finite CW complex homeomorphic to a closed disk that contains one 2-cell, with at least three 0-cells.
We will also call 0-cells, 1-cells and 2-cells the {\em vertices, edges} and {\em faces} respectively.
We say two vertices are {\em adjacent} if they are on the boundary of an edge, and {\em non-adjacent} otherwise.
\begin{defn}\label{defn:fsr}
A {\em finite subdivision rule} $\mathcal{R}$ consists of
\begin{enumerate}
\item\label{defn:fsr:1} a finite collection of {oriented} polygons $\{P_i: i=1,..., k\}$;
\item\label{defn:fsr:2} a subdivision $\mathcal{R}(P_i)$ that decomposes each polygon $P_i$ into $m_i \geq 2$ closed 2-cells
$$P_i = \bigcup_{j=1}^{m_i} P_{i, j}$$
so that each edge of $\partial P_i$ contains no vertices of $\mathcal{R}(P_i)$ in its interior {and each $2$-cell $P_{i,j}$ inherits the orientation of $P_i$, and}
\item\label{defn:fsr:3} {a map $\sigma \colon \bigcup_{i=1}^k (\{i\}\times \{1,\dots,m_i\})\to \{1,\dots,k\}$ and a collection of {orientation-preserving} cellular maps, denoted by
$$
\psi_{i,j}: P_{\sigma(i,j)} \to P_{i,j},\quad i\in \{1,\dots,k\}, \,\, j\in \{1,\dots,m_i\},
$$
that are homeomorphisms between the open 2-cells. We call $P_{\sigma(i,j)}$ the \textit{type} of the closed 2-cell $P_{i,j}$.}
\end{enumerate}
For simplicity, we use the notation $\mathcal R=\{P_i\}_{i=1}^k$ for the finite subdivision rule $\mathcal R$, {suppressing the decomposition in \eqref{defn:fsr:2} and the cellular maps in \eqref{defn:fsr:3}.}
\end{defn}
By condition \eqref{defn:fsr:3}, we can iterate the procedure and obtain $\mathcal{R}^n(P_i)$.
By condition \eqref{defn:fsr:2}, we can identify the $1$-skeleton $\mathcal{G}_i^n$ of $\mathcal{R}^{n}(P_i)$ as a subgraph of the $1$-skeleton $\mathcal{G}_{i}^{n+1}$ of $\mathcal{R}^{n+1}(P_i)$.
We denote the direct limit  by 
$$
\mathcal{G}_{i} = \lim_{\rightarrow} \mathcal{G}_i^n= \bigcup_n \mathcal{G}_{i}^n
$$ 
for $i=1,.., k$.
We call $\mathcal{G}_i$ the {\em subdivision graphs} for $\mathcal{R}$.
We shall embed $P_i$ {and finite graphs $\mathcal{G}_i^n$} in $\widehat\C$ and {call each complementary component of $\mathcal{G}_i^n$ a face of $\mathcal{G}_i^n$.}
We denote $F_i^{\ext}:=\overline{\widehat\C - P_i}$, and call it the {\em external face}.
Note that {$F_i^{\ext}$} is not subdivided, and remains a face of $\mathcal{G}_i^n$ for all $n \geq 0$.


\begin{defn}\label{defn:cyl}
Let $\mathcal{R}=\{P_i\}_{i=1}^k$ be a finite subdivision rule. We say $\mathcal R$ is
\begin{itemize}
    \item {{\em simple} if, for each $i\in \{1,\dots,k\}$, $\mathcal{G}(P_i)$ is a simple graph, i.e., no pair of vertices is connected by multiple edges and no edge connects a vertex to itself;}
    \item {\em irreducible} if for each $i\in \{1,\dots,k\}$, $\partial P_i$ is an induced subgraph of $\mathcal{G}_i$ i.e., $\partial P_i$ contains all edges of $\mathcal{G}_i$ that connects vertices in $\partial P_i$;
    \item {\em acylindrical} if for each $i\in \{1,\dots,k\}$ and any pair of non-adjacent vertices $v,w \in \partial P_i$, the two components of $\partial P_i - \{v,w\}$ are connected in $\mathcal{G}_i - \{v, w\}$.
    We call $\mathcal{R}$ {\em cylindrical} otherwise.
\end{itemize}
\end{defn}
We remark that by cutting each polygon $P_i$ into finitely many pieces if necessary, we can always make a finite subdivision rule $\mathcal{R}$ irreducible.
{We remark that from the definition, if $\mathcal{R}$ is cylindrical, then there exists $i \in \{1,\dots, k\}$ and a pair of non-adjacent vertices $v, w\in \partial P_i$ so that for each $n \geq 1$, the two components of $\partial P_i - \{v, w\}$ are disconnected in $\mathcal{G}_i^n - \{v, w\}$.}

\begin{figure}[htp]
    \centering
    \begin{subfigure}[b]{0.7\textwidth}
        \centering
        \includegraphics[width=\textwidth]
    {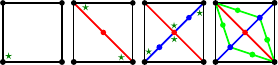}
    \caption{}
    \label{fig:subda}
    \end{subfigure}
    \begin{subfigure}[b]{0.7\textwidth}
        \centering
        \includegraphics[width=\textwidth]
    {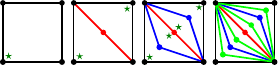}
    \caption{}
    \end{subfigure}
    \caption{An example of an acylindrical subdivision rule on the top, and a cylindrical subdivision rule on the bottom, {where each rule is iterated thrice. In both examples, a quadrilateral is divided into two quadrilaterals, but with different identifications. Each face is identified with the original quadrilateral by an orientation-preserving homeomorphism so that the corners marked by a star match.}}
    \label{fig:subd}
\end{figure}

\subsection{Teichm\"uller space and finite approximation}\label{ss:teich}
Our first result gives a precise description of the Teichm\"uller space of subdivision graphs for $\mathcal{R}$.
We remark that the techniques we use can be applied to more general subdivision rules. To simplify the presentation, we state our theorems for finite subdivision rules, and refer the readers to Appendix~\ref{sec:z2sr} for many generalizations.
\begin{theoremx}\label{thm:LR}
Let $\mathcal{R} =\{P_i\}_{i=1}^k$ be a simple, irreducible finite subdivision rule, with associated subdivision graphs $\{\mathcal{G}_i\}_{i=1}^k$.
\begin{enumerate}[label=(\roman*)]
\item The subdivision graphs $\mathcal{G}_i$ are isomorphic to the nerves of infinite circle packings for all $i\in \{1,\dots,k\}$ if and only if $\mathcal{R}$ is acylindrical;
\item Suppose that $\mathcal{R}$ is acylindrical. 
Then 
\begin{enumerate}
    \item any two circle packings $\mathcal{P}, \mathcal{P}' \in \Teich(\mathcal{G}_i)$ are quasiconformally homeomorphic;
    \item the Teichm\"uller space $\Teich(\mathcal{G}_i)$ 
is isometrically homeomorphic to the Teichm\"uller space $\Teich(\Pi_{F_i^{\ext}})$ of the ideal polygon associated to the external face $F_i^{\ext}$ of $\mathcal{G}_i$. 
\end{enumerate}
\end{enumerate}
\end{theoremx}

\begin{remark}
    We remark that by (a), there is a natural Teichm\"uller metric on $\Teich(\mathcal{G}_i)$, which is the metric used in (b).
    We also note that if $F_i^{\ext}$ has $e_i$ edges, then $\Teich(\Pi_{F_i^{\ext}})$ is homeomorphic to $\R^{e_i-3}$.
    Thus, the circle packing with nerve $\mathcal{G}_i$ is combinatorially rigid if and only if $e_i = 3$.
\end{remark}

\subsection*{Exponential convergence of finite approximation}
{We now consider the approximation of infinite circle packings with finite circle packings.}
Let $\mathcal{P}, \mathcal{P}' \in \Teich(\mathcal{G}^n_i)$ be two finite circle packings.
We say they have the same {\em external class} if $\pi_{F_i^{\ext}}(\mathcal{P}) = \pi_{F_i^{\ext}}(\mathcal{P}')$, where 
$$
\pi_{F_i^{\ext}}: \Teich(\mathcal{G}^n_i) = \prod_{F \text{ face of }\mathcal{G}^n_i} \Teich(\Pi_F) \longrightarrow \Teich(\Pi_{F_i^{\ext}})
$$ 
is the projection map to the Teichm\"uller space for the external face $F_i^{\ext}$.


Suppose $\mathcal{R}$ is acylindrical. 
Theorem \ref{thm:LR} implies that there exists a unique infinite circle packing for each external class.
The existence and uniqueness are intimately connected. 
The infinite circle packing is constructed as the limit of finite circle packings for $\mathcal{G}^n_i$ with fixed external class as $n\to\infty$.
To put our rigidity result in perspective, note that a circle packing for $\mathcal{G}^n_i$ with a fixed external class is not rigid unless all non-external faces of $\mathcal{G}^n_i$ are triangles.
Our second result shows that these finite circle packings stabilize exponentially fast (see Figure \ref{fig:fa}, c.f. \cite{He91} and \S~\ref{subsec:rc}).
\begin{theoremx}\label{thm:LB}
Let $\mathcal{R} =\{P_i\}_{i=1}^k$ be a simple, irreducible, acylindrical finite subdivision rule, with associated subdivision graphs $\{\mathcal{G}_i\}_{i=1}^k$.
Then there exist constants $C$, $n_0 \in \N$ and $ \delta < 1$ {depending only on the subdivision rule $\mathcal{R}$} so that for any $j$, the following holds.

Let $n \geq n_0$ and let $\mathcal{P}, \mathcal{Q}\in \Teich(\mathcal{G}_i^{n+j})$ be two finite circle packings with the same external class.
Let $\mathcal{P}^j, \mathcal{Q}^j$ be the corresponding sub-circle packings of $\mathcal{P}, \mathcal{Q}$ associated to the subgraph $\mathcal{G}^j_i \subseteq \mathcal{G}^{n+j}_i$.
Then
$$
d(\mathcal{P}^j, \mathcal{Q}^j) \leq \min\{C, d(\mathcal{P}, \mathcal{Q})\} \cdot \delta^{n-n_0},
$$
where $d$ is the Teichm\"uller distance between the two circle packings.
\end{theoremx}

We remark that this means that there exists a quasiconformal homeomorphism
$\Psi: \widehat\C \longrightarrow \widehat\C$ between $\mathcal{P}^j$ and $\mathcal{Q}^j$
whose dilatation $K(\Psi)$ satisfies
$$
\frac{1}{2} \log K(\Psi) \leq \min\{C, d(\mathcal{P}, \mathcal{Q})\} \cdot\delta^{n-n_0}.
$$
In particular, if we normalize the circle packings by M\"obius transformations appropriately so that  $\Psi$ fixes $0, 1, \infty$, then 
$$
\| \Psi - id\|_{C^0} = O(\delta^{n}),
$$
where $\| - \|_{C^0}$ is the $C^0$ norm with respect to the spherical metric.

\begin{figure}[htp]
    \centering
    \includegraphics[width=0.32\textwidth]{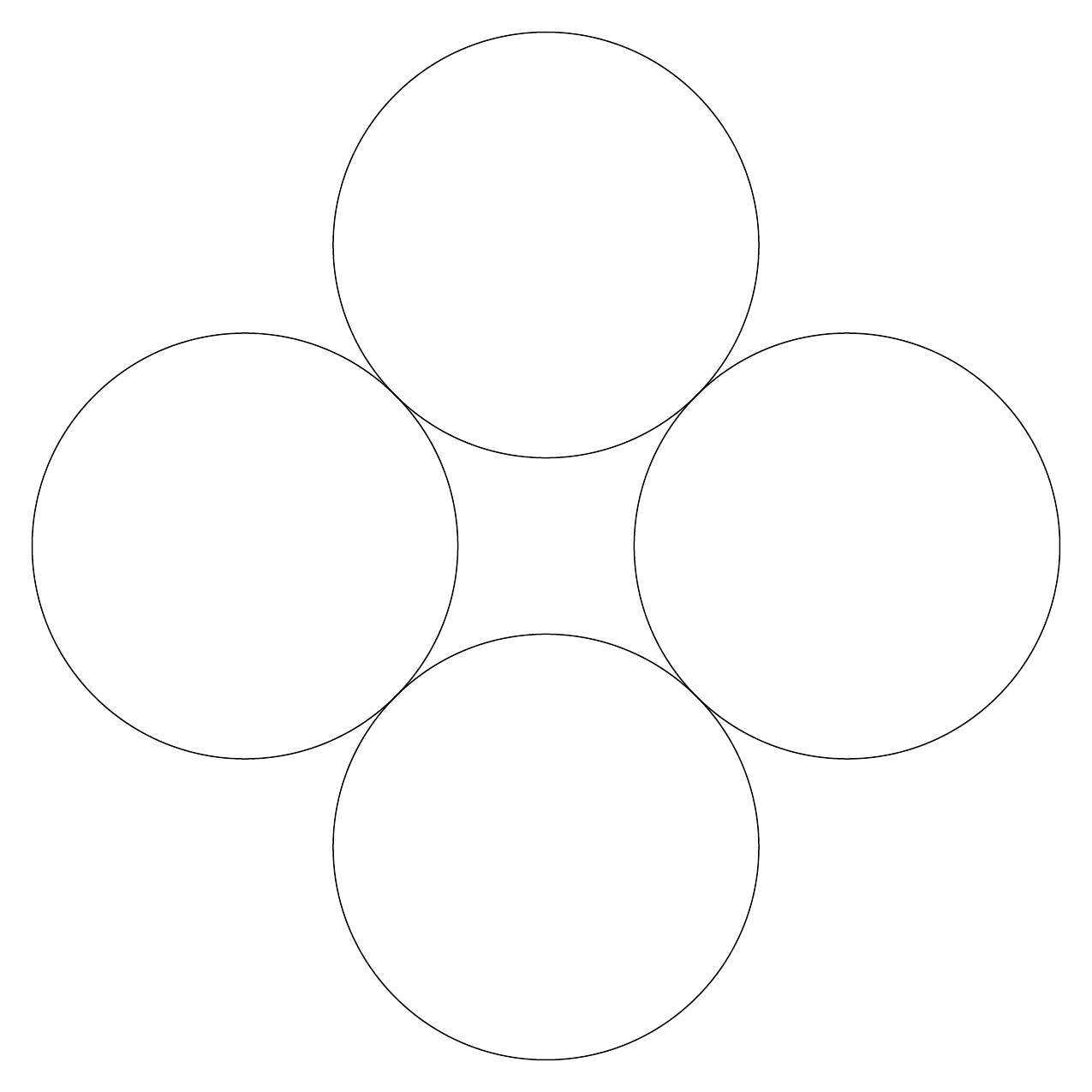}
    \includegraphics[width=0.32\textwidth]{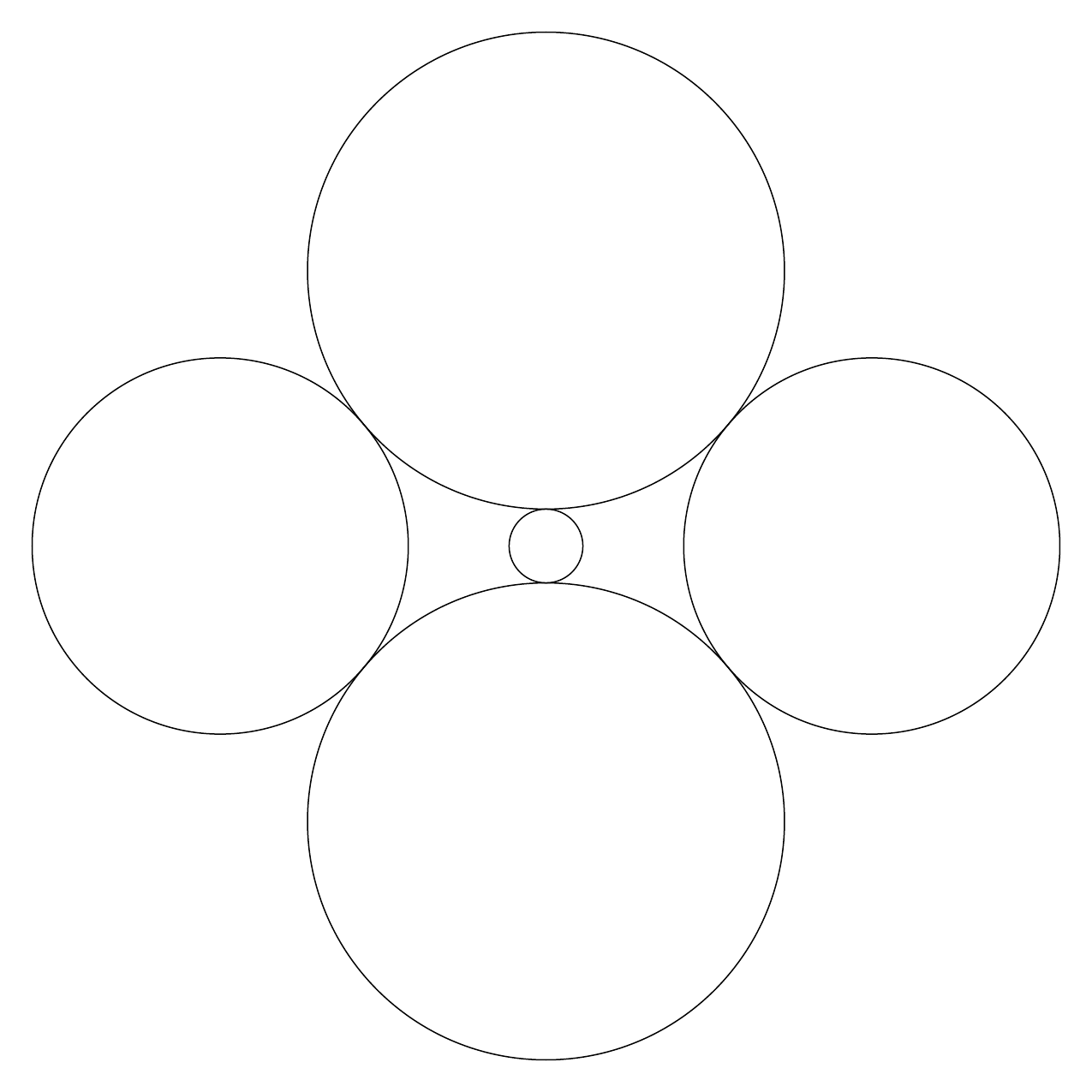}
    \includegraphics[width=0.32\textwidth]{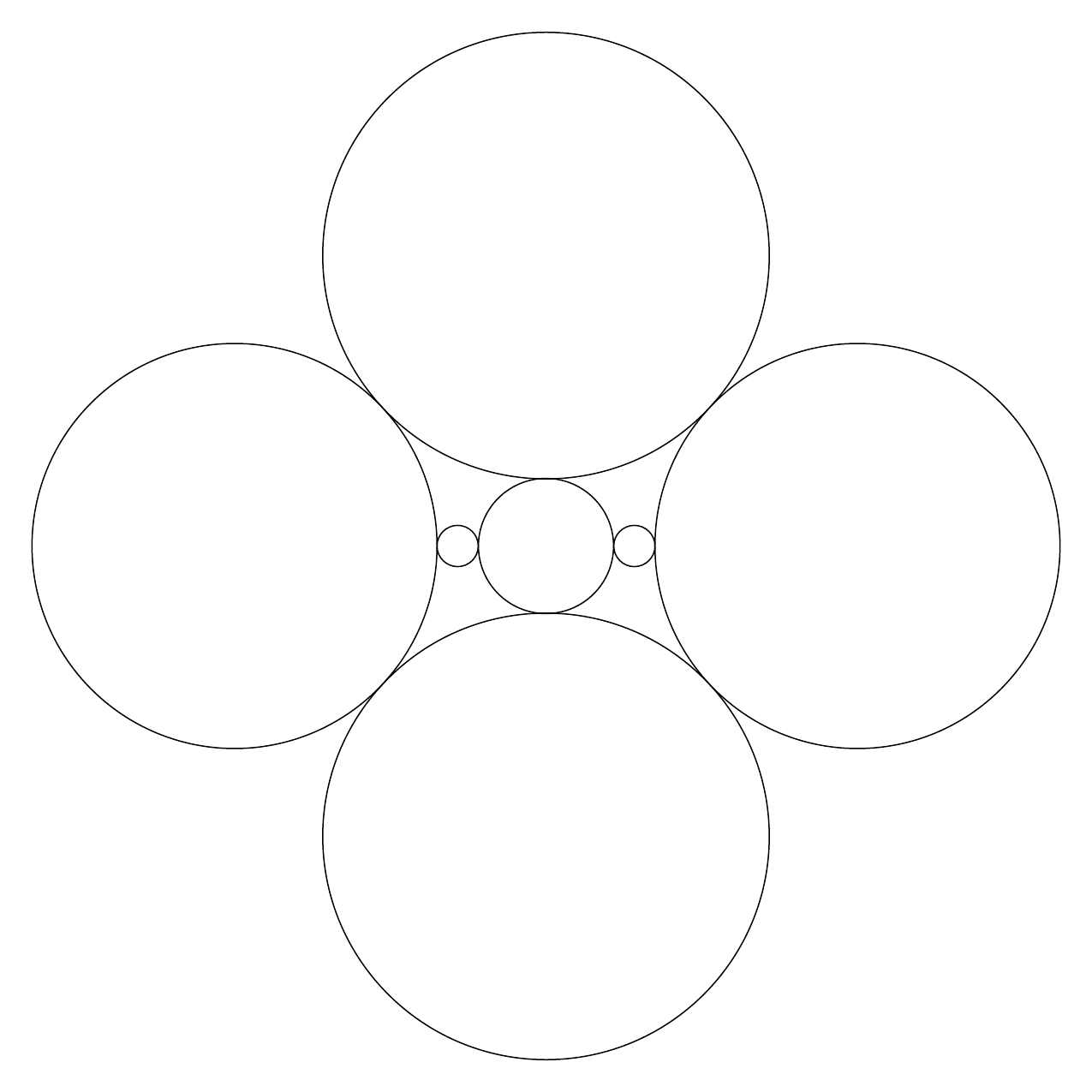}
    \includegraphics[width=0.32\textwidth]{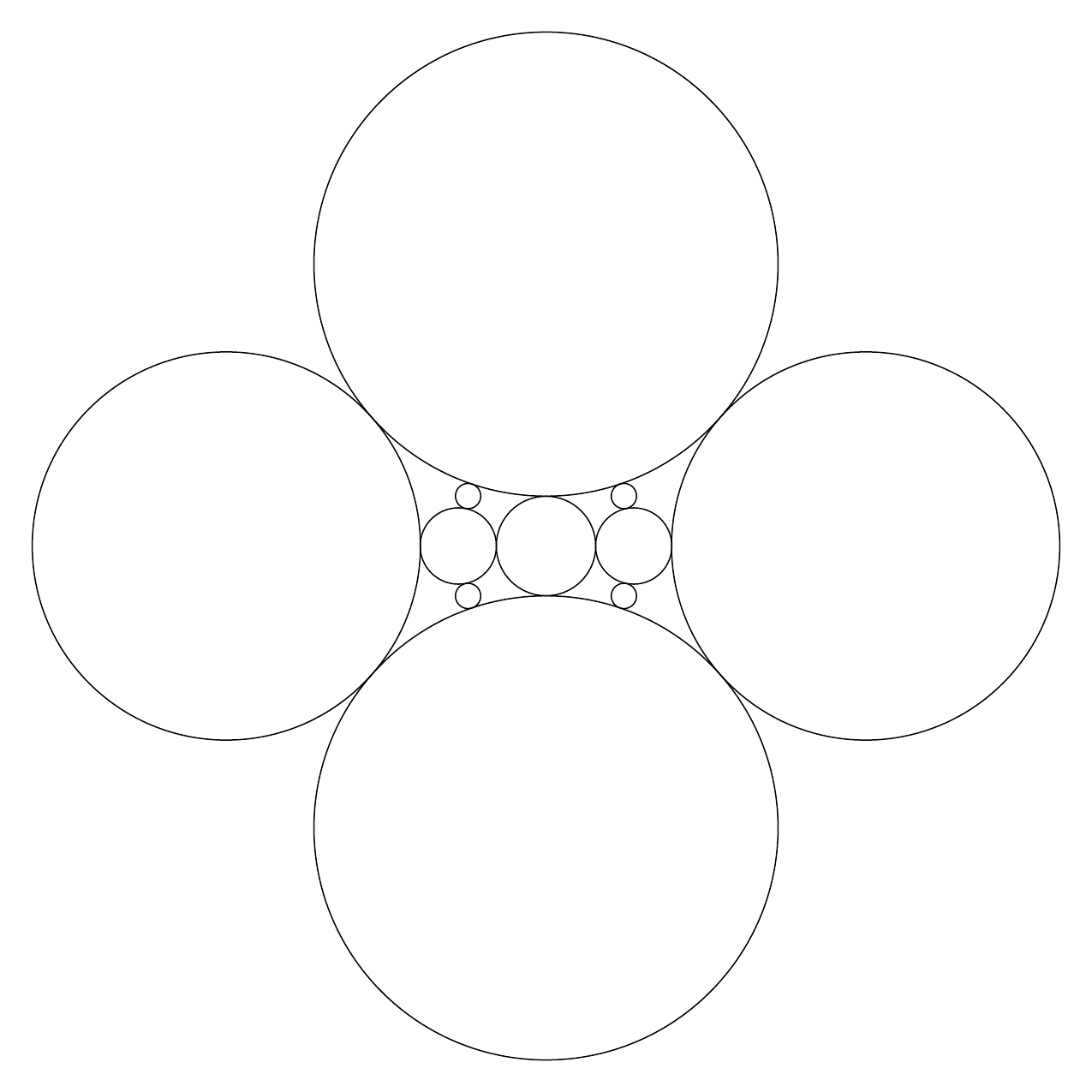}
    \includegraphics[width=0.32\textwidth]{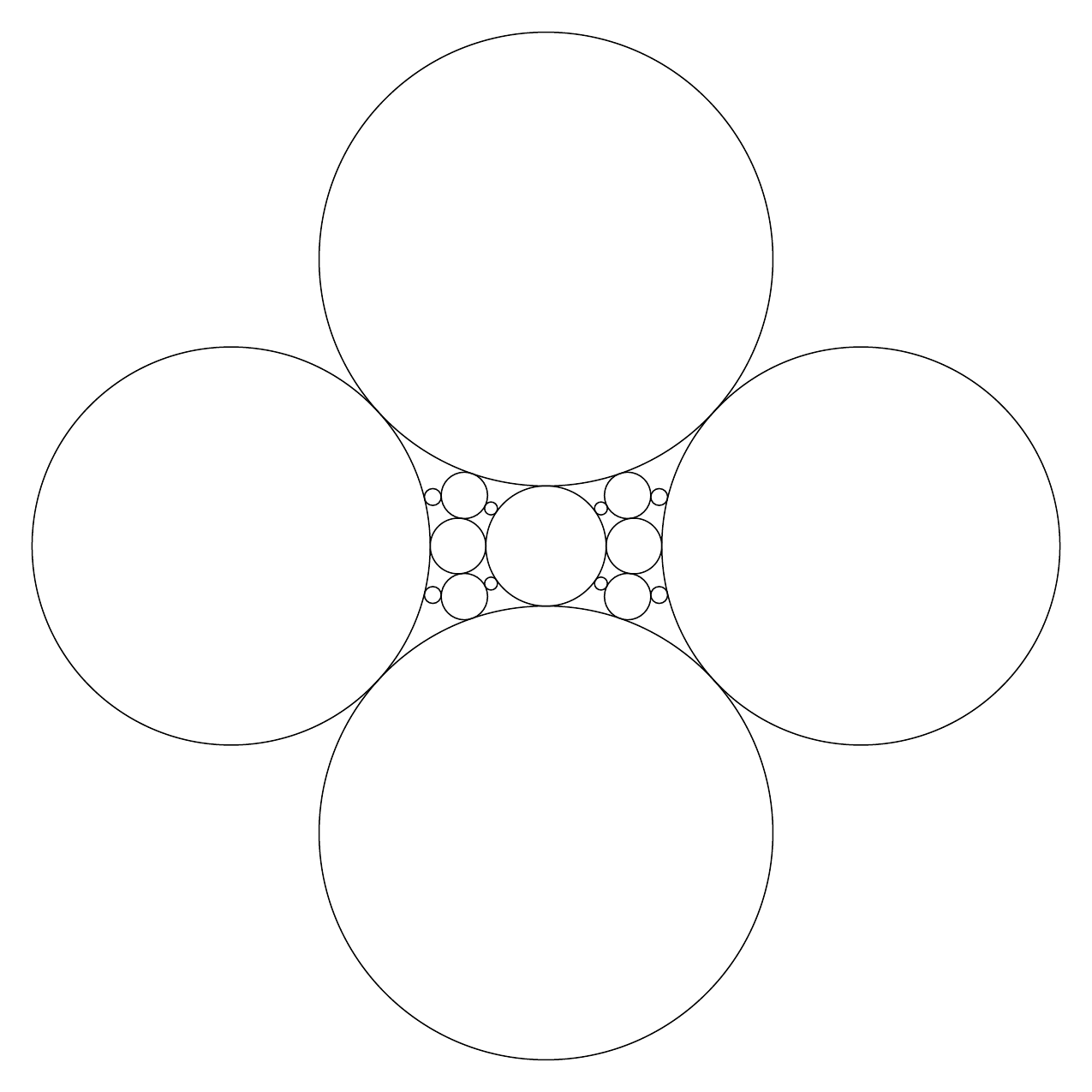}
    \includegraphics[width=0.32\textwidth]{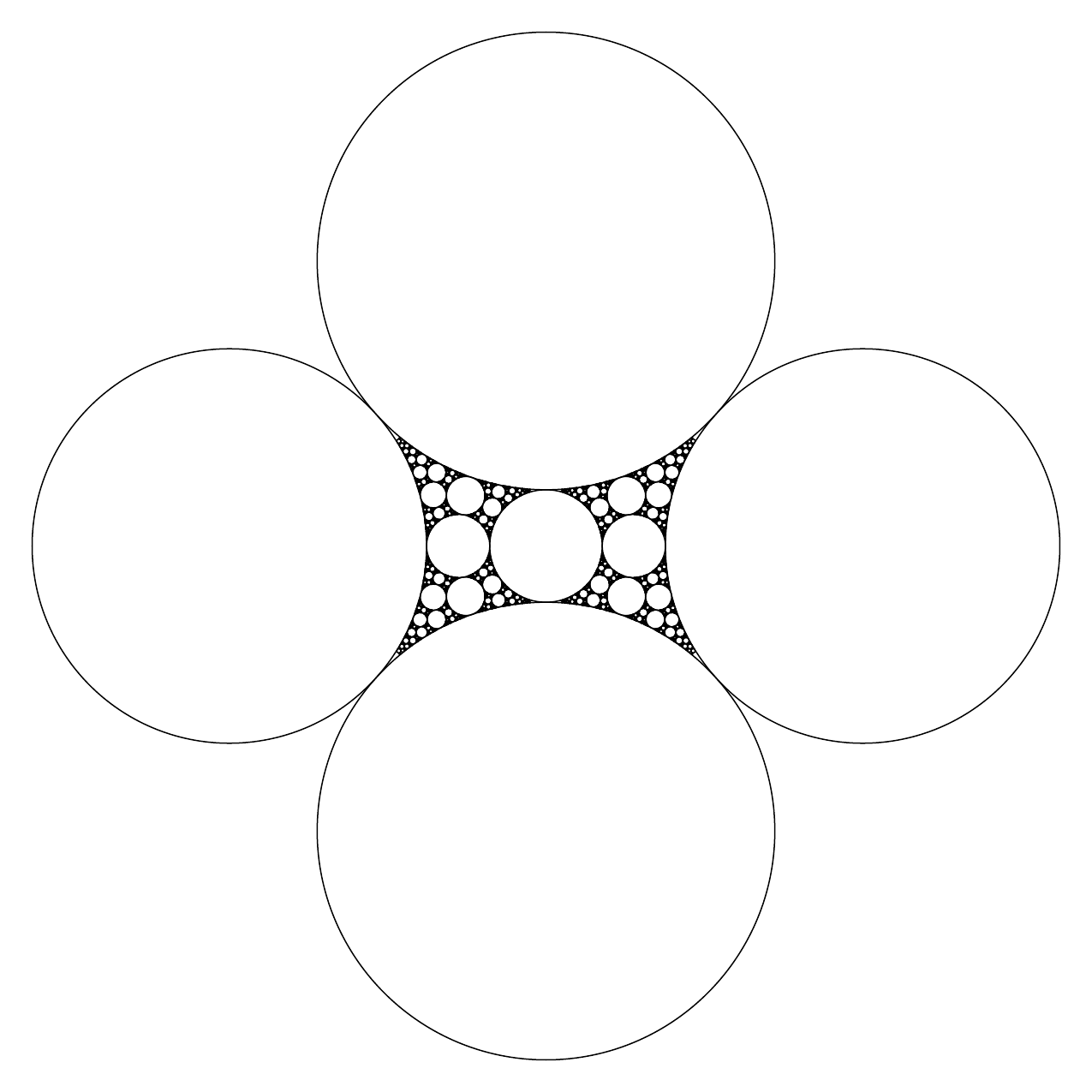}
    \caption{The finite circle packings of level $0, 1, 2, 3, 4, 15$ of the subdivision graph in Figure \ref{fig:subda}. 
    These circle packings are chosen so that one can fit in a circle touching all four sides of each complementary quadrilateral region. One notices that the circles at a given level stabilize quite rapidly as asserted by Theorem \ref{thm:LB}. {The region bounded by the outermost 4 circles is the skinning interstice for $\mathcal{P}$.}}
    \label{fig:fa}
\end{figure}

\subsection{Renormalizations and inflexibility}\label{ss:intro_renom}
{We now explain that infinite circle packings for a subdivision graph, while in general not rigid by Theorem~\ref{thm:LR}, are nevertheless {\em inflexible}. The inflexibility is manifest at small scales: any quasiconformal map between two such infinite circle packings is automatically differentiable and conformal on the limit set (Theorem~\ref{thm:EC}). To illustrate this phenomenon, we study renormalization on circle packings.}

Let $\mathcal{P}\in \Teich(\mathcal{G}_i)$ be an infinite circle packing.
{Since the finite subdivision rule $\mathcal{R}$ is assumed to be irreducible, the boundary of the external face $\partial F_i^{\ext} \subseteq \mathcal{G}_i$ is a Jordan curve. Therefore, the complement of the union of `outermost' closed disks 
$$
\bigcup_{v \text{ vertex } \partial F_i^{\ext}} \overline{D_v}
$$ 
consists of two regions, where $D_v$ is the disk in the circle packing $\mathcal{P}$ associated to the vertex $v$.}
One of the two, denoted by $\Pi(\mathcal{P})$, has non-trivial intersection with the limit set $\Lambda(\mathcal{P})$.
We call $\Pi(\mathcal{P})$ the {\em skinning interstice} for $\mathcal{P}$ (see Figure~\ref{fig:fa}).
Define the space
\begin{align*}
\Sigma:=\bigcup_{i=1}^k \bigg(\Big\{(\mathcal{P},x): &\mathcal{P} \text{ is a marked circle packing with nerve } \mathcal{G}_i, \\
&x\in \overline{\Pi(\mathcal{P})}\Big\}\big/\PSL_2(\C)\bigg).
\end{align*}
By Theorem \ref{thm:LR}, $\Sigma$ is homeomorphic to $\bigcup_{i=1}^k \R^{e_i-3} \times \overline{\Pi(\mathcal{P}_i)}$, where $\mathcal{P}_i$ is (any) circle packing with nerve $\mathcal{G}_i$. Moreover, any two pairs $(\mathcal{P}, x)$ and $(\mathcal{P}', x')$ on the same fiber are quasiconformally homeomorphic.

\begin{figure}[htp]
    \centering
    \includegraphics[width=0.8\textwidth]{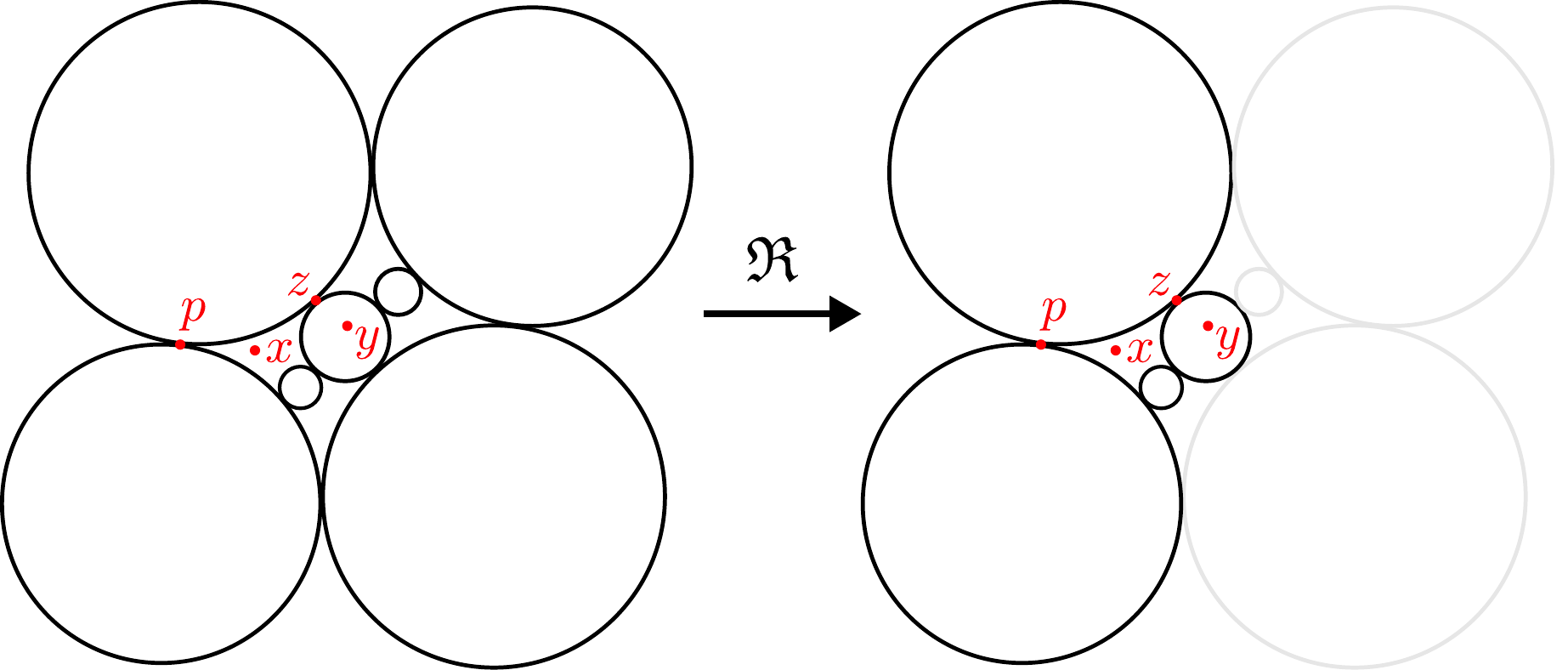}
    \caption{An illustration of the renormalization operator. Note that $(\mathcal{P}, x), (\mathcal{P}, z), (\mathcal{P}, p) \in \Sigma^1 \subseteq \Sigma$, but $(\mathcal{P}, x) \in \Sigma - \Sigma^1$. The point $z \in \partial \overline{\Pi(\mathcal{P}_F)} \cap \partial \overline{\Pi(\mathcal{P}_{F'})}$ for two non-external faces, so $(\mathcal{P}, z)$ lifts to two points in $\widetilde{\Sigma}^1$. On the other hand, $x, p \in \overline{\Pi(\mathcal{P}_F)}$ for some unique non-external face $F$, so $(\mathcal{P}, x)$ (or $(\mathcal{P}, p)$) lifts to a point in $\widetilde{\Sigma}^1$.}
    \label{fig:renom}
\end{figure}

{Let $F$ be a non-external face of $\mathcal{G}^1_i$, and $\mathcal{P}_F$ be the infinite sub-circle packing of $\mathcal{P}$ corresponding to the face $F$. We denote the corresponding skinning interstice by $\Pi(\mathcal{P}_F)$. Consider the space
\begin{align*}
\Sigma^1:=\bigcup_{i=1}^k \bigg(\Big\{(\mathcal{P},x): &\mathcal{P} \text{ is a marked circle packing with nerve } \mathcal{G}_i, \\
&x\in \bigcup_F\overline{\Pi(\mathcal{P}_F)}\Big\}\big/\PSL_2(\C)\bigg) \subseteq \Sigma.
\end{align*}
where the union inside $\{\}$ is taken over all non-external faces of $\mathcal{G}^1_i$.
Then we can define the renormalization of $(\mathcal{P}, x) \in \Sigma^1$ by
$$
\mathfrak{R}((\mathcal{P}, x)) = (\mathcal{P}_F, x) \in \Sigma,
$$
if $x \in \overline{\Pi(\mathcal{P}_F)}$.}
We remark that technically, if $x\in \partial \overline{\Pi(\mathcal{P}_F)} \cap \partial \overline{\Pi(\mathcal{P}_{F'})}$, then the definition of $\mathfrak{R}$ requires a choice.
This point $x$ is necessarily the tangent point where two circles touch.
We can resolve this issue by blowing up such a tangent point into two points, and denote this new space by $\widetilde{\Sigma}$ with the projection map $\pi: \widetilde{\Sigma} \longrightarrow \Sigma$ {(see \S~\ref{ss:renom} for more details)}.
In this way, we can construct a renormalization operator 
$$
\mathfrak{R}: \widetilde{\Sigma}^1 \subseteq \widetilde{\Sigma} \longrightarrow \widetilde{\Sigma}.
$$ 
We say $(\mathcal{P}, x)$ is {\em infinitely renormalizable} if $\mathfrak{R}^n$ is defined at $\pi^{-1}((\mathcal{P}, x))$ for all $n$.
{Note that $(\mathcal{P}, x)$ is infinitely renormalizable if and only if $x$ is in an infinite nested skinning interstices. This happens if and only if $x \in \Lambda(\mathcal{P})\cap \overline{\Pi(\mathcal{P})}$ (see also \S~\ref{ss:renom}).}
Thus, the non-escaping locus $\widetilde{\Sigma}^\infty$ is homeomorphic to a fibration over a Cantor set, whose fibers are Euclidean spaces.
It projects via $\pi$ to $\Sigma^\infty$, which is homeomorphic to $\bigcup_{i=1}^k \R^{e_i-3}\times (\Lambda(\mathcal{P}_i) \cap \overline{\Pi(\mathcal{P}_i)}).
$

The following theorem states that as we zoom in at a limit point, a homeomorphism between two circle packings for subdivision graphs converges exponentially fast to a conformal map.
\begin{theoremx}[Geometric inflexibility]\label{thm:EC}
Let $\mathcal{R} =\{P_i\}_{i=1}^k$ be a simple, irreducible, acylindrical finite subdivision rule, with associated subdivision graphs $\{\mathcal{G}_i\}_{i=1}^k$.
Then there exist constants $C$, $n_0 \in \N$ and $ \delta < 1$ {depending only on the subdivision rule $\mathcal{R}$} so that the following holds.

Suppose that $(\mathcal{P}, x), (\mathcal{P}', x') \in \Sigma^\infty$ are homeomorphic.
Then for all $n\geq n_0$,
\begin{align*}
 d(\mathfrak{R}^n((\mathcal{P}, x)), \mathfrak{R}^n((\mathcal{P}', x')){)} \leq \min\{C, d((\mathcal{P}, x), (\mathcal{P}', x'))\}\cdot \delta^{n-n_0},
\end{align*}
where $d$ is the Teichm\"uller metric.
\end{theoremx}

We remark that $\mathfrak{R}^n((\mathcal{P}, x))$ is interpreted as $\pi(\mathfrak{R}^{n}(\pi^{-1}((\mathcal{P}, x))))$.
In the case when $\pi^{-1}((\mathcal{P}, x))$ consists of two points,
$$
\mathfrak{R}^n((\mathcal{P}, x)) = \{(\mathcal{P}_{1,n}, x_{1,n}), (\mathcal{P}_{2,n}, x_{2,n})\}
$$ 
consists of two points for all sufficiently large $n$.
The Teichm\"uller distance is the maximum of the Teichm\"uller distance on pairs of circle packings.

\subsection*{Universality and regularity}\label{subsec:univeralregular}
{We now state some consequences of Theorem~\ref{thm:EC} regarding the local geometry and universality of such circle packings (cf.\ \cite[\S 9.4]{McM96}).}

Let $X \subseteq \C$ be a subset that is not necessarily open.
A map $\Psi:X \longrightarrow \C$ is $C^{1+\alpha}$-conformal at $z$ if the complex derivative 
$$
\Psi'(z):= \lim_{z+t \in X, t\to 0} \frac{\Psi(z+t) - \Psi(z)}{t}
$$ 
exists, and
\begin{equation}\label{eqn:asymp}
\Psi(z+t) = \Psi(z) + \Psi'(z)t + O(|t|^{1+\alpha})
\end{equation}
for all $z+t \in X$ and $t$ sufficiently small (cf.\ \cite[\S B.6]{McM96}). 
{We say $\Psi$ is $C^{1+\alpha}$-conformal on $X$ if it is $C^{1+\alpha}$ at every point $z$ and the exponent $\alpha$ and constants in $O(|t|^{1+\alpha})$ is independent of $z \in X$.}

\begin{theoremx}\label{thm:LS}
Let $\mathcal{R} =\{P_i\}_{i=1}^k$ be a simple, irreducible, acylindrical finite subdivision rule, with associated subdivision graphs $\{\mathcal{G}_i\}_{i=1}^k$.
\begin{enumerate}
    \item(Asymptotic conformality) Let $\mathcal{P}, \mathcal{P}'\in \Teich(\mathcal{G}_i)$ normalized by M\"obius map so that $\Pi(\mathcal{P}), \Pi(\mathcal{P}') \subseteq \C$ and let $\Psi$ be any homeomorphism between the two circle packings. Then $\Psi|_{\Lambda(\mathcal{P})\cap \overline{\Pi(\mathcal{P})}}$ is $C^{1+\alpha}$-conformal.
    \item(Renormalization periodic point) Let $(\mathcal{P}_0, x_0)\in \Sigma^\infty$ be a renormalization periodic point with period $q$. Then there exists a M\"obius transformation $\phi \in\PSL_2(\C)$ so that $\phi((\Lambda(\mathcal{P}_q), x_0)) = (\Lambda(\mathcal{P}_0), x_0)$ where $\mathcal{P}_q$ is the sub-circle packing of $\mathcal{P}_0$ corresponding to $\mathfrak{R}^q(\mathcal{P}_0, x_0)$. The multiplier $\mu:= \phi'(x_0)$ is called the {\em scaling factor} for $(\mathcal{P}_0, x_0)$.
    \item(Hyperbolic vs parabolic) Either 
    \begin{itemize}
        \item $x_0$ is a cusp point of $\partial \overline{\Pi(\mathcal{P}_0)}$, in which case $\phi$ is a parabolic M\"obius transformation with a fixed point at $x_0$; or
        \item $x_0$ is not a cusp point of $\partial \overline{\Pi(\mathcal{P}_0)}$, in which case the scaling factor satisfies $|\mu| > 1$.
    \end{itemize} 
    \item(Universality) The scaling factor is universal, in the following sense.
    Let $\mathcal{Q}$ be any circle packing.
    Suppose there exists a homeomorphism $\Psi: \widehat\C \longrightarrow \widehat\C$ so that $\Psi(\Lambda(\mathcal{P}_0)) \subseteq \Lambda(\mathcal{Q})$.
    Then 
    \begin{itemize}
        \item $\Lambda(\mathcal{Q})$ is asymptotically self-similar at $\Psi(x_0)$ by a parabolic M\"obius transformation if $|\mu|=1$.
        \item $\Lambda(\mathcal{Q})$ is asymptotically self-similar at $\Psi(x_0)$ with scaling factor $\mu$ if $|\mu|>1$;
    \end{itemize}
\end{enumerate}
\end{theoremx}
\begin{remark}
    {We remark that by adding $\Psi(z+t) = \Psi(z) + \Psi'(z)t + O(|t|^{1+\alpha})$ with $\Psi(z) = \Psi(z+t) - \Psi'(z+t)t + O(|t|^{1+\alpha})$, we see that $C^{1+\alpha}$-conformality implies that $\Psi'(z)$ is $\alpha$-H\"older continuous on the limit set $\Lambda(\mathcal{P})\cap \overline{\Pi(\mathcal{P})}$ (c.f \cite[Proposition 8.2]{Mer12} and see \S~\ref{sss:qr}).}
    We also remark that since there is no additional restriction for the homeomorphism $\Psi$ on the interior of the disks in the circle packing, the asymptotic conformality only applies to the restriction of $\Psi$ on the limit set.

    In \S \ref{subsec:tm}, we introduce the notion of a {\em Teichm\"uller mapping} between two circle packings, which is a homeomorphism between the circle packings that satisfies some additional desired properties.
    We prove that a Teichm\"uller mapping is asymptotically conformal, {without restriction to the limit set}, at all combinatorially deep points (see Theorem \ref{thm:tmac} and Remark \ref{rmk:cm}).

    We also remark that since the scaling factor is $1$ at a cusp point of $\partial \overline{\Pi(\mathcal{P}_0)}$, the renormalization operator $\mathfrak{R}$ is {\em not} a hyperbolic operator on $\Sigma^\infty$.
\end{remark}

\subsection{Existence and rigidity for spherical subdivision graphs}\label{ss:sphericalsub}
Subdivision graphs for a finite subdivision rule $\mathcal{R}$ has a special external face that is not subdivided.
For many applications in dynamics and geometry, the circle packing is dense on the whole Riemann sphere.
Such circle packings can be modeled by {\em spherical subdivision graphs} as follows.
\subsection*{$\mathcal{R}$-complexes}
Let $\mathcal{R}=\{P_i\}_{i=1}^k$ be a finite subdivision rule.
An $\mathcal{R}$-complex is a finite CW complex $X$ so that 
for any closed 2-cell $S$, there exist $P_{\sigma(S)} \in \{P_i\}_{i=1}^k$ {where $\sigma(S)$ is the index,} and an orientation-preserving cellular map $\psi_S: P_{\sigma(S)} \longrightarrow S$ that is a homeomorphism between open 2-cells.
We call $P_{\sigma(S)}$ the {\em type} of $S$.

The subdivision rule gives a sequence of subdivisions $\mathcal{R}^n(X)$ for $X$.
We always assume that $\mathcal{R}$ is the {\em minimal} finite subdivision rule that supports $X$, i.e. every polygon in $\{P_i\}_{i=1}^k$ appears as the type of some closed 2-cell of $\mathcal{R}^n(X)$ for some $n$.
The 1-skeleton of subdivisions $\mathcal{R}^n(X)$ of $X$ gives a nested sequences of graphs
$$
\mathcal{G}^0 \subseteq \mathcal{G}^1 \subseteq \mathcal{G}^2 \subseteq ..., \text{ and } \mathcal{G} = \lim_{\rightarrow} \mathcal{G}^n = \bigcup_{n} \mathcal{G}^n.
$$

If the $\mathcal{R}$-complex $X$ is homeomorphic to the sphere $S^2$, the plane graph $\mathcal{G}$ is called a {\em spherical subdivision graph for $\mathcal{R}$}.
We remark that any spherical subdivision graph for $\mathcal{R}$ can be constructed by piecing together finitely many subdivision graphs of $\mathcal{R}$ along their boundaries.

Our next result gives a characterization of existence and rigidity.
\begin{theoremx}\label{thm:A}
Let $\mathcal{G}$ be a simple spherical subdivision graph for an irreducible finite subdivision rule $\mathcal{R}$.
\begin{enumerate}[label=(\roman*)]
\item The graph $\mathcal{G}$ is isomorphic to the nerve of an infinite circle packing $\mathcal{P}$ if and only if $\mathcal{R}$ is acylindrical.
\item Moreover, if such a circle packing exists, it is combinatorially rigid.
\end{enumerate}
\end{theoremx}


Similar to the case of subdivision graphs, finite approximations stabilize exponentially fast.
\begin{theoremx}\label{thm:B}
Let $\mathcal{G}$ be a simple spherical subdivision graph for an irreducible, acylindrical finite subdivision rule $\mathcal{R}$.
Then there exist constants $C$, $n_0 \in \N$ and $ \delta < 1$ {depending only on $\mathcal{G}$} so that for any $k$, the following holds.

Let {$n\geq n_0$ and let} $\mathcal{P}, \mathcal{Q}\in \Teich(\mathcal{G}^{n+j})$.
Let $\mathcal{P}^j, \mathcal{Q}^j$ be the corresponding sub-circle packings for the subgraph $\mathcal{G}^j \subseteq \mathcal{G}^{n+j}$.
Then 
$$
d(\mathcal{P}^j, \mathcal{Q}^j) \leq \min\{C, d(\mathcal{P}, \mathcal{Q})\} \cdot \delta^{n-n_0}.
$$
\end{theoremx}

\subsection{Rigidity of Kleinian circle packings}\label{ss:rigidityKleinian}
A circle packing $\mathcal{P}$ is called {\em Kleinian} if its limit set $\Lambda(\mathcal{P})$ is equal to the limit set $\Lambda(\Gamma)$ of some Kleinian group $\Gamma$.
The nerve of a Kleinian circle packing may not be a spherical subdivision graph for a finite subdivision rule, but instead follows a more general subdivision rule (see \S \ref{sec:z2sr} and \S \ref{sec:kcp}).
In fact, the nerve of a geometrically finite Kleinian circle packing is a graph with finite subdivision rule if and only if the Kleinian group has no rank 2 cusps {(see Theorem~\ref{thm:ksdr} and Appendix~\ref{appx:topology})}.
Our theory is well-suited for this broader context.
As an application, we prove the following.
\begin{theoremx}\label{thm:C}
Let $\mathcal{P}$ be a Kleinian circle packing. Suppose the corresponding Kleinian group $\Gamma$ is geometrically finite. Then $\mathcal{P}$ is combinatorially rigid.
\end{theoremx}

\begin{remark}
In contrast, geometrically infinite Kleinian circle packings are in general not combinatorially rigid.
Indeed, by changing the ending lamination associated to a geometrically infinite end carefully, it is possible to change the limit set while the nerve remains the same. {See the discussion in Appendix~\ref{appx:topology}, which relies on \cite[Theorem~1.3]{NS12}}.
\end{remark}

As an immediate corollary, we have that the {orientation-preserving} homeomorphism group equals the conformal symmetry group, extending the result for generalized Apollonian gaskets in \cite{LLMM23}.
\begin{corx}\label{cor:H=C}
  Let $\mathcal{P}$ be a geometrically finite Kleinian circle packing.
  Then 
  $$
  {\Homeo^+}(\Lambda(\mathcal{P})) = \Conf(\Lambda(\mathcal{P})).
  $$
\end{corx}

Since any Kleinian group with circle packing limit set is a finite index subgroup of the full conformal symmetry group of the limit set, we have
\begin{corx}\label{cor:comm}
  Let $\Gamma_1, \Gamma_2$ be two geometrically finite Kleinian groups with homeomorphic circle packing limit set.
  Then they are commensurable.
\end{corx}


\subsection{Discussion on related work}\label{sec:discussion}
\subsubsection{Rigidity of hexagonal circle packing and other triangulations}\label{subsubsec:rt}
Circle packings provide a conformally natural way to embed a planar graph into a surface.
In the 1980s, Thurston proposed a constructive, geometric approach to the Riemann mapping theorem using circle packings.
This conjecture was proved by Rodin-Sullivan in \cite{RS87}.
The proof depends crucially on the non-trivial {\em uniqueness} of the hexagonal packing as the only packing in the plane with the triangular lattice as its nerve.

Using combinatorial arguments, Schramm gave a generalization of the above rigidity result.
Let $\mathcal{P}$ be a circle packing on the Riemann sphere $\widehat \C$ whose nerve $\mathcal{G}$ is a planar triangulation.
In \cite[Theorem 1.1]{Sch91}, Schramm proved that if $\widehat\C - \text{carrier}(\mathcal{P})$ is at most countable, then $\mathcal{P}$ is combinatorially rigid.
Here the {\em carrier} of a packing $\mathcal{P}$ is the union of the closed disks in $\mathcal{P}$ and the `interstices' bounded by three mutually touching circles in the complement of the packing.
The rigidity of the hexagonal packing follows immediately, since the carrier is the whole complex plane.

He and Schramm continued to develop many new tools to weaken the hypothesis of above result.
In \cite{HS93}, they proved a rigidity result with the hypothesis that $\widehat\C - \text{carrier}(\mathcal{P})$ is a countable union of points or disks.
In \cite{HS94}, 
they proved a rigidity result for $\mathcal{P}$ with possibly uncountably many one-point components in $\widehat\C - \text{carrier}(\mathcal{P})$, but under the hypothesis that the boundary has $\sigma$-finite linear measure.
These rigidity results have a close connection to the Koebe's conjecture that every planar domain can be mapped conformally onto a circle domain, and the techniques in \cite{HS93, HS94} led to a breakthrough towards the Koebe's conjecture.

The concept of circle packings can be generalized to any closed orientable surface with a projective structure. A version of the Koebe-Andreev-Thurston theorem in this setting already appeared in \cite{Thu82}. Since a surface of genus $g\ge1$ in general supports many different projective structures, a natural question is to determine the subspace of projective structures supporting a circle packing with a fixed nerve. See \cite{KMT03, KMT06} for results in this direction.

In the above mentioned work, the methods require crucially that the nerve is a triangulation.
We remark that in our setting, the graph $\mathcal{G}$ is never a planar triangulation, nor locally finite.

\subsubsection{Rate of convergence of finite circle packings}\label{subsec:rc}
In his thesis \cite{He91}, He provided a quantitative estimate for the rate of convergence of finite approximation of the hexagonal circle packings.
Let $\mathcal{G}$ be the nerve of the regular hexagonal circle packing.
Then we have a natural level structure
$
\mathcal{G}^0 \subseteq \mathcal{G}^1 \subseteq \mathcal{G}^2 \subseteq ...$
where $\mathcal{G}^0$ is a single point, and $\mathcal{G}^n$ is the subgraph generated by $\mathcal{G}^{n-1}$ and all vertices adjacent to $\mathcal{G}^{n-1}$.
Let $\mathcal{P}^n, \mathcal{Q}^n$ be two circle packing of $\C$ with nerve $\mathcal{G}^n$.
Let $\mathcal{P}^1, \mathcal{Q}^1$ be the subpackings corresponding to the subgraph $\mathcal{G}^1 \subseteq \mathcal{G}^n$.
Then with appropriate normalization, He proved that there exists a homeomorphism 
$\Psi: \C \longrightarrow \C$ that sends the circle packing $\mathcal{P}^1$ to $\mathcal{Q}^1$ with
$\| \Psi - id\|_{C^0} = O(\frac{1}{n})$ (cf.\ exponential convergence in Theorem \ref{thm:LB} and \ref{thm:B}).
This estimate gives a quantitative bound on the convergence of the discrete Riemann mapping in \cite{RS87}.
See also \cite{HR93, HS96, HS98, HL10} for some generalizations, which are closely related to regularity of local symmetries.


\subsubsection{Finite subdivision rules}\label{subsubsec:fsr}
With a motivation from the Cannon's conjecture, finite subdivision rules are introduced and studied extensively in a series of papers by Cannon, Floyd and Parry \cite{CFP01, CFP06a, CFP06b}.
They appear naturally in the study of dynamics of rational maps and geometric group theory (see \cite{CFKP03, CFP07, CFPP09}).
Some connections between finite subdivision rules and circle packings have been explored in \cite{BS97, Ste03, BS17}.

In those previous {works} mentioned above, it is usually assumed that {the} subdivision has bounded valence and it satisfies some expansion property, such as {\em mesh} approaching zero (see \cite{CFP06a, CFP06b}).
This expansion property requires that the edges are cut into smaller and smaller pieces for iterations of the subdivision.
This is in contrast with the setting of this paper.
The subdivision we obtain always have infinite valence.
Moreover, in order to create an infinite graph from the 1-skeleton, the mesh is never shrinking to zero (see Assumption (2) in Definition \ref{defn:fsr}).

\subsubsection{Quasisymmetric rigidity}\label{sss:qr}
A {\em Schottky set} is the complement of a union of disjoint open disks in $\widehat\C$.
This includes examples like circle packings and round Sierpinski carpets.
In \cite{BKM09}, it is proved that a Schottky set $X$ with zero area is quasisymmetrically rigid.
More precisely, any quasisymmetric homeomorphism of $X$ is the restriction of a M\"obius transformation.
Also see \cite{BM13} for other quasisymmetric rigidity results.
We remark that it is crucial that the homeomorphisms are quasisymmetric in this general setting as any two Sierpinski carpets are homeomorphic.


{Quasisymmetries} of {\em relative Schottky sets} are studied in \cite{Mer12, Mer14}.
We remark that a Schottky set has a different definition in \cite{Mer12, Mer14} compared to \cite{BKM09}.
In order to apply He-Schramm's uniformization results for relative circle domains, and with the main application on quasisymmetry groups of a Sierpiski carpet Julia set in mind, it is assumed in \cite{Mer12, Mer14} that the complementary disks of a Schottky set have disjoint closures.
In particular, it does not cover the case that we are considering in this paper.
It is proved that {any} local {quasisymmetry} has complex derivative on the Schottky set (see \cite[Theorem 1.2]{Mer12}), and the derivative is bi-Lipschitz (see \cite[Proposition 8.2]{Mer12}) (cf.\ Theorem \ref{thm:LS}).
It is conjectured that a local {quasisymmetry} possesses higher degree of regularity (see \cite[Conjecture 1.3]{Mer12}).



\subsubsection{Kleinian circle packings}
Let $\Gamma$ be a geometrically finite Kleinian group whose limit set $\Lambda(\Gamma)$ is equal to the limit set $\Lambda(\mathcal{P})$ of an infinite circle packing $\mathcal{P}$.
Then the corresponding three manifold $M = \Gamma\backslash\Hyp^3$ is {\em acylindrical} (cf.\ Theorem~\ref{thm:GFAC1}).
For geometrically finite acylindrical Kleinian group, it follows from \cite{McM90} that there exists a unique Kleinian group with totally geodesic boundary in its quasiconformal conjugacy class.
Such Kleinian groups have Schottky sets as their limit sets, and provide many examples of Kleinian circle packings.

The Apollonian circle packing is an example of a Kleinian circle packing, and its arithmetic, geometric, and dynamical properties have been extensively studied in the literature (for a non-exhaustive list, see e.g. \cite{GLMWY03, GLMWY05, BF11, KO11, OS12, BK14, Zha22}).
More recently, there have been many new and exciting developments in the study of more general Kleinian circle packings \cite{KN19, KK21, BKK22, LLM22}.

\subsubsection{Relative hyperbolic groups and the Cannon conjecture}\label{subsec:ggt}
The {\em Cannon conjecture} \cite{Can91} and the {\em Kapovich-Kleiner conjecture} \cite{KK00} are two important conjectures in geometric group theory.
They assert that if a Gromov hyperbolic group has boundary homeomorphic to $S^2$ or a Sierpinski carpet, then it is virtually Kleinian.
For the Kapovich-Kleiner conjecture, the major difficulty lies in finding a quasi-symmetric embedding of the boundary of the group into $S^2$.
Once this is achieved, the quasi-symmetric rigidity theorem for Sierpinski carpet \cite{Bon11} can be applied and obtain the result.

It is natural to ask the analogous question for circle packing boundaries.
Due to the touchings of the circles, the formulation involves {\em relative hyperbolic groups}. 
This was formulated as a more general conjecture in \cite{HW20}.
Recently, the conjecture was proved for circle packings in \cite{HPW23}.
In line with Bonk's approach, our combinatorial rigidity will give an alternative proof as long as the boundary of the relative hyperbolic group is homeomorphic to a Kleinian circle packing.

Our rigidity result for Kleinian circle packings has more applications for relative hyperbolic groups.
It is a general philosophy in geometric group theory that many properties of the group are determined by its boundary. For example, Paulin proved that the quasi-isometry class of a hyperbolic group is uniquely determined by the topology and quasi-symmetric structure of its Gromov boundary \cite{Pau96}, and Mackay-Sisto generalized this result to relative hyperbolic groups with its Bowditch boundary \cite{MS24}.
For finitely generated relative hyperbolic groups whose boundary is homeomorphic to a circle packing, it follows from our rigidity result and the general Cannon conjecture in this case confirmed in \cite{HPW23} that the topology of the boundary alone uniquely determines the group up to commensurability.

\subsubsection{Renormalization and iterations on Teichm\"uller spaces}
Renormalization was introduced into dynamics in the mid 1970s by Feigenbaum, Coullet and Tresser, and it has established itself as a powerful tool for the study of nonlinear systems whose essential form is repeated at infinitely many scales (for a non-exhaustive list, see e.g. \cite{DH85, McM94, McM96, McM98, Lyu99, IS08, AL11, DLS20, AL22}).
The connections between renormalization theories of quadratic polynomials and Kleinian groups are explained in \cite{McM96}.
Iterations on Teichm\"uller spaces have many applications in dynamics and geometry (see e.g. \cite{Thu88, McM90, DH93, Sel11}).

\subsubsection{Further applications}
It is conjectured that besides some trivial examples, the Julia set of a rational map is not quasiconformally homeomorphic to the limit set of a Kleinian group (see \cite{LLMM23}).
This conjecture is proved in various settings when the Julia set or the limit set is a Sierpinski carpet (see \cite{BLM16, QYZ19}).
In an upcoming sequel \cite{LZ23}, the rigidity results in this paper will be used to prove the conjecture for circle packing limit sets of some reflection groups.

In an upcoming paper of the first author and D. Ntalampekos, the exponential contraction of renormalizations will play an important role in characterizing Julia sets that can be quasiconformally uniformized to a circle packing (c.f. \cite{Bon11, BLM16} for Sierpinski carpet Julia sets, and \cite{McM90} for Sierpinski carpet/circle packing limit sets).

\subsection{Methods and outline}\label{subsec:mandt}
In \S \ref{sec:teich}, we recall some classical results on the Teichm\"uller spaces of finite circle packings, and relate these spaces to classical Teichm\"uller spaces of surfaces.

The {\em skinning map} for circle packings is introduced in \S \ref{sec:sbit}.
Roughly speaking, let $\mathcal{H}$ be a subgraph of a finite graph $\mathcal{G}$.
The natural pullback of the inclusion map $i:\mathcal{H} \longrightarrow\mathcal{G}$ gives a map on the Teichm\"uller space
$$
\tau=\tau_{\mathcal{H}, \mathcal{G}}: \Teich(\mathcal{G}) \longrightarrow \Teich(\mathcal{H}).
$$
We then characterize when the image $\tau(\Teich(\mathcal{G}))$ is bounded in $\Teich(\mathcal{H})$ based on how $\mathcal{H}$ sits in $\mathcal{G}$ (see Theorem \ref{thm:bit} and \ref{thm:bitr}).
The construction and the theorems are motivated by Thurston's Bounded Image Theorem for skinning maps of hyperbolic 3-manifolds (see the discussion in \S \ref{subsec:td}).
In \S \ref{sec:ecp}, we prove the existence part of Theorem \ref{thm:LR} and \ref{thm:A} by taking geometric limit.

{By a theorem of McMullen \cite[Theorem 5.3]{McM90}, the skinning map has derivative $< 1$.
The skinning map we constructed is the restriction on some real slice of the usual skinning map.
Thus it is also a contraction.}
Together with our Bounded Image Theorem {(Theorem \ref{thm:bit} and \ref{thm:bitr})}, the contraction is uniform.
In \S \ref{sec:ism}, we apply iterations of the skinning map to harvest compounding contraction.
This allows us to prove uniqueness, and more generally, the exponential convergence results.
In \S \ref{sec:url}, we apply exponential convergence to prove the regularity and universality results.
With the appropriate setup, the argument is similar to the one in renormalization theory of quadratic polynomials.

Finally, in \S \ref{sec:z2sr}, we illustrate how our methods can be generalized to other subdivision rules.
In particular, we consider the case of interpolations of finite subdivision rules (see Theorem \ref{thm:ILR}) as well as $\Z^2$-subdivision rule (see Theorem \ref{thm:Z2EU}). 
The former allows us to change the combinatorics of circle packings, and lays down the foundations to study parameter spaces of circle packings in the future.
The latter provides a model for rank 2 cusps in a Kleinian group, which is used in \S \ref{sec:kcp} to show geometrically finite Kleinian circle packings are combinatorially rigid.

\subsection*{Future work}
The skinning map mentioned above plays an important role in Thurston's hyperbolization theorem. Y.~Minsky conjectured that the diameter of the image of the skinning map of an acylindrical manifold is bounded above by a constant depending only on the topological complexity of the boundary. In an upcoming paper \cite{LZ24}, we use some of the ideas developed in this paper to prove this conjecture for all acylindrical reflection groups, which include kissing reflection groups discussed in \S\ref{sec:teich}.


\subsection*{Acknowledgement}
The authors thank {Sergiy} Merenkov for explaining to us the literature and his work on rigidity results for circle packings, along with many invaluable suggestions.
The authors also thank Dima Dudko, Misha Lyubich and Curt McMullen for many useful discussions. The circle packings in Figure~\ref{fig:fa} are produced with the program \textbf{GOPack} developed by C.\ Collins, G.\ Orick and K. Stephenson \cite{COK17}. The limit sets in Figures~\ref{fig:pinched_nbhd}, \ref{fig:apollonian_ball} and \ref{fig:z2sreg} are produced with C. McMullen's program \textbf{lim}.

\section{Teichm\"uller space and finite circle packings}\label{sec:teich}
In this section, we recall some facts about finite circle packings and their deformation spaces. Many details can be found in \cite{LLM22}; see also \cite{Bro85, Bro86}.

\subsection{Teichm\"uller space of circle packings}
Let $\mathcal{G}$ be a connected ({finite or infinite}) simple plane graph {with embedding $f:\mathcal{G}\longrightarrow\hat{\mathbb{C}}$. Another plane graph $\mathcal{G}'$ (with embedding $f':\mathcal{G}'\longrightarrow\hat{\mathbb{C}}$) is said to be \emph{isomorphic} to $\mathcal{G}$ if there exists an abstract graph isomorphism $g:\mathcal{G}\to\mathcal{G}'$ and an orientation-preserving homeomorphism $\varphi:\hat{\mathbb{C}}\longrightarrow \hat{\mathbb{C}}$ so that $f'=\varphi\circ f\circ g^{-1}$. For simplicity, we will refer to $g$ as a \emph{plane graph isomorphism} between $\mathcal{G}$ and $\mathcal{G}'$ and suppress the homeomorphism $\varphi$.} 

{\begin{defn}
    A circle packing $\mathcal{P}$ is called \emph{marked} by $\mathcal{G}$ if there exists a plane graph isomorphism $g$ between $\mathcal{G}$ and the nerve of $\mathcal{P}$. For each vertex $v$ of $\mathcal{G}$, we denote by $D_{v,\mathcal{P}}$ the disk in $\mathcal{P}$ corresponding to $g(v)$, and $C_{v,\mathcal{P}}$ its boundary circle.
\end{defn}
Note that the existence of a circle packing marked by $\mathcal{G}$, when {$\mathcal{G}$ is finite}, is guaranteed by the Finite Circle Packing Theorem.}

{We also want to discuss the deformation space of circle packings marked by a plane graph.}
\begin{defn}\label{defn:teichCP}
    The \emph{Teichm\"uller space of circle packings associated to $\mathcal{G}$} is defined as
{\begin{equation*}
\Teich(\mathcal{G}):=\left\{\mathcal{P} \text{ is a finite circle packing marked by $\mathcal{G}$}\middle\}\right/\sim
\end{equation*}}
where $\mathcal{P}'\sim\mathcal{P}$ if there exists a M\"obius transformation $\eta$ so that $D_{v,\mathcal{P}'}=\eta(D_{v,\mathcal{P}})$ for every vertex $v$ of $\mathcal{G}$.
\end{defn}
A natural topology on the space is defined by declaring {$[\mathcal{P}_n]\to[\mathcal{P}]$ if there exist representatives $\mathcal{P}_n,\mathcal{P}$ so that} $D_{v,\mathcal{P}_n}\to D_{v,\mathcal{P}}$ for every vertex $v$.

From now on, generally for simplicity, we drop the dependence on $\mathcal{P}$ in the notations when it is clear which packing we refer to. {We will also use the same notation to denote a circle packing $\mathcal{P}$ and its equivalence class under $\sim$, and it will be clear from context which one we refer to.}

{Given any pair of circle packings $\mathcal{P},\mathcal{P}'\in\Teich(\mathcal{G})$, there exists a homeomorphism $\Psi:\hat{\mathbb{C}}\longrightarrow\hat{\mathbb{C}}$ so that $D_{v,\mathcal{P}'}=\Psi(D_{v,\mathcal{P}})$. In other words, $\mathcal{P},\mathcal{P}'$ are homeomorphic via a marking-preserving homeomorphism of $\hat{\mathbb{C}}$. When $\mathcal{G}$ is finite, such a homeomorphism may be chosen as quasiconformal (in fact, even smooth).}

{The Teichm\"uller space of circle packings carries a natural (extended) metric.
\begin{defn}
    The \emph{Teichm\"uller (extended) metric} on $\Teich(\mathcal{G})$ is given by
    $$d(\mathcal{P},\mathcal{P}')=\frac12\log K,$$
    where $K$ is the minimal dilatation among all marking-preserving quasiconformal homeomorphisms between $\mathcal{P},\mathcal{P}'$. If no such quasiconformal homeomorphism exists, then $d(\mathcal{P},\mathcal{P}'):=\infty$.
\end{defn}
By our remark above, when $\mathcal{G}$ is finite, this gives an actual metric.}

\subsection{Finite circle packings and reflection groups}
From now on, we assume $\mathcal{G}$ is finite.
Given $\mathcal{P}\in\Teich(\mathcal{G})$, let $\Gamma:=\Gamma_{\mathcal{P}}$ be the group generated by reflections in the circles $C_v$. This group is called the \emph{kissing reflection group} associated to $\mathcal{P}$ in \cite{LLM22}.

Let $\Aut^\pm(\widehat{\mathbb{C}})\cong\Isom(\mathbb{H}^3)$ be the group of M\"obius and anti-M\"obius transformations, which is also the group of isometries of $\mathbb{H}^3$. A \emph{quasiconformal deformation} of $\Gamma$ is a discrete and faithful representation $\xi:\Gamma\longrightarrow\Aut^\pm(\widehat{\mathbb{C}})$ induced by a quasiconformal map $f:\widehat{\mathbb{C}}\longrightarrow\widehat{\mathbb{C}}$ (i.e.~$\xi(\gamma)=f\circ\gamma\circ f^{-1}$ for all $\gamma\in\Gamma$). {In particular the representation $\xi$ is \emph{type-preserving}, i.e.~it sends elliptic (resp. parabolic, hyperbolic) elements to elliptic (resp. parabolic, hyperbolic) elements.}

The \emph{quasiconformal deformation space} of $\Gamma$ is defined as
$$\mathcal{QC}(\Gamma):=\{\xi:\Gamma\longrightarrow \Aut^\pm(\widehat{\mathbb{C}})\text{ is a quasiconformal deformation} \}/\sim$$
where $\xi\sim\xi'$ if they are conjugates of each other by a M\"obius transformation. We can put the algebraic topology on the space, i.e. {$[\xi_n]\to[\xi]$ if there exist representatives $\xi_n,\xi$ so that $\xi_n(\gamma)\to\xi(\gamma)$ for all $\gamma\in\Gamma$. As in the case of circle packings, for simplicity, we will use the same notation to denote a representation $\xi$ and its equivalence class under $\sim$; which one we refer to will be clear from context.}

{Given another circle packing $\mathcal{P}'\in\Teich(\mathcal{G})$ and the associated reflection group $\Gamma_{\mathcal{P}'}$, the marking induces a type-preserving group isomorphism between $\Gamma$ and $\Gamma_{\mathcal{P}'}$, although it is not immediately clear that this isomorphism is induced by a quasiconformal map. On the other hand, the discussion in \cite[\S 3]{LLM22} (especially the paragraph before Definition 3.14 and Lemma 3.15) implies the following:
\begin{prop}\label{prop:deformation_space}
For any connected finite simple plane graph $\mathcal{G}$, there exists a continuous injection
$$f:\mathcal{QC}(\Gamma)\longrightarrow \Teich(\mathcal{G})$$
so that for any $\Gamma'\in\mathcal{QC}(\Gamma)$, the group generated by reflections in $f(\Gamma')$ is precisely $\Gamma'$.
\end{prop}
The proof goes roughly as follows. Given any $\xi\in\mathcal{QC}(\Gamma)$, first we observe that $\xi(g)$ is also a circular reflection if $g$ is a reflection in any circles in $\mathcal{P}$. In this way we obtain a new collection $\mathcal{P}'$ of circles whose tangency pattern must be the same as $\mathcal{P}$, with marking provided by the representation $\xi$. See \cite[\S 3.3]{LLM22} for details.}

{To conclude that $\mathcal{QC}(\Gamma)\cong \Teich(\mathcal{G})$, we need to discuss some geometric objects associated to a finite circle packing.}

\subsection{Teichm\"uller space of interstices}
Let $F$ be a face of $\mathcal{G}$. {Note that in general $F$ may not be a Jordan domain (recall that the plane graph $\mathcal{G}$ is embedded in $\hat{\mathbb{C}}$).} Then there exists a polygon $P_{F}$ and a cellular map $\psi_F:P_{F} \longrightarrow F$ which restricts to a homeomorphism between the interiors {(see Figure~\ref{fig:ideal_boundary} for an illustration).}
\begin{figure}[htp]
    \centering
    \includegraphics[width=0.8\textwidth]{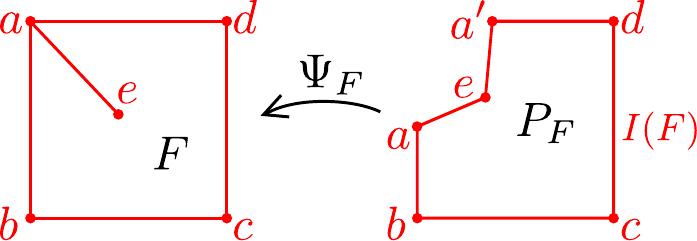}
    \caption{{A non-Jordan face $F$ and its ideal boundary $I(F)$. The map $\Psi_F$ is a homeomorphism between the interiors. Both vertices $a,a'$ on the right is mapped to $a$ on the left.}}\label{fig:ideal_boundary}
\end{figure}
Note that the map $\psi_F$ is a homeomorphism on the entire polygon $P_F$ if and only if $F$ is a Jordan domain.

{Adopting the terminology from \cite{LLM22}, the boundary $\partial P_{F}$ is called the {\em ideal boundary} of $F$, denoted by $I(F)$. Informally, $I(F)$ is the splitting and unfolding of the boundary edges of $F$ to a Jordan curve.} Abusing notations, we often use the same symbols to denote vertices and edges of $P_F$ and those of $F$, especially when $F$ is a Jordan domain.

{
\begin{defn}
Consider the following set
$$\Pi:=\widehat{\mathbb{C}}-\bigcup_v\overline{D_v}.$$
Each face $F$ of $\mathcal{G}$ gives a connected component $\Pi_{F}$ of $\Pi$, which we call the \emph{interstice} of the packing $\mathcal{P}$ for the face $F$.
\end{defn}}
Note that $\Pi_{F}$ is the interior of {a polygon with the same number of edges as $I(F)$, bounded by circular arcs tangent at the vertices. We call such a polygon an \emph{ideal polygon}.}

The group $\Gamma$ contains a subgroup of index $2$ of orientation-preserving elements, which we denote by $\widetilde\Gamma$. Let $\Lambda$ and $\Omega$ be its limit set and domain of discontinuity respectively. Note that $\Pi_{F}$ is contained in a unique {connected} component $\Omega_{F}$ of $\Omega$. Let $\Gamma_{F}$ be the subgroup of $\Gamma$ generated by circles corresponding to vertices of $F$. Then $\Pi_{F}$ is the interior of a fundamental domain of the action of $\Gamma_{F}$ on $\Omega_{F}$.

 It is easy to see that $\widetilde\Gamma$ is \emph{geometrically finite}, i.e.\ its action on the hyperbolic 3-space $\mathbb{H}^3$ has a finite-sided fundamental polyhedron {(see \cite[Corollary~3.16]{LLM22})}.
 Consider the Kleinian manifold
$$M:=\widetilde\Gamma\backslash(\mathbb{H}^3\cup\Omega).$$
Its boundary $\partial M=\bigcup_{F}X_{F}$ has a connected component for each face $F$ of $\mathcal{G}$. Each $X_{F}$ is topologically a punctured sphere, and its number of punctures is the same as the number of vertices on {$I(F)$}. Conformally, $X_{F}\cong \widetilde\Gamma_{F}\backslash \Omega_{F}$, where $\widetilde\Gamma_{F}$ is the index 2 subgroup of $\Gamma_{F}$ consisting of orientation-preserving elements. In fact, it is conformally the double of $\Pi_{F}$, {obtained by gluing two copies of $\Pi_F$ along their boundary edges while leaving vertices as punctures}. The anti-holomorphic reflection that exchanges the two copies of $\Pi_{F}$ is an involution, denoted by $\sigma_F:X_{F}\longrightarrow X_{F}$.

By the quasiconformal deformation theory of Ahlfors, Bers, Maskit and others, including a nicely presented version of Sullivan \cite{Sul81}, there exists a holomorphic surjection
$$\Psi:\prod_{F}\Teich(X_{F})\longrightarrow\mathcal{QC}(\widetilde\Gamma),$$
which is a homeomorphism if each component of $\Omega$ is simply connected. The failure of injectivity comes from a component of $\Omega$ that is not simply connected, e.g.~by performing a Dehn twist along an essential curve in that component which is homotopically trivial in the three manifold $M$. {The map $\Psi$} restricts to a surjection
$$\Psi:\prod_{F}\Teich^{\sigma_F}(X_{F})\longrightarrow\mathcal{QC}(\Gamma),$$
where $\Teich^{\sigma_F}(X_{F})$ denotes the quasiconformal deformation space of $X_{F}$ invariant under $\sigma_F$. Note that since $X_{F}$ is determined by $\Pi_{F}$, we also view $\Teich^{\sigma_F}(X_{F})\cong \Teich(\Pi_{F})$ as a deformation space of the ideal polygon $\Pi_F$.

We claim this restriction is also injective.
{\begin{lem}\label{lem:bij}
    The map
    $$\Psi:\prod_{F}\Teich^{\sigma_F}(X_{F})\longrightarrow\mathcal{QC}(\Gamma)$$
    is a bijection.
\end{lem}}
\begin{proof}
    By the discussion above, it remains to show injectivity. Indeed, non-injectivity is only possible if there exist two different ways of decomposing some $X_{F}$ into two equal pieces of marked ideal polygon, together with an anti-holomorphic involution that exchanges the two pieces.
    
    Conformally, we can put one piece in the upper half plane, and the other piece the lower half plane, with punctures on the extended real line. The only anti-holomorphic involution fixing all the punctures is  $z\mapsto\bar z$, as the invariant set of such an involution is always a circle.

    {Finally, note that the punctures of $X_F$ are marked by edges of $P_F$. This marking can be recovered from any representation $\xi:\Gamma\longrightarrow\Aut^{\pm}(\mathbb{C})$ in $\mathcal{QC}(\Gamma)$, as we recall from the discussion following Proposition~\ref{prop:deformation_space} that $\xi$ recovers the marking on the corresponding circle packing.}
    
    {It follows that there is only one way of decomposing $X_F$, so the map is injective, as desired.}
\end{proof}

Combined with Proposition~\ref{prop:deformation_space}, we have
\begin{theorem}\label{thm:dc}
For any connected finite simple plane graph $\mathcal{G}$,
$$\Teich(\mathcal{G})\cong\prod_{F}\Teich(\Pi_F).$$
\end{theorem}
{\begin{proof}
    First note that we can define a map
    $$h:\Teich(\mathcal{G})\longrightarrow\prod_{F}\Teich(\Pi_F)$$
    by recording the conformal structure on each interstice. Moreover, it is easy to see $\Psi^{-1}=h\circ f$, where $f$ is the map from Proposition~\ref{prop:deformation_space}, and $\Psi$ is the map from Lemma~\ref{lem:bij}. Since $\Psi$ is a bijection and $f$ is an injection, we conclude that $h$ is surjective and injective, and hence a bijection.
\end{proof}}
Moreover, it is easy to see that the identification above is isometric with respect to the Teichm\"uller metric on $\Teich(\mathcal{G})$ and that of the product $\prod_{F}\Teich(\Pi_F)$, defined as the maximum of the Teichm\"uller metrics on $\Teich(\Pi_F)$.

\subsection{Fenchel-Nielsen coordinates}
Given a face $F$ with the cellular map $\psi_F:P_F\longrightarrow F$, any pair of non-adjacent vertices $v,w$ of $I(F)$ gives rise to an essential simple closed curve $\widetilde\gamma^F_{vw}$ of $X_F$, invariant under the involution $\sigma_F$. {See for example the right picture in Figure~\ref{fig:double}.}

More precisely, if $\psi_F(v)\neq\psi_F(w)$, let $g_v$ and $g_w$ be the reflections in the circles $C_v$ and $C_w$ respectively (note that here we use the same symbols to denote the corresponding vertices of $F$). Then the element $g_vg_w\in\widetilde\Gamma_F$ is loxodromic (whose two fixed points are contained in $D_v$ and $D_w$ respectively) and determines the curve $\widetilde\gamma^F_{vw}$ on $X_F$. On the other hand, if $\psi_F(v)=\psi_F(w)$, then this curve can be represented by a nontrivial loop in $\Omega_F$, but is null-homotopic in $M$, so cannot be represented by a nontrivial element in $\widetilde{\Gamma}$ (for an illuminating picture, {see Figure~\ref{fig:double} as well as \cite[Fig.~3.1]{LLM22}}).

Geometrically, note that $v,w$ determine two non-adjacent sides $I_v$ and $I_w$ of the interstice $\Pi_F$. In the hyperbolic metric induced from $X_F$, let $\gamma^F_{vw}$ be the geodesic segment perpendicular to both $I_v$ and $I_w$. The double of $\gamma^F_{vw}$ in $X_F$ is the simple closed geodesic $\widetilde\gamma^F_{vw}$. We remark that every $\sigma_F$-invariant simple closed geodesic on $X_F$ arises this way.

\begin{figure}[htp]
    \centering
    \includegraphics[width=0.45\textwidth]{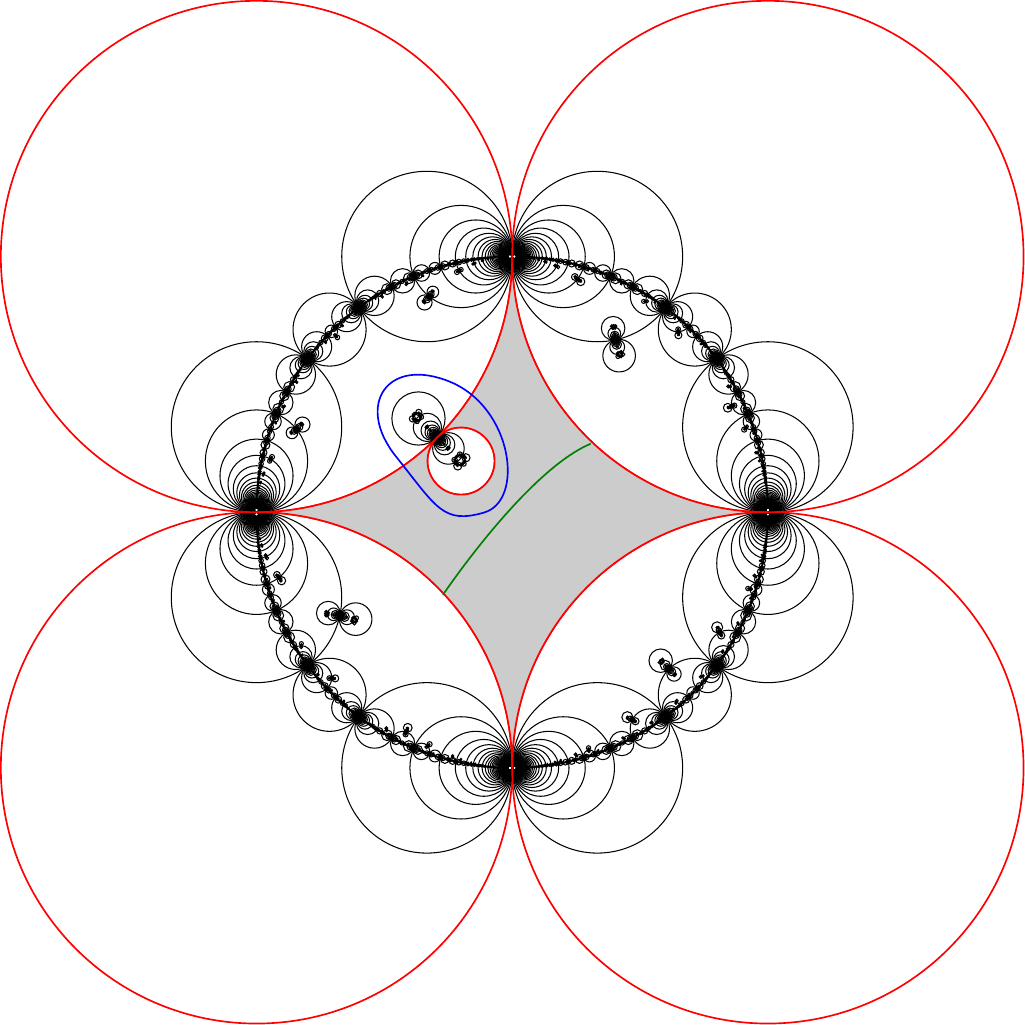}
    \includegraphics[width=0.25\textwidth]{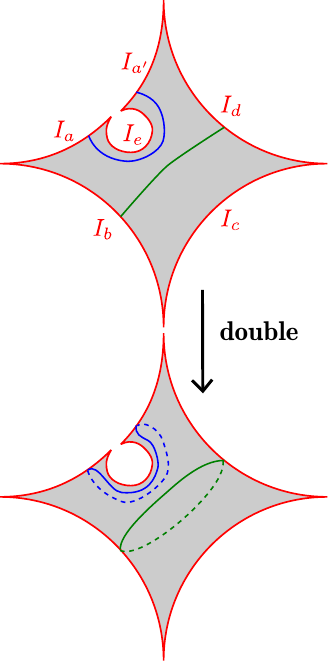}
    \caption{{Interstice for a non-Jordan face with the same combinatorics as Figure~\ref{fig:ideal_boundary}. Note that each vertex gives a side of the interstice and each non-adjacent pair in $I(F)$ gives an invariant simple closed curve on the double $X_F$, obtained by the double of a path connecting the pair of sides. Since $\Psi_F(a)=\Psi_F(a')$, a lift of the corresponding invariant curve is a loop in the domain of discontinuity. On the other hand, since $\Psi_F(b)\neq\Psi_F(d)$, a lift of the corresponding invariant curve is an axis of a loxodromic element.}}\label{fig:double}
\end{figure}

Note that we can decompose $X_F$ into pairs of pants along a maximal collection of disjoint $\sigma_F$-invariant simple closed geodesics. Indeed, any such decomposition comes from a triangulation of $F$, where added edges between two non-adjacent vertices of $F$ determine the collection of simple closed geodesics in the way described above. 

The hyperbolic structure on $X_F$ is completely determined by the hyperbolic lengths of geodesics in such a collection. {Indeed, the fixed set of $\sigma_F$ consists of complete geodesics from cusp to cusp (these come from where we glue two copies of $\Pi_F$), and relative to this collection of transversals, the pants are glued together without twists. This gives Fenchel-Nielsen coordinates without any twist parameters on
$$\Teich(\Pi_F)\cong \Teich^{\sigma_F}(X_{F})\cong \mathbb{R}_+^{m-3}$$
where $m$ is the number of vertices of $F$.}

Finally, we show that to go to infinity in $\Teich(\Pi_F)$, some geodesic arc $\gamma_{vw}^F$ must shrink to zero. Given $\epsilon>0$, define $\Teich_{\ge\epsilon}(\Pi_F)\subseteq\Teich(\Pi_F)$ to be the subset of ideal polygons satisfying the condition that for any non-adjacent pair of vertices $v,w$, the geodesic arc $\gamma_{vw}^F$ has hyperbolic length $\ge\epsilon$.
\begin{prop}\label{prop:fnb}
    For any $\epsilon>0$, $\Teich_{\ge\epsilon}(\Pi_F)$ is compact.
\end{prop}
\begin{proof}
    Note that in $\Teich_{\ge\epsilon}(\Pi_F)$, given a pair of non-adjacent vertices $v,w$, for any point $x$ on $\gamma_{vw}^F$, on one of the two sides, a perpendicular geodesic segment starting at $x$ of length $\epsilon/2$ is contained entirely in $\Pi_F$. Since the hyperbolic area of $\Pi_F$ is fixed, this gives an upper bound on the length of $\gamma_{vw}^F$ as well.
\end{proof}
{We remark that this compact set depends only on $\epsilon$ and the number of vertices of $I(F)$, as $\Teich(\Pi_F)$ may be identified as a slice of the Teichm\"uller space of genus 0 surfaces with the same number of punctures as vertices of $I(F)$.}

\section{The skinning map and bounded image theorem}\label{sec:sbit}
In this section, we will define the skinning map for circle packings, and characterize when a skinning map has bounded image. {The main objective of this section is to prove the Bounded Image Theorem, stated as Theorem~\ref{thm:bit} and Theorem~\ref{thm:bitr} below.}

Let $\mathcal{G}$ be a connected finite simple plane graph and $\mathcal{H} \subseteq \mathcal{G}$ be a connected {\em induced subgraph}, i.e., a connected graph formed by a subset of the vertices of $\mathcal{G}$ and {\em all} of the edges connecting pairs of vertices in the subset.

Then there exists a natural pullback of the inclusion map $i: \mathcal{H} \xhookrightarrow{} \mathcal{G}$ on the Teichm\"uller space of finite circle packings
$$
\tau= \tau_{\mathcal{H}, \mathcal{G}}: \Teich(\mathcal{G}) \longrightarrow \Teich(\mathcal{H}).
$$
Indeed, if $\mathcal{P} \in \Teich(\mathcal{G})$ is a circle packing with nerve $\mathcal{G}$, and $\mathcal{Q}$ be the corresponding sub-circle packing associated to $\mathcal{H}$, then $\tau(\mathcal{P}) = \mathcal{Q}$.
We call this map $\tau$ the {\em skinning map} associated to $\mathcal{H} \subseteq \mathcal{G}$.



We now define an important notion of {\em cylindrical} and {\em acylindrical} subgraphs (cf.\ Definition \ref{defn:cyl} and Figure \ref{fig:subd}).
\begin{defn}
Let $\mathcal{G}$ be a connected finite simple plane graph and $\mathcal{H} \subseteq \mathcal{G}$ be a connected induced subgraph.
Let $F$ be a face of $\mathcal{H}$ that is a Jordan domain.
It is called {\em acylindrical} in $\mathcal{G}$ if for any two non-adjacent vertices $v, w\in \partial F$, the two components of $\partial F - \{v,w\}$ are connected {via a sequence of edges} in $F\cap \mathcal{G}  - \{v,w\}$.
Otherwise, it is called {\em {cylindrical}} in $\mathcal{G}$.

Suppose all faces of $\mathcal{H}$ are Jordan domains.
The subgraph $\mathcal{H}$ is called {\em acylindrical} in $\mathcal{G}$ if every face of $\mathcal{H}$ is acylindrical in $\mathcal{G}$.
Otherwise, it is called {\em cylindrical} in $\mathcal{G}$.
\end{defn}
\begin{figure}[htp]
    \centering
    \includegraphics[width=0.5\textwidth]{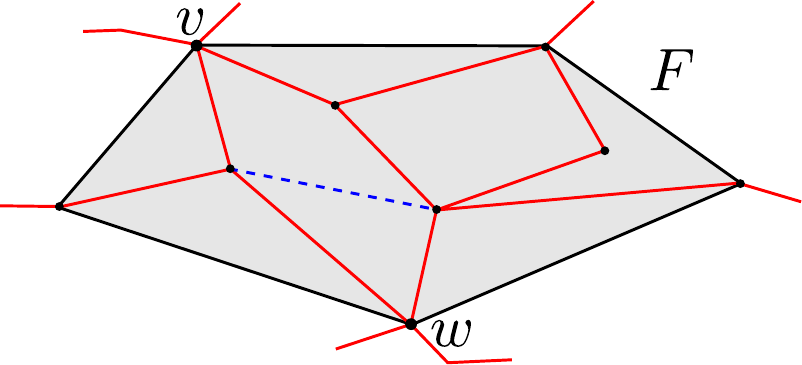}
    \caption{Let $F$ be the shaded face. It is cylindrical without the dashed edge, and acylindrical with it. In the cylindrical case, the two non-adjacent vertices $v, w \in \partial F$ are on the boundary of a face $F'$ of $\mathcal{G}$.}
\end{figure}

{
\begin{rmk}\label{rmk:cylin}
    We remark that it is easy to see that a Jordan domain face $F$ is cylindrical if and only if there are two non-adjacent vertices $v, w\in \partial F$, and a face $F'$ of $\mathcal{G}$ such that $F' \subseteq F$ and $v, w \in \partial F'$.
\end{rmk}
}

The following is the analogue of {\em Thurston's Bounded Image Theorem} in our setting (see \cite{Thu82}).
\begin{theorem}[Bounded Image Theorem]\label{thm:bit}
    Let $\mathcal{G}$ be a connected finite simple plane graph, and $\mathcal{H} \subseteq \mathcal{G}$ be a connected induced subgraph. Suppose that every face of $\mathcal{H}$ is a Jordan domain.
    Then the image $\tau(\Teich(\mathcal{G})) \subseteq \Teich(\mathcal{H})$ is bounded if and only if $\mathcal{H}$ is acylindrical in $\mathcal{G}$.
\end{theorem}

Recall that by Theorem \ref{thm:dc}, we have
$$
\Teich(\mathcal{H}) = \prod_{F \text{ face of }\mathcal{H}} \Teich(\Pi_{F}).
$$
Let $\pi_F: \Teich(\mathcal{H}) \longrightarrow \Teich(\Pi_{F})$ be the corresponding projection.
We actually prove the following stronger version of the Bounded Image Theorem.
\begin{theorem}[Relative Bounded Image Theorem]\label{thm:bitr}
    Let $\mathcal{G}$ be a connected finite simple plane graph, and $\mathcal{H} \subseteq \mathcal{G}$ be a connected induced subgraph. Let $F$ be a face of $\mathcal{H}$ that is a Jordan domain.
    Then the image $\pi_{F}(\tau(\Teich(\mathcal{G}))) \subseteq \Teich(\Pi_{F})$ is bounded if and only if $F$ is acylindrical in $\mathcal{G}$.
\end{theorem}

{The rest of this section is organized as follows. In \S~\ref{subsec:td}, we state the connections between the combinatorial properties of the graph and the topological properties of the corresponding 3 manifold. In \S~\ref{ss:wellconn}, we state and prove the key combinatorial lemma for acylindrical subgraphs. In \S~\ref{ss:bit}, we use the combinatorial lemma to prove Theorem~\ref{thm:bit} and ~\ref{thm:bitr}. Finally in \S~\ref{ss:non-Jordan}, we illustrate the importance of the assumption of Jordan domain in our Bounded Image Theorem.}



\subsection{Topological descriptions}\label{subsec:td}
In this subsection, we give a topological description of acylindricity, and compare the skinning maps for circle packings with Thurston's skinning maps.
The aim is to provide motivation and connect our combinatorial definition to many well-known concepts in hyperbolic geometry and topology.
Our proof of the Bounded Image Theorem (Theorem~\ref{thm:bit} and \ref{thm:bitr}) is self-contained and does not require this connection.

\subsection*{Compressing disks}
Let $\mathcal{P}$ be a circle packing with nerve $\mathcal{H}$. As in \S \ref{sec:teich},
let $\Gamma$ be the group generated by reflections along circles in $\mathcal{P}$, and $\widetilde\Gamma$ be the index 2 subgroup of $\Gamma$.
Let $F$ be a face of $\mathcal{H}$.
Then $F$ corresponds to a boundary component $X_F$ of the Kleinian manifold $M = \widetilde\Gamma\backslash(\mathbb{H}^3\cup\Omega(\widetilde\Gamma))$.

A {\em compressing disk} $D$ is an embedded disk in $M$ with $\partial D \subseteq \partial M$ such that $\partial D$ does not bound a disk in $\partial M$.
We say a component $X_F\subseteq\partial M$ is \emph{incompressible} if there does not exist any compressing disk $D$ with $\partial D\subseteq X_F$.

The following gives a topological characterization for Jordan domain faces.
\begin{theorem}\label{thm:jordan}
    Let $F$ be a face of $\mathcal{H}$.
    Then $F$ is a Jordan domain if and only if $X_F$ is incompressible.
\end{theorem}
\begin{proof}
    It is easy to see that $X_F$ is incompressible if and only if the corresponding domain of discontinuity $\Omega_F$ is simply connected {(see Figure~\ref{fig:double})}. This is equivalent to the fact that the limit set $\Lambda_F$ of the group $\Gamma_F$ is connected.

    The following argument is a special case of \cite[Prop.~3.4]{LLM22}, but we include it for completeness. For more details, see \cite[Prop.~3.4]{LLM22} and the discussion preceding it.

    {For each $v\in\partial F$, let $g_v$ be the reflection in $D_v$. Let $\mathcal{D}^0:=\bigcup_{v\in\partial F}\overline{D_v}$ be the union of the closed disks corresponding to vertices on $\partial F$. And define inductively
    $$\mathcal{D}^{i+1}=\bigcup_{v\in\partial F}g_v\cdot\overline{\mathcal{D}^i-\overline{D_v}}$$
    which is the union of the images under $g_v$ of all closed disks in $\mathcal{D}^i$ outside $D_v$.}

    {Note that $\{\mathcal{D}^i\}$ is nested, and $\Lambda_F=\bigcap_{i=0}^\infty\mathcal{D}^i$. Moreover, if $F$ is a Jordan domain, $D^0$ is the union of a cycle of closed disks, and so is $\mathcal{D}^i$ for all $i\ge0$. In particular, $\mathcal{D}^i$ is connected, and so is $\Lambda_F$.}

    {On the other hand, if $F$ is not a Jordan domain, then there exists a vertex $v$ on $\partial F$ so that deleting it disconnects $\partial F$. This means that disks outside $D_v$ form at least two disjoint chains, and it follows that $\mathcal{D}^1$ is disjoint, and so is $\Lambda_F$ {(see Figure~\ref{fig:double})}.}
\end{proof}

\subsection*{Cylinders}
Let $F$ be a face of $\mathcal{H} \subseteq \mathcal{G}$ that is a Jordan domain.
Let $\mathcal{P}_F$ be a finite circle packing with nerve $F \cap \mathcal{G}$.
Denote the corresponding reflection group, Kleinian group and Kleinian manifold by {$\Gamma_{F\cap\mathcal{G}}$, $\widetilde\Gamma_{F\cap\mathcal{G}}$ and $M_{F\cap\mathcal{G}}$}.

Let $F_{\ext}$ be the external face of $F \cap \mathcal{G}$, i.e., $F_{\ext}=\overline{\widehat\C - F}$.
Let {$X_{\ext} \subseteq \partial M_{F\cap\mathcal{G}}$} be the corresponding surface.

A {\em cylinder} is a continuous map 
$$
f:(S^1\times[0,1],S^1\times\{0,1\})\longrightarrow(M_{F\cap\mathcal{G}},\partial M_{F\cap\mathcal{G}}).
$$
It is called {\em boundary parallel} if it can be homotoped rel boundary to a cylinder in {$\partial M_{F\cap\mathcal{G}}$}.
A cylinder is \emph{essential} if its boundary components $f(S^1\times\{0\}),  f(S^1\times\{1\})$ are essential curves of $\partial M_{F\cap\mathcal{G}}$, and it is not boundary parallel.

The following theorem provides an equivalent topological definition of acylindrical faces in $\mathcal{G}$.
\begin{theorem}\label{thm:acy}
    Let $F$ be a face of $\mathcal{H} \subseteq \mathcal{G}$ that is a Jordan domain with at least 4 vertices on its boundary.
    Then $F$ is acylindrical in $\mathcal{G}$ if and only if 
    {there exists no essential cylinder in $M_{F\cap\mathcal{G}}$ with one boundary component in $X_{\ext}$.}
\end{theorem}
\begin{proof}
    {We prove the contrapositive of both directions.}
    Suppose first $F$ is cylindrical in $\mathcal{G}$. Then there exists two non-adjacent vertices $v,w$ on $\partial F$, and a face $F'$ of $\mathcal{G}$ contained in $F$ so that $v,w\in \partial F'$ {(see Remark~\ref{rmk:cylin})}. Let $g_v$ and $g_w$ be reflections in $C_v$ and $C_w$ respectively. Then $g_vg_w$ is in both $\widetilde\Gamma_{F'}$ and $\widetilde\Gamma_{F_{\ext}}$. In other words, the curves $\tilde\gamma_{vw}^{F'}$ and $\tilde\gamma_{vw}^{F_{\ext}}$ are homotopic in $M_{F\cap\mathcal{G}}$. The homotopy between them gives an essential cylinder.

    Conversely, suppose $f:(S^1\times[0,1],S^1\times\{0,1\})\longrightarrow(M_{F\cap\mathcal{G}},\partial M_{F\cap\mathcal{G}})$ is an essential cylinder with $f(S^1\times\{0\})\subset X_{\ext}$ and $f(S^1\times\{1\})\subset X_{F'}$ for some face $F'$ of $\mathcal{G}$. Then there exists a lift of $f$
    $$\tilde f:\mathbb{R}\times[0,1]\longrightarrow \mathbb{H}^3\cup\Omega(\Gamma_{F\cap \mathcal{G}})$$
    so that $\tilde f(\mathbb{R}\times\{0\})\subset\Omega_{\ext}$ and $\tilde f(\mathbb{R}\times\{1\})\subset g\cdot\Omega_{F'}$ for some $g\in\Gamma_{F\cap\mathcal{G}}$, {and $\Omega_{\ext}, \Omega_{F'}$ are components of $\Omega(\Gamma_{F\cap \mathcal{G}})$ corresponding to $F_{\ext}, F'$ respectively}. {Moreover, since $F$ is a Jordan domain, by choosing a different lift if necessary} we may assume the two endpoints of $\tilde f(\mathbb{R}\times\{0\})$ lie in two different disks $D_v$ and $D_w$ for a pair of distinct vertices $v,w\in\partial F$.

    {We first show that $g\cdot \Omega_{F'}$ must meet both $D_v$ and $D_w$.}
    {Set $\mathcal{S}:=f(S^1\times[0,1])$ and $\tilde{\mathcal{S}}:=\tilde f(\mathbb{R}\times[0,1])$ its lift.
    Let $g_0$ be the generator of the deck group of the covering $\tilde{\mathcal{S}}\longrightarrow\mathcal{S}$. This is a loxodromic element representing the core curve of the cylinder $S$. A fundamental domain $\mathcal{F} \subseteq \tilde{\mathcal{S}}$ for the action of $g_0$ on $\tilde{\mathcal{S}}$ is a strip intersecting $\Omega_{\ext}$ and $g\cdot\Omega_{F'}$ in compact segments. Thus $\mathcal{F}$ is itself a compact subset of $\mathbb{H}^3\cup\Omega(\Gamma_{F\cap \mathcal{G}})$. Note that the action of $g_0\in\pi_1(M_{F\cap\mathcal{G}})$ is properly discontinuous on $\mathbb{H}^3\cup\Omega(\Gamma_{F\cap \mathcal{G}})$, so $g_0^n\mathcal{F}$ limits to the attracting fixed point of $g_0$ as $n\to\infty$, and to the repelling fixed point as $n\to-\infty$. It follows that the endpoints of both $\tilde f(\mathbb{R}\times\{0\})$ and $\tilde f(\mathbb{R}\times\{1\})$ are precisely the fixed points of $g_0$. We conclude that $g\cdot\Omega_{F'}$ meets both $D_v$ and $D_w$.}

    {We now show that $g$ stabilizes $\Omega_{F'}$. 
    Suppose not, then $g\cdot\Omega_{F'}$ is contained entirely in one of the disks $D_u$ for some vertex $u$ of $F\cap \mathcal{G}$ (see \cite[Lemma 3.1]{LLM22}). But this is not possible.}

    {We now show that the fixed points of $g_0$ are contained in $\bigcup_{u\in \partial F'}\overline{D}_{u}$.}
    {From the discussions in \cite[\S3.1]{LLM22}, we know that the limit set of the stabilizer of $\Omega_F'$ in $\Gamma_{F\cap\mathcal{G}}$ is contained in the union of the closed disks corresponding to vertices on $\partial F'$. In particular, the fixed points of $g_0$ must be contained in the union of $\overline{D}_{u}$ as $u$ ranges over all vertices on $\partial F'$.}
    Hence $v,w\in \partial F'$.
    
    {If $v$ and $w$ are nonadjacent in $\partial F$, then we can conclude that $F$ is cylindrical (see Remark~\ref{rmk:cylin}). Otherwise, note that since the two circles $C_v$ and $C_w$ are tangent, $g_wg_v$ is a parabolic element fixing the tangent point between the two circles.}
    {Denote the fixed points of $g_0$ by $a, b$ with $a\in D_v, b \in D_w$. Now consider a different lift $g_v\tilde f$ of $f$. Since $b \notin D_v$, $g_vb \in D_v$ and $g_va \in D_{s_1}$ for some $s_1 \in \partial F_{\ext}$.}

    {If $s_1\neq w$, then replacing $\tilde f$ with $g_v\tilde f$ in the argument above gives $s_1\in\partial F'$ as well. Among $s_1,v,w$, two of them must be non-adjacent. Then we can again conclude that $F$ is cylindrical.}
    
    {If $s_1=w$ then we in turn consider the lift $g_wg_v\tilde f$ of $f$. Note $g_wg_v\tilde f (\mathbb{R}\times\{0\})$ has one endpoint $g_wg_va \in D_{s_2}$ for some vertex $s_2 \in \partial F_{\ext}$ and the other one $g_wg_vb \in D_{w}$. As above, either $s_2\neq v$ (and we are done), or we consider $g_vg_wg_v\tilde f$. Inductively, after finite steps $k$, we must have $s_k \notin \{v, w\}$, for otherwise the endpoints $a, b$ of $\tilde f(\mathbb{R}\times\{0\})$ are the fixed points of $g_wg_v$. But this is not possible as $g_wg_v$ is parabolic.}
\end{proof}


\subsection*{Skinning maps}
{In the following, we briefly discuss the construction of the skinning map for $3$-manifold, and refer the readers to \cite[\S 3]{McM90} and \cite[\S 1.4]{Ken10} for more details.}
Let $\Gamma$ be a geometrically finite Kleinian group {without accidental parabolics}, and $M$ the corresponding Kleinian manifold.
Suppose $\partial M$ is incompressible. Then the quasiconformal deformation space of $\Gamma$ can be identified with $\Teich(\partial M)$, {which parametrizes the conformal structure on $\partial M$}.
The cover of $\Int(M)$ corresponding to $\partial M$ is a finite union of quasi-Fuchsian manifolds, {and homeomorphic to $\partial M\times\mathbb{R}$}. {Each quasi-Fuchsian manifold has two ends, one of which is shared with $M$.} 
We then obtain a map
$$
\Teich(\partial M)\cong\mathcal{QC}(\Gamma) \longrightarrow \Teich(\partial M) \times \Teich(\overline{\partial M}).
$$
The first coordinate is the identity map ({corresponding to shared ends}), and the second coordinate is Thurston's skinning map 
$$
\tau_M: \Teich(\partial M) \longrightarrow \Teich(\overline{\partial M}),
$$
{which corresponds to the other end of the quasi-Fuchsian manifolds not shared with $M$}.

To see the connection with our setting, let us consider the following situation.
Let $\mathcal{G}$ be a finite simple plane graph. 
Suppose that the faces $F_1,..., F_k$ of $\mathcal{G}$ are all Jordan domains.
Let $\mathcal{H}_i = \partial F_i$.
{Note that $\mathcal{H}_i$ has two faces, one shared with $\mathcal{G}$, i.e. $F_i$, and the other one denoted by $F_i^{\ext}$.}

Let $\mathcal{P}$ be a circle packing with nerve $\mathcal{G}$, and $\Gamma$ be the group generated by reflection along circles in $\mathcal{P}$.
Let $M$ be the Kleinian manifold for the index 2 subgroup of $\Gamma$. {The reflections induce an orientation-reversing isometry $\sigma: M\longrightarrow M$.}
We denote by $\Gamma_i$ and $M_i$ the subgroup and Kleinian manifold associated to $\mathcal{H}_i$.
Note that each $M_i$ is a quasi-Fuchsian manifold, {with one end coming from the shared face $F_i$, and the other end coming from $F_i^{\ext}$.}

By Theorem \ref{thm:dc}, we have
$$
\Teich(\mathcal{G}) = \Teich^\sigma(\partial M) \text{ and } \Teich(\mathcal{H}_i) = \Teich^\sigma(\partial M_i).
$$
Note that 
$$
\Teich(\partial M) \times \Teich(\overline{\partial M}) = \prod_i \Teich(\partial M_i).
$$
Putting everything together, we have the following theorem giving a precise connection of Thurston's skinning map and the skinning maps in our setting.
\begin{theorem}\label{thm:csk}
    With the notations above, Thurston's skinning map
    $$
    (id, \tau_M): \Teich(\partial M) \longrightarrow \Teich(\partial M) \times \Teich(\overline{\partial M}) = \prod_i \Teich(\partial M_i)
    $$
    restricts to a map
    $$
    (id, \tau_M): \Teich^\sigma(\partial M) = \Teich(\mathcal{G})\longrightarrow  \Teich^\sigma(\partial M) \times \Teich^\sigma(\overline{\partial M}) = \prod_i \Teich(\mathcal{H}_i).
    $$
    Moreover, we have the identity $(id, \tau_M)|_{\Teich^\sigma(\partial M)} = \prod_i \tau_{\mathcal{H}_i, \mathcal{G}}$
\end{theorem}
\begin{proof}
    For each face $F_i$ of $\mathcal{G}$, the stabilizer of the corresponding domain of discontinuity $\Omega_i$ is precisely $\Gamma_i$. So the restriction of the Thurston's skinning map to the reflection locus has image in the reflection locus. 
    
    {For each quasi-Fuchsian manifold $M_i$, recall that one of its ends comes from face $F_i$ shared with $\mathcal{G}$, and the other end comes from $F_i^{\ext}$. This end is precisely the image of the skinning map. In particular, the map $\tau_{\mathcal{H}_i,\mathcal{G}}$ precisely gives $(id,\tau_M)$ composed with projection onto the corresponding factor.}
\end{proof}


\subsection{Well-connectedness}\label{ss:wellconn}
In this subsection, we prove the following proposition which says acylindrical subgraphs contains a lot of connecting paths.
\begin{prop}\label{prop:cp}
    Let $F$ be a face of $\mathcal{H} \subseteq \mathcal{G}$ that is a Jordan domain.
    Suppose that $F$ is acylindrical in $\mathcal{G}$.
    Then for any pair of non-adjacent vertices $v, w\in \partial F$, there exists a simple path $\gamma \in F \cap \mathcal{G}$ connecting $v, w$ so that $\Int(\gamma) \subseteq \Int(F)$.    
\end{prop}
To prove this, we first state the following separation lemma.
\begin{lem}\label{lem:sep}
    Let $F$ be a face of $\mathcal{H} \subseteq \mathcal{G}$ that is a Jordan domain.
    Suppose that $F$ is acylindrical in $\mathcal{G}$.
    Then for any pair of non-adjacent vertices $v, w \in \partial F$, there exists a path $\alpha \in F\cap \mathcal{G}$ with $\Int(\alpha) \subseteq \Int(F)$ that separates $v, w$ in $F$.
    More precisely, $v, w$ are in two different components of $F - \alpha$.
\end{lem}
\begin{proof}
    Since $F$ is acylindrical in $\mathcal{G}$, there exists a path $\alpha \in F\cap \mathcal{G}$ that connects the two components of $\partial F - \{v, w\}$.
    This path $\alpha$ separates $v, w$.
\end{proof}
\begin{proof}[Proof of Proposition \ref{prop:cp}]
    First note that since we can construct a simple path from a path by deleting closed loops, it suffices to construct a path $\gamma \in F \cap \mathcal{G}$ connecting $v, w$ so that $\Int(\gamma) \subseteq \Int(F)$.

    To construct this path, we first pick the two {neighboring} vertices $x_1, y_1 \in \partial F$ of $v$, and let $I_1$ be the component of $\partial F - \{x_1, y_1\}$ that contains $w$.
    By Lemma \ref{lem:sep}, there exists a path $\gamma_1$ connecting $v$ to $v_1$ where $v_1 \in I_1$ with $\Int(\gamma_1) \subseteq \Int(F)$.

    Let $I_2$ be the component of $\partial F - \{v, v_1\}$ that contains $w$.
    Clearly, the two boundary points in $\partial I_2$ are non-adjacent.
    Now applying Lemma \ref{lem:sep} to the two points in $\partial I_2$ and obtain a path $\gamma_2$ connecting a point $v_2 \in I_2$ to $\partial F - \overline{I_2}$.
    Since the end points of $\gamma_1$ and $\gamma_2$ are linked, $\Int(\gamma_1) \cup \Int(\gamma_2)$ is connected.

    Inductively, suppose $w \notin \{v, v_1, v_2, ..., v_{k-1}\}$.
    Then let $I_k$ be the component of $\partial F - \{v, v_1, v_2, ..., v_{k-1}\}$ that contains $w$.
    We can apply Lemma \ref{lem:sep} to $\partial I_k$ and obtain a path $\gamma_k$ connecting $v_k \in I_k$ to $\partial F - \overline{I_k}$ with $\Int(\gamma_k) \subseteq \Int(F)$.
    By construction, it is easy to see that $\Int(\gamma_1) \cup ... \cup \Int(\gamma_k)$ is connected (see Figure~\ref{fig:path}).

    \begin{figure}[htp]
    \centering
    \includegraphics[width=0.5\textwidth]{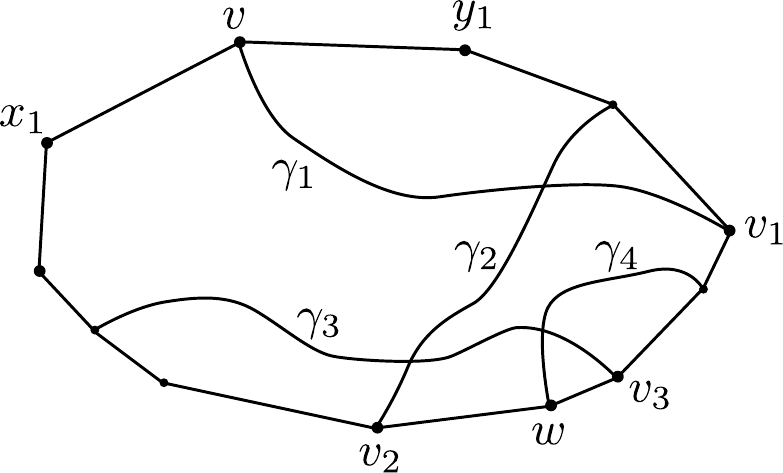}
    \caption{An illustration of the the construction of the path connecting $v, w$ in the proof of Lemma~\ref{lem:sep}.}
    \label{fig:path}
\end{figure}
    
    Since the number of vertices in $I_k$ by at least $1$ less than the number of vertices in $I_{k-1}$, this process eventually terminates, say at $k$.
    Thus, $w = v_k$ is an end point of $\gamma_k$.
    Since $\Int(\gamma_1) \cup ... \cup \Int(\gamma_k)$ is connected, we can find a path connecting $v, w$ as desired.
\end{proof}

\subsection{Bounded image theorem}\label{ss:bit}
In this subsection, we prove Theorem \ref{thm:bitr}, from which Theorem \ref{thm:bit} follows immediately.
We remark that it is possible to give a similar argument as in Thurston's original proof of the Bounded Image Theorem \cite{Thu82}.
\begin{proof}[Proof of Theorem \ref{thm:bitr}]
    Suppose that $F$ is cylindrical.
    Then there exist non-adjacent vertices $v, w \in \partial F$, and a face $F'$ of $\mathcal{G}$ contained in $F$ so that $v, w \in \partial F'$ (see Remark~\ref{rmk:cylin}).
    Let $\Gamma_{\mathcal{G}}$ be the reflection groups associated to a circle packing $\mathcal{P}_{\mathcal{G}}$ with nerve $\mathcal{G}$.
    Similarly, let $\Gamma_{\mathcal{H}} \subseteq \Gamma_{\mathcal{G}}$ be the subgroup associated to the induced subgraph $\mathcal{H}$.

    As in \S\ref{sec:teich}, let $\Omega_{F'}$ (and $\Omega_{F}$) be the component of domain of discontinuity for $\Gamma_{\mathcal{G}}$ (and $\Gamma_{\mathcal{H}}$) associated to $F'$ (and $F$ respectively).
    Note that $\Omega_{F'} \subseteq \Omega_{F}$.
    We also denote by $\Gamma_F \subseteq \Gamma_{\mathcal{H}}$, $\Gamma_{F'}\subseteq \Gamma_{\mathcal{G}}$ the groups generated by circles corresponding to vertices of $\partial F$, $\partial F'$ respectively, and $\widetilde\Gamma_F$, $\widetilde\Gamma_{F'}$ their index 2 subgroups of orientation preserving elements.
    Let $g_{v}, g_{w} \in \Gamma_{\mathcal{H}} \subseteq \Gamma_{\mathcal{G}}$ be reflections along circles $C_v, C_w$ respectively.
    Then $g_v, g_w \in \Gamma_F \cap \Gamma_{F'}$.
    Since $v, w$ are non-adjacent in $\partial F$ and $\mathcal{H}$ is an induced subgraph of $\mathcal{G}$, they are non-adjacent in the ideal boundary $I(F')$ of $F'$. 
    Thus, the product $g_vg_w$ corresponds to $\sigma$-invariant simple closed curves $\gamma$ and $\gamma'$ in $X_{F}:= \widetilde\Gamma_F\backslash\Omega_F$ and $X_{F'}:= \widetilde\Gamma_{F'}\backslash\Omega_{F'}$.
    Since $\Omega_{F'} \subseteq \Omega_{F}$, by Schwarz lemma, we have
    $$
    l_{X_F}(\gamma) \leq l_{X_{F'}}(\gamma').
    $$
    By quasiconformally deforming the surface $X_{F'}$, we can find a sequence in $\Teich(\mathcal{G})$ so that the length of $\gamma'$ is shrinking to $0$.
    By the above inequality, the length of $\gamma$ is also shrinking to $0$.
    Thus, the image of this sequence under 
    $$
    \pi_F\circ \tau: \Teich(\mathcal{G}) \longrightarrow \Teich(\Pi_F)
    $$ 
    is not bounded.
    This proves one direction.

    \begin{figure}[htp]
    \centering
    \begin{subfigure}[b]{0.45\textwidth}
        \centering
        \includegraphics[width=\textwidth]{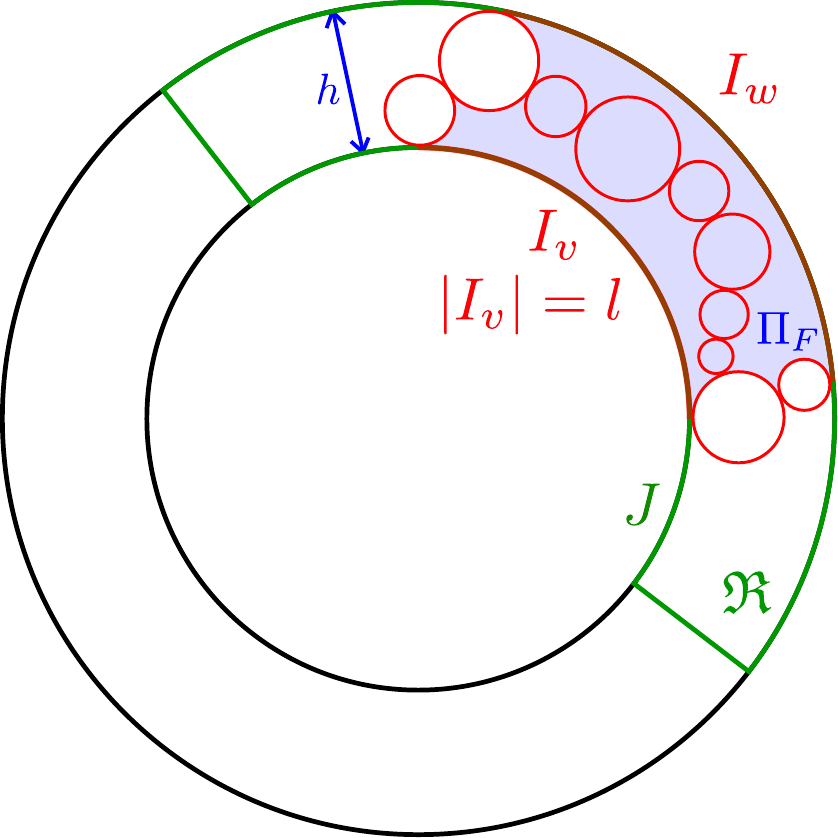}
        \caption{}\label{fig:bda}
    \end{subfigure}
    \begin{subfigure}[b]{0.45\textwidth}
        \centering
        \includegraphics[width=\textwidth]{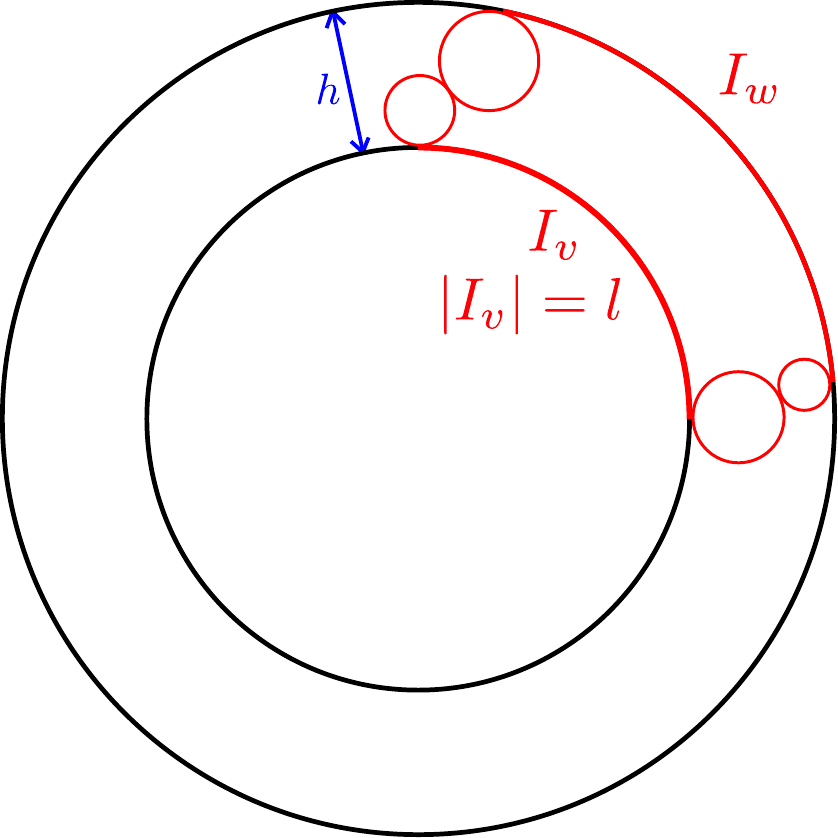}
        \caption{}\label{fig:bdb}
    \end{subfigure}
    \caption{An acylindrical face on the left, where the chain of touching circles in $\Pi_F$ gives an upper bound on $\frac{l}{h}$. A cylindrical face on the right, where $\frac{l}{h}$ can be arbitrarily large by squeezing the red circles.}
    \label{fig:bd}
\end{figure}

    Conversely, suppose that $F$ is acylindrical.
    Let $\mathcal{P}_\mathcal{G}$ be a circle packing with nerve $\mathcal{G}$ and $\mathcal{P}_\mathcal{H}$ be the sub-circle packing associated to $\mathcal{H}$.
    
    Let $\Pi_F$ be the interstice of the circle packing $\mathcal{P}_\mathcal{H}$ for the face $F$, i.e., the interior of the fundamental domain associated to $F$.
    Note that $\Pi_F$ is bounded by circular arcs from circles of $\mathcal{P}_\mathcal{H}$ associated to $\partial F$.
    Let $v, w$ be two non-adjacent vertices on $\partial F$, and let $I_v, I_w$ be the two corresponding circular arcs on $\partial \Pi_F$.
    Then we obtain a {conformal} quadrilateral $(\overline{\Pi_F}, I_v\cup I_w)$ {(i.e, a Jordan domain with two arcs on its boundary, see\cite[p.~52]{A73}).} It is conformally equivalent to 
    $$
    ([0,t]\times [0, 1], [0,t]\times \{0, 1\}).
    $$
    Here $t = \Mod((\overline{\Pi_F}, I_v\cup I_w))$ is the conformal modulus of $(\overline{\Pi_F}, I_v\cup I_w)$.

    {We claim that $\Mod((\overline{\Pi_F}, I_v\cup I_w))$ has a uniform upper bound.}
    Note that $v, w$ are not adjacent in $\partial F$, but they might be adjacent in $\mathcal{H} \subseteq \mathcal{G}$.
    
    Suppose $v, w$ are not adjacent in $\mathcal{H}$.
    Let us normalize by a M\"obius transformation so that $C_v$ and $C_w$ are circles centered at $0$ with radius $1$ and $R>1$ (see Figure~\ref{fig:bda}).
    Since $F$ is acylindrical, there exists a chain of circles in the closed annulus $\overline{A(R)}=\overline{B(0, R)} - B(0,1)$ that bounds the arc $I_v$.
    {Let $N$ be the number of vertices of $\mathcal{G}$ in $F$. Then the number of circles in this chain is bounded by $N$.}
    Let $l$ be the length of the arc $I_v$ and $h = R-1$.
    {Since each circle is contained in the annulus, its diameter $\leq h$. Since $l$ is less than the sum of the diameters of the circles this chain, we have $l \leq N h$.
    Let $J$ be the arc on $\partial B(0,1)$ by extending $I_v$ on both sides by $Nh$. If $l + 2Nh \geq 2\pi$, then we take $J = \partial B(0,1)$. Let $\alpha(t) = te^{i\theta}, \alpha'(t) = te^{i\theta'}, t\in [1,R]$ be the radial line segments connecting $\partial J$ to $\partial B(0, R)$. Let $\mathfrak{R}$ be the circular rectangle bounded by $J, \alpha, \alpha'$ and the corresponding arc $J' \subseteq \partial B(0, R)$. If $J = \partial B(0, 1)$, we set $\mathfrak{R} = A(R)$ (see Figure~\ref{fig:bda}). 
    Since the number of vertices of $\partial F$ is bounded by $N$, $\Pi_F\subseteq \mathfrak{R}$. By the comparison principle (see \cite[\S 4.3]{A73} or \cite[\S 4]{KL}), $\Mod((\overline{\Pi_F}, I_v\cup I_w)) \leq \Mod((\mathfrak{R}, J \cup J'))$. Note the logarithm map takes $\mathfrak{R}$ to an Euclidean rectangle. Thus, $\Mod(\mathfrak{R}, J, J') = \frac{l+2Nh}{\log (1+h)}$ if $l+2Nh \leq 2\pi$ and $\frac{2\pi}{\log(1+h)}$ if $l+2Nh \geq 2\pi$.}
    {In the former case, $h\le\pi/N$, and so
    $$
    \Mod((\overline{\Pi_F}, I_v\cup I_w)) \leq \Mod(\mathfrak{R}, J, J') \leq \frac{3Nh}{\log(1+h)} \leq \frac{3\pi}{\log(1+\pi/N)}.
    $$
    In the latter case, $h\ge2\pi/(3N)$, and so
    $$
    \Mod((\overline{\Pi_F}, I_v\cup I_w)) \leq \Mod(\mathfrak{R}, J, J')\le\frac{2\pi}{\log(1+\frac{2\pi}{3N})}.
    $$}

    Suppose $v, w$ are adjacent in $\mathcal{H}$.
    Then we normalize by a M\"obius transformation so that $C_v$ and $C_w$ are $\{z:\Im(z) = 0\}$ and $\{z:\Im(z) = 1\}$. Then a similar argument shows that 
    $$
    \Mod((\overline{\Pi_F}, I_v\cup I_w)) \leq 3N,
    $$
    where $N$ is again the number of vertices of $\mathcal{G}$ in $F$.

    Let $\gamma_{vw}$ be the simple closed curve on $X_F$ associated to the non-adjacent pair $(v,w)$.
    {Note that the double of $\overline{\Pi_F}$ along $I_v\cup I_w$ gives a tubular neighborhood of $\gamma_{vw}$. Since $\Mod((\overline{\Pi_F}, I_v\cup I_w))$ is bounded, the modulus of any tubular neighborhood of $\gamma_{vw}$ has an upper bound.} Thus, there exists a lower bound for the hyperbolic length $l_{X_F}(\gamma_{vw})$.
    Since this holds for any pair of non-adjacent vertices on $\partial F$, the image $\pi_F(\tau(\Teich(\mathcal{G})))$ is bounded in $\Teich(\Pi_F)$ by Proposition \ref{prop:fnb}.
\end{proof}

\subsection{A counterexample with non-Jordan domains}\label{ss:non-Jordan}
We end our discussion with a counterexample showing that the Jordan domain assumption in the Bounded Image Theorem is essential.

Consider the example as Figure \ref{fig:ce}.
The graph $\mathcal{H}$ is in black, and the graph $\mathcal{G}$ is the union of black and red edges.
The bounded face $F$ of $\mathcal{H}$ is not a Jordan domain.
Let $\Psi_F:P_F \longrightarrow F$ be the cellular map which restricts to a homeomorphism in the interior. 
Note that $P_F$ is acylindrical to the pullback of the graph $\mathcal{G}$.
A circle packing with nerve $\mathcal{G}$ is depicted in the bottom of Figure \ref{fig:ce}.
It is easy to see that we can shrink the small black circle to a point while fixing the other {four} black circles in $\Teich(\mathcal{G})$.
Therefore, $\pi_F\circ \tau(\Teich(\mathcal{G}))$ is not bounded in $\Teich(\Pi_F)$.

\begin{figure}[htp]
    \centering
    \begin{subfigure}[b]{0.45\textwidth}
        \centering
        \includegraphics[width=\textwidth]{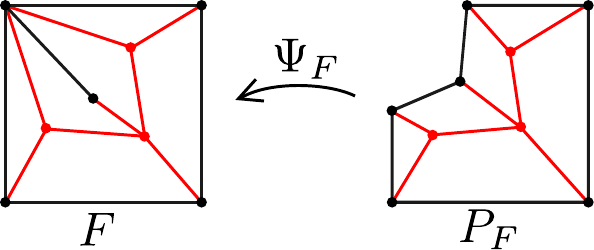}
    \end{subfigure}\hspace{0.02\textwidth}
    \begin{subfigure}[b]{0.35\textwidth}
        \centering
        \includegraphics[width=\textwidth]{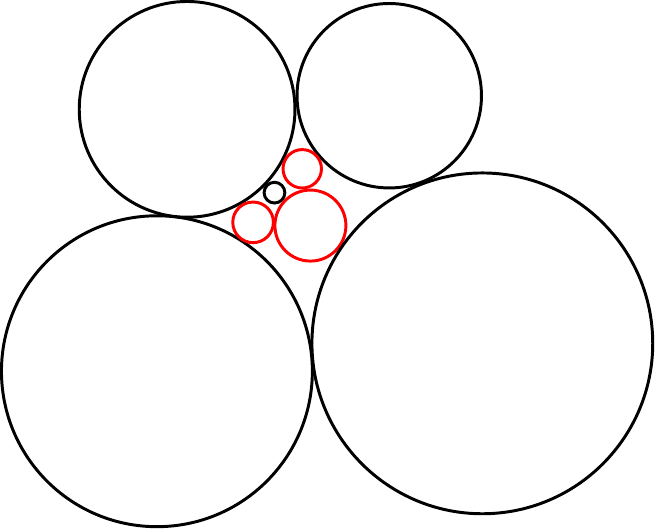}
    \end{subfigure}
    \caption{A counterexample with non-Jordan domains. {This example is a subdivision of Figure~\ref{fig:ideal_boundary}; the induced subdivision on $P_F$ is acylindrical. It is possible to shrink the inner black circle while still able to add three red circles. This corresponds to shrinking the blue curve in Figure~\ref{fig:double}.}}
    \label{fig:ce}
\end{figure}

\section{Existence of circle packings}\label{sec:ecp}
{In this section, we will apply our Bounded Image Theorem (Theorem~\ref{thm:bit}) to prove the existence part of Theorem \ref{thm:LR} and \ref{thm:A}. 
}

{
We first recall that the definition of finite subdivision rule is introduced in \S~\ref{ss:gfsr} (Definition~\ref{defn:fsr}). The main theorem of this section is the following.
}
\begin{theorem}\label{thm:exc}
    Let $\mathcal{R}$ be an irreducible simple finite subdivision rule, with subdivision graphs $\mathcal{G}_i, i=1,..., k$.
    Let $\mathcal{G}$ be a simple spherical subdivision graph for $\mathcal{R}$.
    \begin{itemize}
        \item The subdivision graphs $\mathcal{G}_i$ are isomorphic to the nerves of infinite circle packings $\mathcal{P}_i$ if and only if $\mathcal{R}$ is acylindrical.
        \item The spherical subdivision graph $\mathcal{G}$ is isomorphic to the nerve of an infinite circle packing $\mathcal{P}$ if and only if $\mathcal{R}$ is acylindrical.
    \end{itemize}
\end{theorem}

{The section is organized as follows. To apply Theorem~\ref{thm:bit}, we construct a modification of the subdivision rule in \S~\ref{subsec:sjd}, so that each face is an induced Jordan domain. In \S~\ref{ss:cylindricalsub}, we give combinatorial characterization of cylindrical subdivision rules. Finally, we prove Theorem~\ref{thm:exc} in \S~\ref{ss:existence}.}

\subsection{Subdivision with induced Jordan domain faces}\label{subsec:sjd}
To apply our Bounded Image Theorem (Theorem~\ref{thm:bit}) to $\mathcal{H} = \partial F$ where $F$ is a face of $\mathcal{G}^n_i$, we first introduce the following notion.
\begin{defn}
    Let $F$ be a face of a finite plane graph $\mathcal{G} \subseteq \widehat\C$. We say $F$ is an {\em induced Jordan domain} if $\overline{F}$ is homeomorphic to a closed disk and $\partial F$ is an induced subgraph of $\mathcal{G}$.
\end{defn}
In the following, we construct a modification of $\mathcal{R}$ so that all the faces of the subdivision are induced Jordan domains.

Let $\mathcal{R}=\{P_i\}_{i=1}^k$ be an irreducible simple finite subdivision rule.
Suppose furthermore that $\mathcal{R}$ is acylindrical.

{
Recall that $\mathcal{R}^n(P_i)$ is a $2$-complex homeomorphic to $P_i$ and $\mathcal{G}^n_i$ is its 1-skeleton, which is identified as a graph in $\widehat\C$.
There is a unique face of $\mathcal{G}^n_i$, called external, which is not subdivided. This face corresponds to the complement of $P_i$. The other faces are called non-external, and correspond to faces of $\mathcal{R}^n(P_i)$. We emphasize that faces of $\mathcal{G}^n_i$ include both external and non-external ones, while faces of $\mathcal{R}^n(P_i)$ consist of only non-external ones.}

Since $\mathcal{R}$ is irreducible, $\mathcal{G}^n_i$ is an induced subgraph of $\mathcal{G}^{n+1}_i$.
Since $\mathcal{R}$ is acylindrical, by replacing $\mathcal{R}$ with an iterate if necessary, for the remainder of this section, we assume that 
\begin{itemize}
    \item For each $i = 1,..., k$, the face $P_i$ of $\mathcal{G}^0_i$ is acylindrical in $\mathcal{G}^1_i$.
\end{itemize}
\begin{lem}
    For each $i$, there exist graphs $\widetilde{\mathcal{G}}^1_i$ with
    $$
    \mathcal{G}^1_i \subseteq \widetilde{\mathcal{G}}^1_i \subseteq \mathcal{G}^2_i
    $$
    so that every face of $\widetilde{\mathcal{G}}^1_i$ is an induced Jordan domain.
\end{lem}
\begin{proof}
    Let $F$ be a non-external face of $\mathcal{G}^1_i$.
    Suppose that $F$ is not a Jordan domain.
    Let $\psi_F:P_F \longrightarrow F$ be the cellular map given by the subdivision rule, where $P_F \in \{P_1,..., P_k\}$ {is the type of the face.}
    We denote the pullback of $\mathcal{G}^2_i$ by $\psi_F^*(\mathcal{G}^2_i)$.
    More precisely,
    $$
    \psi_F^*(\mathcal{G}^2_i):= \psi_F^{-1}(F \cap \mathcal{G}^2_i).
    $$
    Note that $P_F$ is acylindrical in $\psi_F^*(\mathcal{G}^2_i)$.
    Therefore, by Proposition \ref{prop:cp}, for any pair of non-adjacent vertices $v, w\in \partial P_F$, there exists a simple path $\gamma$ in $\psi_F^*(\mathcal{G}^2_i)$ with $\Int(\gamma) \subseteq \Int(P_F)$ connecting $v, w$.
    By adding finitely many such paths $\gamma_1,..., \gamma_s$, we can make sure that every face $F' \subseteq F$ of $\partial F \cup \psi_F(\bigcup \gamma_j)$ is an induced Jordan domain.
    
    Let $\widetilde{\mathcal{G}}^1_i$ be the graph {obtained} by adding all such paths in every non-Jordan domain face of $\mathcal{G}^1_i$.
    This graph $\widetilde{\mathcal{G}}^1_i$ satisfies the requirement.
\end{proof}

\subsubsection{Construction of the new subdivision rule $\widetilde{\mathcal{R}}$}
Now we can construct a new finite subdivision rule $\widetilde{\mathcal{R}}$ (see Figure~\ref{fig:modificationR} for an illustration).
It contains polygons $\{P_1,..., P_k, Q_1,..., Q_l\}$, where $P_i$ is the polygon in the original finite subdivision rule $\mathcal{R}$, and $Q_1,..., Q_l$ are non-external faces of $\widetilde{\mathcal{G}}^1_i$ that are not faces of $\mathcal{G}^1_i$ where $i =1,.., k$.

For each polygon $P_i$, the subdivision $\widetilde{\mathcal{R}}(P_i)$ is defined so that the 1-skeleton of $\widetilde{\mathcal{R}}(P_i)$ is $\widetilde{\mathcal{G}}^1_i$.
By construction, each face of $\widetilde{\mathcal{R}}(P_i)$ can be identified with a polygon in $\{P_1,..., P_k, Q_1,..., Q_l\}$, and each edge of $\partial P_i$ contains no vertices of $\widetilde{\mathcal{R}}(P_i)$.

Now consider a polygon $Q_j$.
By definition, we can identify $Q_j$ as a face of $\widetilde{\mathcal{G}}^1_i$ for some $i$.
Since each face of $\mathcal{R}(P_i)$ can be identified with a polygon in $\{P_1,.., P_k\}$, we can consider the subdivision $\widetilde{\mathcal{R}}(\mathcal{R}(P_i))$.
Then we have
$$
P_i \preceq \mathcal{R}(P_i) \preceq \widetilde{\mathcal{R}}(P_i) \preceq \mathcal{R}^2(P_i) \preceq \widetilde{\mathcal{R}}(\mathcal{R}(P_i)) \preceq \mathcal{R}^3(P_i)
$$
where $A\preceq B$ means that $B$ is a subdivision of $A$.
Equivalently, denote the 1-skeleton of $\widetilde{\mathcal{R}}(\mathcal{R}(P_i))$ by $\widetilde{\mathcal{G}}^2_i$. Then we have
$$
\mathcal{G}^1_i \subseteq \widetilde{\mathcal{G}}^1_i \subseteq \mathcal{G}^2_i \subseteq \widetilde{\mathcal{G}}^2_i \subseteq \mathcal{G}^3_i.
$$
Since $Q_j$ is a face of $\widetilde{\mathcal{G}}^1_i$, we define the subdivision $\widetilde{\mathcal{R}}(Q_j)$ so that the 1-skeleton of $\widetilde{\mathcal{R}}(Q_j)$ is $\widetilde{\mathcal{G}}^2_i \cap Q_j$.
By construction, each face of $\widetilde{\mathcal{R}}(Q_j)$ can be identified with a polygon in $\{P_1,..., P_k, Q_1,..., Q_l\}$, and each edge of $\partial Q_j$ contains no vertices of $\widetilde{\mathcal{R}}(Q_j)$.

\begin{figure}[htp]
    \centering
    \includegraphics[width=0.33\textwidth]{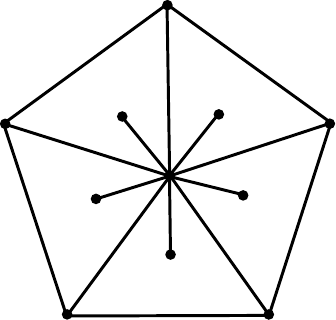}
    \includegraphics[width=0.33\textwidth]{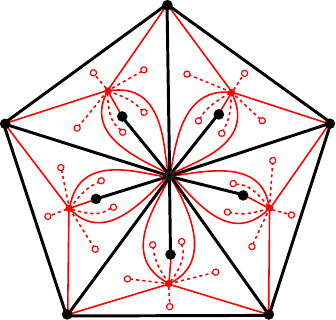}
    \includegraphics[width=0.3\textwidth]{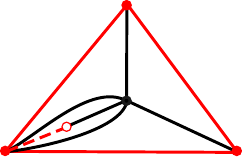}
    \caption{An illustration of the construction of the modification $\widetilde{\mathcal{R}}$. The original subdivision rule is on the left figure, where each face is non-Jordan, and is identified with the origin pentagon.
    The graph $\widetilde{\mathcal{G}}^1$ is the union of the solid red and black edges in the middle figure. The graph $\mathcal{G}^2$ is the union of $\widetilde{\mathcal{G}}^1$ and the dotted red edges.
    Each non-external face of $\widetilde{\mathcal{G}}^1$ is a triangle which is an induced Jordan domain. The subdivision of this triangle $Q$ is illustrated in the right figure.}
    \label{fig:modificationR}
    \end{figure}

In this way, we obtain a new finite subdivision rule $\widetilde{\mathcal{R}}$.
\begin{prop}\label{prop:jd}
    Let $\mathcal{R}=\{P_i\}_{i=1}^k$ be an irreducible simple finite subdivision rule.
    Suppose that $\mathcal{R}$ is acylindrical.
    Then by replacing $\mathcal{R}$ with an iterate if necessary, there exists an irreducible simple acylindrical finite subdivision rule $\widetilde{\mathcal{R}}$ consisting of polygons 
    $$
    \{P_1,..., P_k, Q_1,..., Q_l\}
    $$ 
    with the following properties.
    \begin{enumerate}
        \item The subdivisions satisfy
        $$
        \mathcal{R}^n(P_i) \preceq \widetilde{\mathcal{R}}^n(P_i) \preceq \mathcal{R}^{n+1}(P_i);
        $$
        or equivalently, the 1-skeletons satisfy
        $$
        \mathcal{G}^n_i \subseteq \widetilde{\mathcal{G}}^n_i \subseteq \mathcal{G}^{n+1}_i
        $$
        for all $i=1,.., k$ and $n \geq 1$.
        In particular, $\mathcal{G}_i = \widetilde{\mathcal{G}}_i$ for all $i=1,.., k$.
        \item Every face of $\widetilde{\mathcal{R}}^n(P_i)$ or $\widetilde{\mathcal{R}}^n(Q_j)$ is an induced Jordan domain for all $i=1,..., k, j=1,..., l$ and $n \geq 1$.
    \end{enumerate}
\end{prop}
\begin{proof}
    It is clear that $\widetilde{\mathcal{R}}$ constructed above is simple, and acylindrical.
    By subdividing the polygons in $\{P_1,..., P_k, Q_1,..., Q_l\}$, we can also make sure that $\widetilde{\mathcal{R}}$ is irreducible.

    To prove $(1)$, we note that by construction, we have $\widetilde{\mathcal{R}}^2(P_i) = \widetilde{\mathcal{R}}(\mathcal{R}(P_i))$ for all $i$.
    Since each face of $\mathcal{R}^{j}(P_i)$ can be identified with one of the polygons in $\{P_1,.., P_k\}$,
    $$
    \widetilde{\mathcal{R}}(\mathcal{R}^{n-1}(P_i)) = \widetilde{\mathcal{R}}(\mathcal{R}((\mathcal{R}^{n-2}(P_i))) = \widetilde{\mathcal{R}}^2(\mathcal{R}^{n-2}(P_i)) = ... =\widetilde{\mathcal{R}}^n(P_i).
    $$
    Let $F$ be a face of $\mathcal{R}^{n-1}(P_i)$.
    Then $\mathcal{R}(F) \preceq \widetilde{\mathcal{R}}(F) \preceq \mathcal{R}^2(F)$.
    Therefore, we have $\mathcal{R}^n(P_i) \preceq \widetilde{\mathcal{R}}^n(P_i) = \widetilde{\mathcal{R}}(\mathcal{R}^{n-1}(P_i)) \preceq \mathcal{R}^{n+1}(P_i)$.

    To prove $(2)$, it suffices to show that every face of $\widetilde{\mathcal{R}}(P_i)$ and $\widetilde{\mathcal{R}}(Q_j)$ is an induced Jordan domain.
    The statement is clear from our construction for $\widetilde{\mathcal{R}}(P_i)$.

    To prove the statement for $\widetilde{\mathcal{R}}(Q_j)$, we first claim that every face of $\widetilde{\mathcal{R}}^2(P_i)$ is an induced Jordan domain.
    \begin{proof}[Proof of the claim]
        Let $F$ be a face of $\widetilde{\mathcal{R}}^2(P_i) = \widetilde{\mathcal{R}}(\mathcal{R}(P_i))$.
        Let $F'$ be a face of $\mathcal{R}(P_i)$ so that $F \subseteq F'$.
        Let $\psi_{F'}:P_{F'} \longrightarrow F'$ be the cellular map.
        Then by the construction of $\widetilde{\mathcal{R}}$, $\psi_{F'}^{-1}(F)$ is an induced Jordan domain.
        Since for each $i$, $P_i$ is acylindrical in $\mathcal{G}^1_i$, and $\mathcal{R}(F')\preceq \widetilde{\mathcal{R}}(F')$, we have that $\partial P_{F'}$ is acylindrical in $\psi_{F'}^*(\widetilde{\mathcal{G}}^2_i)$.
        Thus, {the boundary $\partial \psi_{F'}^{-1}(F)$ does not contain two non-adjacent vertices of $\partial P_{F'}$(see Remark~\ref{rmk:cylin}).}
        Therefore
        \begin{itemize}
            \item either the face $\psi_{F'}^{-1}(F)$ is disjoint from $\partial P_{F'}$; or
            \item $\psi_{F'}^{-1}(F) \cap \partial P_{F'}$ is a single vertex;
            or
            \item $\psi_{F'}^{-1}(F) \cap \partial P_{F'}$ is a single edge.
        \end{itemize}
        {Since $\psi_{F'}$ is a homeomorphism in the interior, and a cellular map, we conclude that $F$ is also an induced Jordan domain in all three cases.}
    \end{proof}
     Since we can identify the polygon $Q_j$ with a face of $\widetilde{\mathcal{R}}(P_i)$ and $\widetilde{\mathcal{R}}(Q_j) = \widetilde{\mathcal{R}}^2(P_i) \cap Q_j$ by construction, the statement for $\widetilde{\mathcal{R}}(Q_j)$ follows immediately from the claim.
\end{proof}

A similar argument gives the following for spherical subdivision graphs.
{
\begin{prop}\label{prop:jdg}
    Let $\mathcal{G}$ be a simple spherical subdivision graph with irreducible finite subdivision rule $\mathcal{R}$.
    Suppose that $\mathcal{R}$ is acylindrical.
    Then there exists a modification of the subdivision rule $\widetilde{\mathcal{R}}$ so that
    $$
    \mathcal{G} = \lim_{\rightarrow} \widetilde{\mathcal{G}}^n \bigcup \widetilde{\mathcal{G}}^n,
    $$
    and for each $n \geq 0$, each face of $\widetilde{\mathcal{G}}^n$ is an induced Jordan domain.
\end{prop}
}

\subsection{Cylindrical subdivisions}\label{ss:cylindricalsub}
In this subsection, we shall prove the following characterization of cylindrical subdivisions.
\begin{prop}\label{prop:cs}
    Let $\mathcal{R}=\{P_i\}_{i=1}^k$ be a simple irreducible finite subdivision rule.
    Suppose that $\mathcal{R}$ is cylindrical.
    Then there exist {$i = 1,,..., k$,} a finite number $K$ and a pair of non-adjacent vertices $v, w \in \partial P_i$ so that there are infinitely many paths of length $\leq K$ in $\mathcal{G}_i$ connecting $v, w$ with pairwise disjoint interiors.
\end{prop}
\begin{proof}
    {We first claim that there exist $i = 1,..., k$ and a pair of vertices $v, w \in \partial P_i$ and an infinite sequence paths $\gamma_n$ in $\mathcal{G}_i$ connecting $v, w$ with $K = \max_n \text{length}(\gamma_n) < \infty$ and $\gamma_n \neq \gamma_m$ for all $n\neq m$.}
    Indeed, since $\mathcal{R}$ is cylindrical, there exist $i \in \{1,\dots k\}$ and a pair of non-adjacent $v, w \in \partial P_i$ so that the two components of $\partial P_i -\{v,w\}$ are in different components of $\mathcal{G}^n_i-\{v, w\}$ for all $n$ {(see Definition~\ref{defn:cyl} and the remark in the preceding paragraph)}.
    Therefore, there exists a sequence of non-external face $F_n, n\in\N$ of $\mathcal{G}^n_i$ so that $v,w \in \partial F_n$ (see Remark~\ref{rmk:cylin}).
    {By passing to a subsequence}, we may assume the faces are nested, i.e., $F_{n+1} \subseteq F_n$.
    Since $\mathcal{R}(P_i)$ consists of at least two faces, $F_{n+1}$ is a proper subset of $F_n$.
    Let $\gamma_n \subseteq \partial F_n$ be a path connecting $v, w$. {Note that there are two such choices.}
    Since $F_n$ are all different, we can choose $\gamma_n$ to be all different.
    {Since each non-external face of $\mathcal{G}^n_i$ is identified with one of the polygons in $\{P_1,..., P_k\}$, the lengths of $\gamma_n$ are all bounded, so $K = \max_n \text{length}(\gamma_n) < \infty$.} 
    
    {We now claim that either 
    \begin{enumerate}
        \item there are infinitely many $\gamma_n$ with pairwise disjoint interior, or
        \item there are two non-adjacent vertices $v', w' \in \partial P_j$ for some $j = 1,..., k$ with infinitely many distinct paths in $\mathcal{G}_j$ of length $\leq K-1$.
    \end{enumerate} 
    The proposition follows directly from the first case of the claim. In the second case, we inductively apply the claim. Since the lengths of paths are positive integers, eventually we will obtain the first case. Therefore, the proposition follows from the claim.
    }

    {
    To prove the claim, we first consider the case that for each $\gamma_l$, there are only finitely many $\gamma_n$ with $\Int(\gamma_l) \cap \Int(\gamma_n) \neq \emptyset$. Then we can inductively select $\gamma_{n_j}$ so that $\gamma_{n_j}$ has disjoint interior with $\gamma_{n_i}$ for all $i \leq j-1$. Thus we are in the first case of the claim. 
    }

    {Now consider the case that there exists some $\gamma_l$ so that there are infinitely many $\gamma_n$ with
    $\Int(\gamma_l) \cap \Int(\gamma_n) \neq \emptyset$. 
    After passing to a subsequence if necessary, we may assume $\Int(\gamma_l) \cap \Int(\gamma_n) \neq \emptyset$ for all $n\geq l$.
    Thus, there exists $x \in \Int(\gamma_l)$ so that $x \in \gamma_n$ for all $n\geq l$.
    Let $\gamma_n'$ (and $\gamma_n''$) be the sub-path of $\gamma_n$ that connects $v, x$ (and $x, w$ respectively).
    Since $\gamma_n, n \geq l$ are distinct, either $\gamma_n'$ are distinct or $\gamma_n''$ are distinct. 
    Without loss of generality and after passing to a subsequence if necessary, we assume $\gamma'_n \neq \gamma'_m$ for all $n > m \geq l$. 
    Thus, $v, x$ are not adjacent in the ideal boundary of $F_{l+1}$.
    Note that the length of $\gamma'_n$ is at most $K-1$.
    Let $P_j$ be the type of the face $F_{l+1}$. Then $\gamma'_n, n \geq l+1$ induces the required list of paths for the second case of the claim. The proposition follows.}
\end{proof}
{
\begin{rmk}
    We remark that it can be proved directly that the converse of Proposition~\ref{prop:cs} holds. The converse also follows from the existence of circle packings $\mathcal{P}$ for acylindrical subdivision rule (Theorem~\ref{thm:exc}), since for any two disjoint circles in $\mathcal{P}$, there are only finitely many chains of touching circles of bounded length connecting them.
\end{rmk}}

\subsection{Existence of circle packings}\label{ss:existence}
\begin{proof}[Proof of Theorem \ref{thm:exc}]
    Let us prove the case for subdivision graphs.
    The statement for spherical subdivision graph can be proved in a similar way.
    
    Suppose that $\mathcal{R}$ is acylindrical.
    Consider the sequence $\mathcal{G}^n_i$.
    By Proposition \ref{prop:jd}, we may assume that all faces of $\mathcal{G}^n_i$ are induced Jordan domains. 
    Let $F^{\ext}_i$ be the external face.
    By Theorem \ref{thm:dc}, we have
    $$
    \Teich(\mathcal{G}^n_i) = \prod_{F \text{ face of }\mathcal{G}^n_i} \Teich(\Pi_{F}) = \prod_{F \text{ non-external}} \Teich(\Pi_{F}) \times \Teich(\Pi_{F^{\ext}_i}).
    $$
    Fix a conformal structure $\sigma_{\ext} \in \Teich(\Pi_{F^{\ext}_i})$. {Denote the fiber by
    $$
    \Teich_{\sigma_{\ext}}(\mathcal{G}_i^n) = \prod_{F  \text{ non-external}} \Teich(\Pi_{F}) \times \{\sigma_{\ext}\}.
    $$}

    {Consider a sequence of finite circle packings $(\mathcal{P}^n_i: n \in \N)$, where $\mathcal{P}^n_i \in \Teich_{\sigma_{\ext}}(\mathcal{G}^n_i)$.
    Fix an integer $k \in \N$.}
    Let $\tau_{k,n}: \Teich(\mathcal{G}^{n}_i) \longrightarrow \Teich(\mathcal{G}^k_i)$ be the skinning map, where $n> k$.
    {Note that for each $n \geq k$, $\tau_{k,n}(\mathcal{P}^{n}_i)$ is a finite circle packing with nerve $\mathcal{G}^k_i$.}
    By construction, we have
    \begin{equation}\label{eqn:ext}
        \pi_{F^{\ext}_i}(\tau_{k,n}(\mathcal{P}^{n}_i)) = \sigma_{\ext}.
    \end{equation}
    Let $F$ be a non-external face of $\mathcal{G}^k_i$.
    Then $F$ is acylindrical in $\mathcal{G}^{n}_i$ for all $n>k$.
    Therefore, by Theorem \ref{thm:bitr}, 
    \begin{equation}\label{eqn:nonext}
    \pi_{F}(\tau_{k,n}(\Teich(\mathcal{G}^n_i))) = \pi_{F}(\tau_{k,k+1}(\tau_{k+1,n}(\Teich(\mathcal{G}^n_i))))\subseteq \pi_{F}(\tau_{k,k+1}(\Teich(\mathcal{G}^{k+1}_i)))
    \end{equation}
    is uniformly bounded {in $\Teich(\Pi_{F})$} for all $n>k$.
    {Combining the equations \eqref{eqn:ext} and \eqref{eqn:nonext}, we conclude that the infinite sequence of finite circle packings $(\mathcal{P}^n_i: n \in \N_{\geq k})$ are mapped to a bounded set, i.e., 
    \begin{equation}\label{eqn:boundedlevelk}
        \{\tau_{k,n}(\mathcal{P}^{n}_i): n\geq k \} \text{ is bounded in }\Teich(\mathcal{G}^k_i).
    \end{equation}}

    {We normalize the circle packings so that three marked tangent points between level $0$ circles are $0, 1, \infty$. Fix $k \in \N$.
    By \eqref{eqn:boundedlevelk} and Theorem~\ref{thm:dc}, $\Lambda(\tau_{k,n}(\mathcal{P}^n_i))$ is uniformly quasiconformally homeomorphic to $\Lambda(\mathcal{P}^k_i)$ for all $n \geq k$. Thus, by compactness of $K$-quasiconformal maps, after passing to a subsequence, $\Lambda(\mathcal{P}^n_i)$ converges in Gromov-Hausdorff topology to a closed set 
    $$
    \Lambda^k := \lim_{n\to\infty}\Lambda(\tau_{k,n}(\mathcal{P}^n_i)),
    $$  
    which is again a finite circle packing with nerve $\mathcal{G}^k_i$.
    By a diagonal argument, after passing to a further subsequence if necessary, we assume that the limit $\Lambda^k = \lim_{n\to\infty}\Lambda(\tau_{k,n}(\mathcal{P}^n_i))$ exists for all $k \in \N$ and is the limit set of a finite circle packing with nerve $\mathcal{G}_i^k$.
    Note that $\Lambda^k \subseteq \Lambda^l$ if $l\geq k$.}
    
    {
    Let $\Lambda^\infty = \overline{\bigcup_k \Lambda^k}$.
    We will now show that $\Lambda^\infty$ is the limit set of an infinite circle packing with nerve $\mathcal{G}_i$.
    We claim that each complementary component of $\Lambda^\infty$, other than the one that corresponds to the external face, is a disk of the finite circle packing for $\Lambda^k$ for some $k$.
    Let $(F^n)$ be a nested sequence of non-external faces of $\mathcal{G}^n_i$, and let $\Pi(F^n)$ be the corresponding component of $\widehat\C - \Lambda^n$. Note that each complementary component of $\Lambda^\infty$ corresponds to either a disk of the finite circle packing, or the external face, or $\bigcap \Pi(F^{n})$ for some nested sequence of faces.
    Since $\partial F^n$ is acylindrical in $\mathcal{G}^{n+1}_i$, $\partial F^n \cap \partial F^{n+1}$ is either empty, a vertex, or an edge together with the two boundary vertices. In all three cases, since the number of vertices in $\partial F^n$ is uniformly bounded, the diameter of $\Pi(F^n)$ must go to $0$. Thus, $\bigcap \overline{\Pi(F^{n})}$ is a singleton. This proves the claim. By construction, two complementary disks components touch if and only if the corresponding vertices are adjacent in $\mathcal{G}_i$. Therefore, $\Lambda^\infty$ is the limit set of an infinite circle packing with nerve $\mathcal{G}_i$.}

    Conversely, suppose that $\mathcal{R}$ is cylindrical.
    Suppose for contradiction that $\mathcal{P}_i$ is a circle packing with nerve $\mathcal{G}_i$.
    By Proposition \ref{prop:cs}, there exists a pair of non-adjacent vertices $v, w\in \partial P_i$ with infinitely many paths in $\mathcal{G}_i$ of uniformly bounded length connecting $v, w$ with pairwise disjoint interiors.
    {Let $C_v, C_w$ be the corresponding circles in $\mathcal{P}_i$. Since $v, w$ are non-adjacent and $\mathcal{R}$ is irreducible, $C_v, C_w$ are disjoint. Normalize by a M\"obius map, we assume that $C_v = \partial B(0,1), C_w = \partial B(0, R), R> 1$. By Proposition \ref{prop:cs}, there are infinitely many chains of touching disks connecting $C_v$ and $C_w$ in $A(R) = B(0, R) - \overline{B(0, 1)}$ of bounded length $\leq N$ with pairwise disjoint interior. Since each chain connects $C_v$ and $C_w$, the maximum diameters of disks in this chain must be $\geq (R-1)/N$. Thus the area of the union of disks in each chain has a lower bound. Since the disks have pairwise disjoint interior, the total area is infinite. But the area of the annulus $A(R)$ is a finite number.  This is a contradiction, and the theorem follows.}    
\end{proof}

Let $\Teich(\mathcal{G}_i) := \{\mathcal{P}_i: \mathcal{P}_i \text{ has nerve } \mathcal{G}_i\}/\PSL_2(\C)$ be the Teichm\"uller space with nerve $\mathcal{G}_i$.
Let $F^{\ext}_{i}$ be the external face of $\mathcal{G}_i$.
Then there exists a map
$$
\Phi_i: \Teich(\mathcal{G}_i) \longrightarrow \Teich(\Pi_{F^{\ext}_{i}}).
$$
Note that { in the proof of Theorem \ref{thm:exc}, we show that for each fixed conformal class $\sigma_{\ext} \in \Teich(\Pi_{F^{\ext}_{i}})$, we can construct an infinite circle packing $\Lambda^\infty$ with external class $\sigma_{\ext}$. Therefore, we have the following}
\begin{prop}\label{prop:surj}
    Suppose that $\mathcal{R}$ is acylindrical. Then $\Phi_i$ is surjective.
\end{prop}

We will see in \S \ref{subsec:vec} that this map is an isometric bijection with respect to the Teichm\"uller metrics.

\section{Iterations of the skinning map}\label{sec:ism}
{In this section, we will prove exponential convergence of finite approximations (Theorem~\ref{thm:LB} and \ref{thm:B}) using iterations of the skinning map introduced in \S~\ref{sec:sbit}. This allows us to describe the moduli space of infinite circle packings and complete the proof of Theorem~\ref{thm:LR} and \ref{thm:A}, as well as to prove the geometric inflexibility of circle packings (Theorem \ref{thm:EC}).}

{
The section is organized as follows. In \S~\ref{subsec:su}, we set up the notations to apply the iteration. The key objective is to construct the diameter constant $C$ and the contraction constant $\delta < 1$. In \S~\ref{ss:fixedexternalclass}, we prove Theorem~\ref{thm:LB},  \ref{thm:B} and \ref{thm:A}. In \S~\ref{subsec:vec}, we prove Theorem~\ref{thm:LR}. The proof of Theorem~\ref{thm:EC} is contained in \S~\ref{ss:renom}. Finally, we introduce a nice class of quasiconformal homeomorphism between two infinite circle packings in \S~\ref{subsec:tm}.
}

\subsection{Diameter and contraction constant}\label{subsec:su}
Let $\mathcal{R}=\{P_i\}_{i=1}^k$ be a simple irreducible acylindrical finite subdivision rule.
By Proposition \ref{prop:jd} and replacing $\mathcal{R}$ with an appropriate modification if necessary, we may assume that
\begin{enumerate}
    \item every face $F$ of $\mathcal{G}_i^k$ is an induced Jordan domain; and
    \item every non-external face $F$ of $\mathcal{G}_i^k$ is acylindrical in $\mathcal{G}_i^n$ for all $n > k$.
\end{enumerate}

{In the following, we will construct two constants
\begin{itemize}
    \item $C$: a constant that bounds the diameter of the skinning image; and 
    \item $\delta < 1$: the contraction constant for the derivative of the skinning map.
\end{itemize}
We summarize the construction as follows (see \S~\ref{sss:maps} and \S~\ref{sss:compactset} for more details).  
We first associate a map $\tau_F$ for each face of $\mathcal{G}^1_i$. Note that $F$ can be either external or non-external. 
If $F$ is non-external, we use Theorem \ref{thm:bitr} to construct a compact set $K_F \subseteq \Teich(\Pi_F)$ that contains the image of $\tau_F$. It is more subtle if $F$ is external. In this case, we first note that $P_i$ is identified with finitely many faces, listed by $F_l, l =1,..., m$, of $\mathcal{G}^1_j, j \in \{1,\dots k\}$ by the subdivision rule. We can identify $F$ with the complement of such faces. We construct the compact set by applying Theorem \ref{thm:bitr} to $F_l \subset \mathcal{G}^1_j$. These compact sets give rise to a compact set
$$
K_i  = \prod K_F \subseteq \Teich(\mathcal{G}_i^1), i \in \{1,\dots, k\}.
$$
We define $C$ as the maximal diameter of $K_i$, and define $\delta$ as the supremum of the derivative of the skinning map in $K_i$. We conclude by a theorem of McMullen (see Lemma~\ref{lem:nuc}) that $\delta< 1$.
}
\subsubsection{A finite list of maps $\tau_F$}\label{sss:maps}
Let $F$ be a face of $\mathcal{G}_i^1$ ({which may be external or non-external}).
As $F$ is an induced Jordan domain, let $\mathcal{H}_F = \partial F$ be the induced subgraph of $\mathcal{G}_i^1$. Then $\mathcal{H}_F$ has exactly two faces, one of which being $F$; denote by $F^c$ the other face.
Thus, we have
$$
\Teich(\mathcal{H}_F) = \Teich(\Pi_F) \times \Teich(\Pi_{F^c}).
$$
We define the map
$$
\tau_F: \Teich(\mathcal{G}_i^1) = \prod_{S \text{ face of }\mathcal{G}_i^1} \Teich(\Pi_S) \longrightarrow \Teich(\Pi_{F^c})
$$
by $\tau_F(\mathcal{P}) = \pi_{F^c}(\tau_{\mathcal{H}_F, \mathcal{G}^1_i}(\mathcal{P}))$, where $\tau_{\mathcal{H}_F, \mathcal{G}^1_i}: \Teich(\mathcal{G}_i^1) \longrightarrow \Teich(\mathcal{H}_F)$ is the skinning map and $\pi_{F^c}: \Teich(\mathcal{H}_F) \longrightarrow \Teich(\Pi_{F^c})$ is the projection map.

By Theorem \ref{thm:csk}, $\tau_F$ is the composition of a projection map and the restriction of Thurston's skinning map.
Since Thurston's skinning map is always a contraction with respect to the Teichm\"uller metric (see \cite[Theorem 5.3]{McM90}), we have the following lemma.
\begin{lem}\label{lem:nuc}
    Let $F$ be a non-external face of $\mathcal{G}_i^1$.
    Then for any $\mathcal{P} \in \Teich(\mathcal{G}_i^1)$, we have
    $$
    \| d\tau_F (\mathcal{P}) \| < 1,
    $$
    where the norm is computed with respect to the Teichm\"uller metrics on $\Teich(\mathcal{G}_i^1)$ and $\Teich(\Pi_{F^c})$.
\end{lem}

\subsubsection{A finite list of compact sets $K_i \subseteq  \Teich(\mathcal{G}_i^1)$}\label{sss:compactset}
Let $F$ be a face of $\mathcal{G}_i^1$. {We now construct a compact set $K_F \subseteq \Teich(\Pi_F)$ that contains all geodesic segments connecting two points in the skinning image.}

We consider two cases.

Case (1): {$F$ is non-external.}
Then $F$ is acylindrical in $\mathcal{G}_i^2$.
So by Theorem \ref{thm:bitr}, we have that $\pi_{F}(\tau_{\mathcal{G}_i^1, \mathcal{G}_i^2}(\Teich(\mathcal{G}_i^2)))$ is a bounded set.
{Let $B(a, R)\subseteq\Teich(\Pi_F)$ be a ball of radius $R$ covering $\pi_{F}(\tau_{\mathcal{G}_i^1, \mathcal{G}_i^2}(\Teich(\mathcal{G}_i^2)))$. Let $K_F = \overline{B(a, 2R)}$.
Then $K_F$ is a compact set that contains every geodesic segments connecting two points in $\pi_{F}(\tau_{\mathcal{G}_i^1, \mathcal{G}_i^2}(\Teich(\mathcal{G}_i^2)))$.}


Case (2): {$F$ is external.}
Let $j = 1,..., k$.
Then there is a finite list of non-external faces $F_l$ of $\mathcal{G}^1_j$ that are identified with $P_i$ by the subdivision rule.
{Let $F_l$ be such a face. Let $\mathcal{H}_{l} = \partial F_l$ be the induced subgraph of $\mathcal{G}^1_j$ and let $F^c_l$ be the complementary face of $\mathcal{H}_{l}$.
Since $F$ is the external face, we have a natural identification between $\Teich(\Pi_{F^c_l})$ and $\Teich(\Pi_F)$.
Since $\mathcal{R}$ is acylindrical, it is easy to see that $F^c_l$ is acylindrical in $\mathcal{G}_j^2$.
Thus, by Theorem \ref{thm:bitr}, we have that
$$
\pi_{F^c_l}(\tau_{\mathcal{H}_{l}, \mathcal{G}_j^2}(\Teich(\mathcal{G}_j^2))) \subseteq \Teich(\Pi_{F^c_l}) \cong \Teich(\Pi_F)
$$
is a bounded set.
Therefore $A_F:= \bigcup_l \pi_{F^c_l}(\tau_{\mathcal{H}_{l}, \mathcal{G}_j^2}(\Teich(\mathcal{G}_j^2))) \subseteq \Teich(\Pi_F)$ is bounded, where the union is over all non-external faces of $\mathcal{G}^1_j, j = 1,..., k$ that are identified with $P_i$.
We use a similar construction as in Case (1) to obtain a compact set $K_F \subseteq \Teich(\Pi_F)$ that contains all every geodesic segments connecting two points in $A_F$.}

We call $K_F$ the corresponding compact set for the face $F$, and denote
$$
K_i:= \prod_{F \text{ face of } \mathcal{G}^1_i} K_F \subseteq \prod_{F \text{ face of } \mathcal{G}^1_i} \Teich(\Pi_F) = \Teich(\mathcal{G}_i^1)
$$
the corresponding compact set for $\mathcal{G}_i^1$.

\subsection*{Diameter $C$ and contraction constant $\delta$}
Let $C_i$ be the diameter of $K_i$, and $C := \max_{i=1,..., k} C_i$.
Let 
$$
\delta_i:= \max_{F \text{ non-external face of }\mathcal{G}^1_i}\max_{\mathcal{P} \in K_i} \| d\tau_F (\mathcal{P}) \|,
$$
and let $\delta:= \max_{i=1,..., k} \delta_i$.
By Lemma \ref{lem:nuc}, $\delta < 1$.

\subsection{Fixed external class}\label{ss:fixedexternalclass}
In this subsection, we shall discuss exponential convergence and rigidity of circle packings for subdivision graphs with a fixed external class. Similar arguments give the same results for circle packings for spherical subdivision graphs.

Let $F^{\ext}_i$ be the external face for $\mathcal{G}_i^n$.
Recall that
$$
    \Teich(\mathcal{G}^n_i) = \prod_{F \text{ non-external}} \Teich(\Pi_{F}) \times \Teich(\Pi_{F^{\ext}_i}).
$$
Fix a conformal structure $\sigma_{\ext} \in \Teich(\Pi_{F^{\ext}_i})$, and consider the fiber
$$
\Teich_{\sigma_{\ext}}(\mathcal{G}_i^n) = \prod_{F  \text{ non-external}} \Teich(\Pi_{F}) \times \{\sigma_{\ext}\}.
$$

To ease the notations, we shall often drop the subscripts for the skinning maps, and denote
$$
\tau = \tau_{\mathcal{G}_i^n, \mathcal{G}_i^{n+1}}:\Teich(\mathcal{G}_i^{n+1}) \longrightarrow \Teich(\mathcal{G}_i^{n}).
$$
We shall denote by $\tau^l: \Teich(\mathcal{G}_i^{n+l}) \longrightarrow \Teich(\mathcal{G}_i^{n})$ the composition of the skinning maps.

\begin{lem}\label{lem:cf}
    Let $k \geq 1$, and let $\mathcal{P}_1, \mathcal{P}_2 \in \Teich_{\sigma_{\ext}}(\mathcal{G}_i^{k+2})$.
    Then
    $$
    d(\tau^2(\mathcal{P}_1), \tau^2(\mathcal{P}_2)) \leq \delta \cdot d(\tau(\mathcal{P}_1), \tau(\mathcal{P}_2)) \leq \min\{C, d(\mathcal{P}_1, \mathcal{P}_2)\}\cdot \delta,
    $$
    where $C$ is the diameter and $\delta$ is the contraction constant defined in \S \ref{subsec:su}, and $d$ is the Teichm\"uller distance on the corresponding spaces.
\end{lem}
\begin{proof}
    Let $F$ be a face of $\mathcal{G}_i^k$. We consider two cases.
    
    Case(1): Suppose that $F$ is non-external.
    Let 
    $$
    \mathcal{G}_F:= \mathcal{G}_i^{k+1} \cap F.
    $$
    Then $\mathcal{G}_F$ is identified by the subdivision rule to $\mathcal{G}^1_j$ for some $j$.
    Therefore, we can identify $\Teich(\mathcal{G}_F)$ with $\Teich(\mathcal{G}^1_j)$.
    Since $k\geq 1$ and $F$ is a non-external face of $\mathcal{G}_i^k$, by our construction of the compact sets, we have
    $$
    \tau_{\mathcal{G}_F, \mathcal{G}_i^{k+2}}(\Teich(\mathcal{G}_i^{k+2})) \subseteq K_j \subseteq \Teich(\mathcal{G}^1_j).
    $$
    Let $F^{\ext}_{j}$ be the external face for $\mathcal{G}^1_j$.
    Note that $(F^{\ext}_{j})^c$ is $P_j$, which can be identified with $F$.
    Thus, we have the following commutative diagram
    \[ \begin{tikzcd}
\Teich(\mathcal{G}^{k+2}_i) \arrow{r}{\tau} \arrow[swap]{rd}{\tau_{\mathcal{G}_F, \mathcal{G}^{k+2}_i}} & \Teich(\mathcal{G}^{k+1}_i) \arrow{r}{\tau} \arrow{d}{\tau_{\mathcal{G}_F, \mathcal{G}^{k+1}_i}} & \Teich(\mathcal{G}^{k}_i) \arrow{d}{\pi_F} \\%
& K_j \subseteq \Teich(\mathcal{G}_F) \cong \Teich(\mathcal{G}^1_j) \arrow{r}{\tau_{F_{j}}} & \Teich(\Pi_{(F^{\ext}_{j})^c}) \cong \Teich(\Pi_F)
\end{tikzcd}
\]
    Note that no map in the commutative diagram expands the Teichm\"uller metric.
    Let $\mathcal{P} \in \tau(\Teich(\mathcal{G}_i^{k+2}))$.
    Then $\tau_{\mathcal{G}_F, \mathcal{G}^{k+1}_i} (\mathcal{P}) \in K_j$.
    Thus
    \begin{align*}
    \|d (\pi_F \circ \tau) (\mathcal{P})\| &= \|d (\tau_{F_{j}} \circ \tau_{\mathcal{G}_F, \mathcal{G}^{k+1}_i}) (\mathcal{P})\|\\ 
    &= \|d \tau_{F_{j}} (\tau_{\mathcal{G}_F, \mathcal{G}^{k+1}_i} (\mathcal{P}))\| \cdot \|d \tau_{\mathcal{G}_F, \mathcal{G}^{k+1}_i} (\mathcal{P})\| \\
    &\leq \delta \cdot 1 = \delta.
    \end{align*}
    Since the Teichm\"uller geodesic connecting $\tau_{\mathcal{G}_F, \mathcal{G}^{k+2}_i}(\mathcal{P}_1)$ and $\tau_{\mathcal{G}_F, \mathcal{G}^{k+2}_i}(\mathcal{P}_2)$ is contained in $K_j$, we conclude that
    $ \|d (\pi_F \circ \tau) (\mathcal{P})\| \leq \delta$ along the geodesic connecting $\tau(\mathcal{P}_1)$ and $\tau(\mathcal{P}_2)$. Therefore, by integration along this geodesic,
    $$
    d((\pi_F \circ \tau^2)(\mathcal{P}_1), (\pi_F \circ \tau^2)(\mathcal{P}_2)) \leq \delta \cdot d(\tau(\mathcal{P}_1), \tau(\mathcal{P}_2)).
    $$

    Case(2): Suppose $F$ is the external face. By construction, we have
    $(\pi_{F} \circ \tau^2)(\mathcal{P}) = \sigma_{\ext}$ for all $\mathcal{P} \in \Teich_{\sigma_{\ext}}(\mathcal{G}_i^{k+2})$.
    Thus 
    $$
    d((\pi_F \circ \tau^2)(\mathcal{P}_1), (\pi_F \circ \tau^2)(\mathcal{P}_2)) = 0 \leq \delta \cdot d(\tau(\mathcal{P}_1), \tau(\mathcal{P}_2)).
    $$

    Combining the two, we have 
    $$
    d(\tau^2(\mathcal{P}_1), \tau^2(\mathcal{P}_2)) \leq \delta \cdot d(\tau(\mathcal{P}_1), \tau(\mathcal{P}_2)).
    $$

    Since the diameter of $K_j$ is bounded by $C$, we have $d(\tau(\mathcal{P}_1), \tau(\mathcal{P}_2)) \leq C$.
    Since $d(\tau(\mathcal{P}_1), \tau(\mathcal{P}_2)) \leq d(\mathcal{P}_1, \mathcal{P}_2)$, the lemma follows.
\end{proof}

By induction, we have the following immediate corollary.
\begin{cor}\label{cor:cf}
    Let $k \geq 1$ and $n \geq 1$.
    Let $\mathcal{P}_1, \mathcal{P}_2 \in \Teich_{\sigma_{\ext}}(\mathcal{G}_i^{k+n+1})$.
    Then
    $$
    d(\tau^{n+1}(\mathcal{P}_1), \tau^{n+1}(\mathcal{P}_2))\leq \delta^n \cdot d(\tau(\mathcal{P}_1), \tau(\mathcal{P}_2)) \leq \min\{C, d(\mathcal{P}_1, \mathcal{P}_2)\}\cdot\delta^{n}
    $$
\end{cor}

    Let $\mathcal{G}$ be a simple spherical subdivision graph with finite subdivision rule $\mathcal{R}$.
    By Proposition \ref{prop:jdg}, we may assume that all faces of $\mathcal{G}^n$ are induced Jordan domains.
    Abusing the notations, let
    $$
    \tau = \tau_{\mathcal{G}^n, \mathcal{G}^{n+1}}: \Teich(\mathcal{G}^{n+1}) \longrightarrow\Teich(\mathcal{G}^{n}).
    $$
A similar proof also gives the following version for spherical subdivision graphs.
\begin{lem}\label{lem:cfg}
    Let $k \geq 1$ and $n \geq 1$.
    Let $\mathcal{P}_1, \mathcal{P}_2 \in \Teich(\mathcal{G}^{k+n+1})$.
    Then
    $$
    d(\tau^{n+1}(\mathcal{P}_1), \tau^{n+1}(\mathcal{P}_2)) \leq \min\{C, d(\mathcal{P}_1, \mathcal{P}_2)\}\cdot \delta^{n}.
    $$
\end{lem}

\subsection*{Exponential convergence of finite approximations}
The following estimate is classical and useful in our setting.
\begin{lem}\label{lem:qc}
    Let $\Psi: \widehat \C \longrightarrow \widehat \C$ be a $K$-quasiconformal map that fixes $0,1,\infty$.
    Then for $t \in X:=\widehat\C - \{0, 1,\infty\}$,
    $$
    d_{X}(t, \Psi(t)) \leq \log K,
    $$
    where $d_{X}$ is the hyperbolic metric on $X$.

    More generally, let $\Psi: \widehat \C \longrightarrow \widehat \C$ be a $K$-quasiconformal map sending $a, b, c$ to $a', b', c'$.
    Let $M$ be a M\"obius transformation that sends $a, b, c$ to $a', b', c'$.
    Then for $t \in X_{a,b,c}:=\widehat \C - \{a, b, c\}$, 
    $$
    d_{X_{a',b',c'}}(M(t), \Psi(t)) \leq \log K,
    $$
    where $d_{X_{a',b',c'}}$ is the hyperbolic metric on $X_{a',b',c'}:=\widehat\C - \{a', b', c'\}$.
\end{lem}
\begin{proof}
    Note that $t\in X$ determines a Riemann surface $X_t:=\widehat\C - \{0, 1, \infty, t\}$ in the moduli space of 4-times punctured spheres, and the map $\Psi$ gives a quasiconformal map between the two Riemann surfaces $X_t$ and $X_{\Psi(t)}$.
    Thus $d(X_t, X_{\Psi(t)}) \leq \frac{1}{2}\log K$.
    Since the Teichm\"uller metric descends to $\frac{1}{2} d_{X}$ on the moduli space, 
    $$
    d_{X}(t, \Psi(t)) = 2 d(X_t, X_{\Psi(t)}) \leq \log K.
    $$
    The more general part follows by composing with M\"obius transformations.
\end{proof}

\begin{proof}[Proof of Theorem \ref{thm:LB} and Theorem \ref{thm:B}]
    Theorem \ref{thm:LB} now follows immediately from Corollary \ref{cor:cf}, Lemma \ref{lem:qc} and the fact that the spherical metric is bounded above by the hyperbolic metric of $\widehat\C - \{0, 1, \infty\}$.
    Similarly, Theorem \ref{thm:B} follows from Lemma \ref{lem:cfg} and Lemma \ref{lem:qc}.
\end{proof}

\subsection*{Rigidity of circle packings}
We will use the exponential convergence to prove rigidity of circle packings.
\begin{theorem}\label{thm:rig}
    Let $\mathcal{R}$ be a simple irreducible acylindrical finite subdivision rule, with subdivision graphs $\{\mathcal{G}_i\}_{i=1}^k$, and external face $\{F^{\ext}_i\}_{i=1}^k$.
    Let $\mathcal{G}$ be a simple spherical subdivision graph for $\mathcal{R}$.
    Then
    \begin{itemize}
        \item there exists a unique circle packing with nerve $\mathcal{G}_i$ and external class $\sigma_i \in \Teich(\Pi_{F^{\ext}_i})$; and
        \item there exists a unique circle packing with nerve $\mathcal{G}$.
    \end{itemize}
\end{theorem}
\begin{proof}
    First note that the existence is guaranteed by Theorem \ref{thm:exc} and Proposition \ref{prop:surj}.
    
    Let $\mathcal{P}_1$ and $\mathcal{P}_2$ be two {infinite} circle packings with nerve $\mathcal{G}_i$ and external class $\sigma_i \in \Teich(\Pi_{F^{\ext}_i})$.
    Let $\mathcal{P}^k_j, \mathcal{P}^{k+n+1}_j, j=1,2$ be the finite sub-circle packing associated to $\mathcal{G}^k_i$ and $\mathcal{G}^{k+n+1}_i$.
    {Note that $\mathcal{P}^k_j = \tau^{n+1}(\mathcal{P}^{k+n+1}_j)$.}
    Thus by Corollary \ref{cor:cf},
    $$
        {d(\mathcal{P}^k_1, \mathcal{P}^k_2) \leq \min\{C, d(\mathcal{P}^{k+n+1}_1, \mathcal{P}^{k+n+1}_2)\} \cdot \delta^n \leq C \delta^n.}
    $$
    Since this is true for all $n$, we conclude that $\mathcal{P}^k_1$ and $\mathcal{P}^k_2$ are conformally homeomorphic, i.e., there exists a M\"obius transformation $M$ so that $M(\Lambda(\mathcal{P}^k_1)) = \Lambda(\mathcal{P}^k_2)$.
    Since this is true for all $k$, we conclude that $\mathcal{P}_1$ and $\mathcal{P}_2$ are conformally homeomorphic.

    The rigidity of circle packings for spherical subdivision graphs follows by a similar argument.
\end{proof}

Theorem \ref{thm:A} now follows immediately.
\begin{proof}[Proof of Theorem \ref{thm:A}]
    By combining Theorem \ref{thm:exc} and Theorem \ref{thm:rig}, we prove the statement.
\end{proof}

\subsection{Varying external classes}\label{subsec:vec}
We now explain how circle packings vary as we vary the external classes.

Let $\mathcal{P}$ and $\mathcal{P}'$ be two infinite circle packings with nerve $\mathcal{G}_i$ and external class $\sigma, \sigma' \in \Teich(\Pi_{F^{\ext}_i})$ respectively.
Let $\mathcal{H}$ be an induced finite subgraph of $\mathcal{G}_i$ so that every face is an induced Jordan domain.
Let $F$ be a face of $\mathcal{H}$. To simplify the notations, we define 
$$
\rho_F = \pi_F \circ \tau_{\mathcal{H}, \mathcal{G}_i}: \Teich(\mathcal{G}_i) \longrightarrow \Teich(\Pi_F).
$$
We remark that in the spacial case when $F$ is a face of $\mathcal{G}^1_i$, $\mathcal{H} = \partial F$, and $F^c$ is the other face bounded by $\mathcal{H}$, then
$$
\rho_{F^c} = \tau_F,
$$
where $\tau_F$ is the skinning map associated to a face defined in \S~\ref{sss:maps}.

{We remark that if $\mathcal{H}'$ is another such finite subgraph so that $F$ is a face of $\mathcal{H}'$, then $\pi_F \circ \tau_{\mathcal{H}, \mathcal{G}_i} = \pi_F \circ \tau_{\mathcal{H}', \mathcal{G}_i}$.}
We say a face $F$ of $\mathcal{H}$ is 
\begin{itemize}
    \item {\em external} if $F$ contains the external face $F^{\ext}_i$ of $\mathcal{G}_i$;
    \item {\em maximal} in $\mathcal{H}$ if {
    \begin{align*}
    d(\rho_{F}(\mathcal{P}), \rho_{F}(\mathcal{P}')) &= d(\tau_{\mathcal{H},\mathcal{G}_i}(\mathcal{P}), \tau_{\mathcal{H},\mathcal{G}_i}(\mathcal{P}')) \\&= \max_{F' \text{ face }\mathcal{H}} d(\rho_{F'}(\mathcal{P}), \rho_{F'}(\mathcal{P}')).
    \end{align*}}
\end{itemize}

\begin{lem}\label{lem:dl}
    Let $\mathcal{P}$ and $\mathcal{P}'$ be two infinite circle packings with nerve $\mathcal{G}_i$ and external class $\sigma, \sigma' \in \Teich(\Pi_{F^{\ext}_i})$. Then for each $n \geq 1$ and $i =1,.., k$, the external face of $\mathcal{G}^n_i$ is maximal:
    $$
    d(\sigma, \sigma') = d(\tau_{\mathcal{G}^n_i,\mathcal{G}_i}(\mathcal{P}), \tau_{\mathcal{G}^n_i,\mathcal{G}_i}(\mathcal{P}')).
    $$
\end{lem}
\begin{proof}
    We prove the case for $n = 1$, the general case is the same.
    
    Suppose not. Then there exists a non-external face $F^1$ of $\mathcal{G}^1_i$ that is maximal.
    We will now construct a nested sequence of non-external maximal face $F^n$ of $\mathcal{G}_i^n$ so that
    $$
    d(\rho_{F^n}(\mathcal{P}), \rho_{F^n}(\mathcal{P}')) \to \infty,
    $$
    which would give a contradiction with our Bounded Image Theorem.
        
    Let $F^{1,c} = \overline{\widehat\C - F^1}$.
    Then $F^{1,c}$ is the external face for $\mathcal{G}_{F^1}$.
    Consider the skinning map $\tau_{F^1}{=\pi_{F^{1,c}}\circ\tau_{\partial F^1,\mathcal{G}_i}}: \Teich(\mathcal{G}_i^1) \longrightarrow\Teich(\Pi_{F^{1,c}})$.
    Then by definition, $\tau_{F^1}(\tau_{\mathcal{G}^1_i,\mathcal{G}_i}(\mathcal{P})) = \rho_{F^{1,c}}(\mathcal{P})$.
    {Since $F^1$ is non-external, by Lemma~\ref{lem:nuc}}, we have
    \begin{equation}\label{eq:1}
    \begin{aligned}
    d(\rho_{F^{1,c}}(\mathcal{P}), \rho_{F^{1,c}}(\mathcal{P}')) &= 
    d(\tau_{F^1}(\tau_{\mathcal{G}^1_i,\mathcal{G}_i}(\mathcal{P})), \tau_{F^1}(\tau_{\mathcal{G}^1_i,\mathcal{G}_i}(\mathcal{P}'))) \\
    &< d(\tau_{\mathcal{G}^1_i,\mathcal{G}_i}(\mathcal{P}), \tau_{\mathcal{G}^1_i,\mathcal{G}_i}(\mathcal{P}')) \\
    &= d(\rho_{F^{1}}(\mathcal{P}), \rho_{F^{1}}(\mathcal{P}')).
    \end{aligned}
    \end{equation}

    \begin{figure}[htp]
    \centering
    \includegraphics[width=0.5\textwidth]{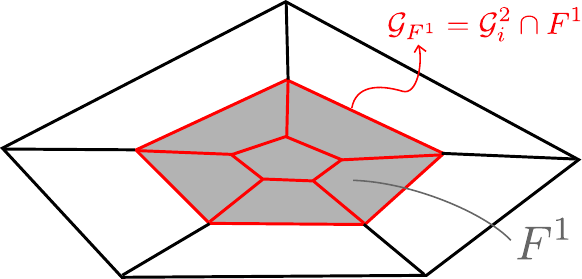}
    \caption{An illustration for the proof of Lemma~\ref{lem:dl}. The shaded region is the face $F^1$ and the graph $\mathcal{G}_{F^1}:= \mathcal{G}^2_i \cap {F^1}$ is colored in red.}
    \label{fig:G_F}
    \end{figure}

    Let $\mathcal{G}_{F^1}:= \mathcal{G}^2_i \cap {F^1}$ (see Figure~\ref{fig:G_F}).
    By the subdivision rule, $\mathcal{G}_{F^1}$ is identified with $\mathcal{G}^1_j$ for some $j = j(F^1) \in \{1,..., k\}$. {Thus, $\Teich(\mathcal{G}_{F^1}) \cong \Teich(\mathcal{G}^1_j)$}, and
    $$
    \tau_{\mathcal{G}_{F^1}, \mathcal{G}_i}(\mathcal{P}), \tau_{\mathcal{G}_{F^1}, \mathcal{G}_i}(\mathcal{P}') \in K_j \subseteq \Teich(\mathcal{G}^1_j).
    $$
    Consider the map {$\tau_{F^{1,c}} := \pi_{F^1} \circ \tau_{\partial F^1, \mathcal{G}_{F^1}}: \Teich(\mathcal{G}_{F^1}) \longrightarrow\Teich(\Pi_{F^1})$. Under the identification of $\mathcal{G}_{F^1}$ with $\mathcal{G}^1_j$, $\tau_{F^{1,c}}$ is conjugate to the skinning map $\tau_{F^{\ext}_j}$. Thus, by our choice of $\delta$ (see \S~\ref{subsec:su}), we have}
    \begin{equation}\label{eq:2}
    \begin{aligned}
    d(\rho_{F^{1}}(\mathcal{P}), \rho_{F^{1}}(\mathcal{P}')) &= d(\tau_{F^{1,c}}(\tau_{\mathcal{G}_{F^1},\mathcal{G}_i}(\mathcal{P})), \tau_{F^{1,c}}(\tau_{\mathcal{G}_{F^1},\mathcal{G}_i}(\mathcal{P}'))) \\
    &\leq \delta \cdot d(\tau_{\mathcal{G}_{F^1},\mathcal{G}_i}(\mathcal{P}), \tau_{\mathcal{G}_{F^1},\mathcal{G}_i}(\mathcal{P}')).  
    \end{aligned}
    \end{equation}
    Let $F^2$ be a maximal face in $\mathcal{G}_{F^1}$.
    By Equations \ref{eq:1} and \ref{eq:2}, we have that $ d(\rho_{F^{1,c}}(\mathcal{P}), \rho_{F^{1,c}}(\mathcal{P}')) < \delta \cdot d(\tau_{\mathcal{G}_{F^1},\mathcal{G}_i}(\mathcal{P}), \tau_{\mathcal{G}_{F^1},\mathcal{G}_i}(\mathcal{P}'))$, i.e., $F^{1,c}$ is not maximal in $\mathcal{G}_{F^1}$.
    Thus, $F^2$ is non-external in $\mathcal{G}_{F^1}$ with
    $$
    d(\rho_{F^{2}}(\mathcal{P}), \rho_{F^{2}}(\mathcal{P}')) \geq \delta^{-1} d(\rho_{F^{1}}(\mathcal{P}), \rho_{F^{1}}(\mathcal{P}')).
    $$
    Now inductively apply the same argument for $\mathcal{G}_{F^{n-1}} = \mathcal{G}^n_i \cap F^{n-1}$, we obtain a non-external maximal face $F^n$ of $\mathcal{G}_{F^{n-1}}$ with
    \begin{align*}
    d(\rho_{F^{n}}(\mathcal{P}), \rho_{F^{n}}(\mathcal{P}')) &\geq \delta^{-1} d(\rho_{F^{n-1}}(\mathcal{P}), \rho_{F^{n-1}}(\mathcal{P}')) \\
    &\geq \delta^{-n} d(\rho_{F^{1}}(\mathcal{P}), \rho_{F^{1}}(\mathcal{P}')) \to \infty.   
    \end{align*}
    But this is a contradiction to Theorem \ref{thm:bitr} as $F^n$ is acylindrical in $\mathcal{G}^{n+1}_i$.
\end{proof}
By induction, we have
\begin{prop}\label{prop:isom}
    Let $\mathcal{P}$ and $\mathcal{P}'$ be two infinite circle packings with nerve $\mathcal{G}_i$ and external class $\sigma, \sigma' \in \Teich(\Pi_{F^{\ext}_i})$.
    Then 
    $$
    d(\mathcal{P}, \mathcal{P}') = d(\sigma, \sigma').
    $$
    In particular, $\mathcal{P}$ and $\mathcal{P}'$ are quasiconformally homeomorphic.
\end{prop}
\begin{proof}
    Let $\Pi_{\ext}$ and $\Pi_{\ext}'$ be the interstices for $\mathcal{P}$ and $\mathcal{P}'$ that correspond to the external face of $\mathcal{G}_i$.
    Since any quasiconformal homeomorphism between $\mathcal{P}$ and $\mathcal{P}'$ sends $\Pi_{\ext}$ to $\Pi_{\ext}'$ and preserves the tangent points, and $d(\sigma, \sigma')$ gives the smallest dilatation of such maps, we have $d(\mathcal{P}, \mathcal{P}') \geq d(\sigma, \sigma')$.

    To prove the other direction, by Lemma \ref{lem:dl}, we have
    $$
    d(\tau_{\mathcal{G}^n_i, \mathcal{G}_i}(\mathcal{P}), \tau_{\mathcal{G}^n_i, \mathcal{G}_i}(\mathcal{P}')) = d(\sigma, \sigma').
    $$
    Thus, we have a sequence of uniformly quasiconformal homeomorphisms between $\tau_{\mathcal{G}^n_i, \mathcal{G}_i}(\mathcal{P}), \tau_{\mathcal{G}^n_i, \mathcal{G}_i}(\mathcal{P}')$ with dilatation $e^{2d(\sigma, \sigma')}$.
    By compactness of $K$-quasiconformal maps and taking a sub-sequential limit, $\mathcal{P}, \mathcal{P}'$ are quasiconformally homeomorphic and $d(\mathcal{P}, \mathcal{P}') \leq d(\sigma, \sigma')$.
\end{proof}

As an immediate corollary, we have
\begin{cor}\label{cor:ib}
    Let $\Phi_i: \Teich(\mathcal{G}_i) \longrightarrow \Teich(\Pi_{F^{\ext}_{i}})$ be the external class map.
    Then $\Phi_i$ is an isometric bijection with respect to the Teichm\"uller metrics.
\end{cor}
\begin{proof}
    By Theorem \ref{thm:rig}, the map is a bijection. By Proposition \ref{prop:isom}, this map is an isometry.
\end{proof}

Theorem \ref{thm:LR} follows immediately.
\begin{proof}[Proof of Theorem \ref{thm:LR}]
    Since $\Teich(\Pi_{F^{\ext}_{i}})$ is homeomorphic to $\R^{e_i-3}$, the statement follows from Theorem \ref{thm:exc}, Theorem \ref{thm:rig} and Corollary \ref{cor:ib}.
\end{proof}

\subsection{Renormalization}\label{ss:renom}
Let $F$ be a non-external face of $\mathcal{G}^1_i$.
Then $F\cap \mathcal{G}_i$ is identified with $\mathcal{G}_{j(F)}$ for some $j(F) \in \{1,..., k\}$.
Denote this identification by
$$
\psi_F:\mathcal{G}_{j(F)} \longrightarrow F\cap \mathcal{G}_i.
$$
This identification induces an identification
$$
\psi_F^*: \Teich(F\cap \mathcal{G}_i) \longrightarrow \Teich(\mathcal{G}_{j(F)}).
$$
We define the renormalization operator
$$
\mathfrak{R}_{F}: \Teich(\mathcal{G}_i) \longrightarrow \Teich(\mathcal{G}_{j(F)})
$$
by $\mathfrak{R}_{F}(\mathcal{P}) = \psi_F^*(\tau_{F\cap \mathcal{G}_i, \mathcal{G}_i}(\mathcal{P}))$.

Let $F^1$ be a non-external face of $\mathcal{G}^1_i$, and let $F^n$ be a face of $\mathcal{G}^n_i$ with $F^n \subseteq F^{n-1}$.
We call $(F^n)_n$ a {\em nested} sequence of faces in $P_i$.

Given a nested sequence of faces $(F^n)_n$, let $j_n$ be the index so that $F^n$ is identified with $P_{j_n}$ by the subdivision rule.
Denote by $\psi_n:\mathcal{G}_{j_n} \longrightarrow F^n\cap \mathcal{G}_i$ the identification.
Then $\psi_n^{-1}(F^{n+1})$ is a non-external face of $\mathcal{G}^1_{j_n}$.
Thus, we can iterate the renormalization operator for a nested sequence of faces.
Abusing the notations, {we will omit the dependence on the nested faces when there is no ambiguity,} and write it as
$$
\mathfrak{R}^n:= \mathfrak{R}_{\psi_{n-1}^{-1}(F^n)} \circ \mathfrak{R}_{\psi_{n-2}^{-1}(F^{n-1})} \circ ... \circ \mathfrak{R}_{F^1} = \psi_n^*\circ \tau_{F^n\cap \mathcal{G}_i, \mathcal{G}_i}.
$$
\begin{prop}\label{prop:adf}
    Let $F^1,..., F^n$ be a nested sequence of faces in $P_i$, and let $\mathcal{P}, \mathcal{P}' \in \Teich(\mathcal{G}_i)$.
    Then
    \begin{align*}
    d(\mathfrak{R}^n(\mathcal{P}), \mathfrak{R}^n(\mathcal{P}')) &\leq \delta^{n-1} d(\mathfrak{R}(\mathcal{P}), \mathfrak{R}(\mathcal{P}')) \\
    &\leq \delta^{n-1} \min\{d(\mathcal{P}, \mathcal{P}'), C\}.   
    \end{align*}
\end{prop}
\begin{proof}
    We will prove the statement for $n=2$. The general case follows by induction.
    Let $\mathcal{G}_{F^1}^1 := \mathcal{G}_i^2 \cap F^1 \cong \mathcal{G}^1_{i_1}$.
    Then 
    $$
    \psi_1^*(\tau_{\mathcal{G}_{F^1}^1, \mathcal{G}_{i}}(\Teich(\mathcal{G}_i))) \subseteq K_{i_1}.
    $$
    Note that $F^2$ is a non-external face of $\mathcal{G}_{F^1}^1$.
    Let $F^{2, c} = \overline{\widehat\C - F^2}$ be the complementary face.
    Therefore, by considering the skinning map for the face $F^2$ in $\mathcal{G}_{F^1}^1$, we have
    \begin{align*}
        d(\rho_{F^{2, c}}(\mathcal{P}), \rho_{F^{2, c}}(\mathcal{P}')) &\leq \delta \cdot d(\tau_{\mathcal{G}_{F^1}^1, \mathcal{G}_{i}}(\mathcal{P}), \tau_{\mathcal{G}_{F^1}^1, \mathcal{G}_{i}}(\mathcal{P})) \\
        & = \delta \cdot d(\rho_{F^{1, c}}(\mathcal{P}), \rho_{F^{1, c}}(\mathcal{P}')),
    \end{align*}
    where the last equality follows from Lemma \ref{lem:dl}.

    Since the diameter of $K_{i_1}$ is bounded by $C$, we have
    \begin{align*}
        d(\rho_{F^{1, c}}(\mathcal{P}), \rho_{F^{1, c}}(\mathcal{P}')) \leq \min\{d(\mathcal{P}, \mathcal{P}'), C\}.
    \end{align*}
    By Proposition \ref{prop:isom}, for any $j$,
    $$
    d(\mathfrak{R}^j(\mathcal{P}), \mathfrak{R}^j(\mathcal{P}')) = d(\rho_{F^{j, c}}(\mathcal{P}), \rho_{F^{j, c}}(\mathcal{P}')),
    $$
    where $F^{j, c} = \overline{\widehat\C - F^j}$.
    The proposition now follows.
\end{proof}
\begin{rmk}
    We remark that the statement and the techniques used are similar as in Lemma \ref{lem:cf} and Corollary \ref{cor:cf}.
    Indeed, the difference here is that renormalization corresponds to {\em zooming in}, while finite approximation as in Lemma \ref{lem:cf} corresponds to {\em zooming out}.
    
    For renormalization, it is crucial that we work with infinite circle packings with nerve $\mathcal{G}_i$, instead of $\mathcal{G}^n_i$.
    Although the renormalization $\mathfrak{R}_F:\Teich(\mathcal{G}^n_i) \longrightarrow \Teich(\mathcal{G}^{n-1}_{j(F)})$ can be defined, it is not a contraction in general.
    Analogously, for finite approximation, it is crucial that the external class for the circle packing is fixed.
\end{rmk}

\subsection*{Encoding nested faces}
The information of an infinite nested sequence of faces can be coded by a point in the limit set of the circle packing.
More precisely, let $\Pi(\mathcal{P}_i)$ be the skinning intersetice of $\mathcal{P}_i$, i.e., the component of the complement $\widehat{\mathbb{C}}-\bigcup_{v\in\partial F^{\ext}_i}\overline{D}_v$ of the union of `outermost' closed disks that has non-trivial intersection with the limit set.
Recall that
$$
\Sigma:=\bigcup_{i=1}^k\{(\mathcal{P}_i,x): \mathcal{P}_i \text{ has nerve } \mathcal{G}_i, x\in \overline{\Pi(\mathcal{P}_i)}\}/\PSL_2(\C).
$$

Let $\mathcal{Q}_j$ be the sub-circle packings of $\mathcal{P}_i$ corresponding to a face of $\mathcal{G}^1_i$, and let $\Pi(\mathcal{Q}_j)$ be the corresponding skinning interstice.
Suppose that $x \in \overline{\Pi(\mathcal{Q}_j)} \subseteq \overline{\Pi(\mathcal{P}_i)}$.
Then we define the renormalization of $(\mathcal{P}_i, x)$ by
$$
\mathfrak{R}((\mathcal{P}_i, x)) = (\mathcal{Q}_j, x) \in \Sigma.
$$

We remark that technically, if $x\in \partial \overline{\Pi(\mathcal{Q}_j)} \cap \partial \overline{\Pi(\mathcal{Q}_k)}$, then the definition of $\mathfrak{R}$ requires a choice.
To deal with this ambiguity, let $\partial_{cusp} \Pi(\mathcal{P}_i)$ be the set of cusps on $\partial \Pi(\mathcal{P}_i)$, i.e., the set of tangent points of the `outermost' closed disks.
We define the space and the projection map 
$$
\pi: \widetilde\Sigma \longrightarrow \Sigma
$$ 
by blowing up a point $(\mathcal{P}_i,x) \in \Sigma$ into two points if $x \in \overline{\Pi(\mathcal{P}_i)} - \partial_{cusp} \Pi(\mathcal{P}_i)$ and $x$ is the tangent point of two circles in $\mathcal{P}_i$.
Let 
$$
\Sigma^{\infty}:=\bigcup_{i=1}^k\{(\mathcal{P}_i,x): \mathcal{P}_i \text{ has nerve } \mathcal{G}_i, x\in \overline{\Pi(\mathcal{P}_i)} \cap \Lambda(\mathcal{P}_i)\}/\PSL_2(\C),
$$
and $\widetilde\Sigma^\infty = \pi^{-1} (\Sigma^{\infty})$.
Then elements in $\widetilde\Sigma^\infty$ are in one-to-one correspondence with an infinite nested sequence of faces.
\begin{proof}[Proof of Theorem \ref{thm:EC}]
    The statement follows from the above discussion and Proposition \ref{prop:adf}.
\end{proof}

\subsection{Teichm\"uller mapping}\label{subsec:tm}
Let $\mathcal{P}$ and $\mathcal{P}'$ be two infinite circle packings with nerve $\mathcal{G}_i$.
Then by Proposition \ref{prop:isom}, they are quasiconformally homeomorphic, i.e., there exists a quasiconformal map $\Psi: \widehat\C \longrightarrow\widehat\C$ with
$$
\Psi(\Lambda(\mathcal{P})) = \Lambda(\mathcal{P}').
$$
Note that there is no restriction of the homeomorphism in the interior of the disks of the circle packing.
In the following, we will define the Teichm\"uller mapping between the two circle packings, which satisfies some additional properties.
We will also prove the existence of the Teichm\"uller mapping.
Further properties of Teichm\"uller mappings will be explored in \S \ref{subsec:cdp}.

Let $\mathcal{P}$ be an infinite circle packing with nerve $\mathcal{G}_i$, and $\mathcal{P}^n = \tau_{\mathcal{G}^n_i, \mathcal{G}_i}(\mathcal{P})$ be the finite sub-circle packing of $\mathcal{P}$ associated to the finite subgraph $\mathcal{G}^n_i \subseteq \mathcal{G}_i$.


Suppose $F$ is a face of $\mathcal{G}^n_i$.
Let $\Pi_F$ be the interstice of $\mathcal{P}^n$ associated to the face $F$.
Let $\Gamma_F$ be the group generated by reflections along the circles in $\mathcal{P}^n$ associated to the vertices in $\partial F$.
Let $\Omega_F$ be the component of domain of discontinuity of $\Gamma_F$ that contains $\Pi_F$, and $\Lambda_F = \partial \Omega_F$ be the limit set of $\Gamma_F$.
See Figure~\ref{fig:reflection} for an illustration.
We also denote by $\mathcal{P}_{F}$ the infinite sub-circle packing associated to $F \cap \mathcal{G}_i$.

 \begin{figure}[htp]
    \centering
    \includegraphics[width=0.3\textwidth]{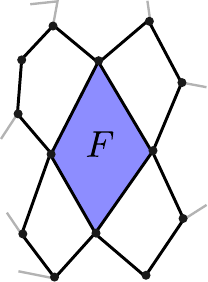}
    \includegraphics[width=0.5\textwidth]{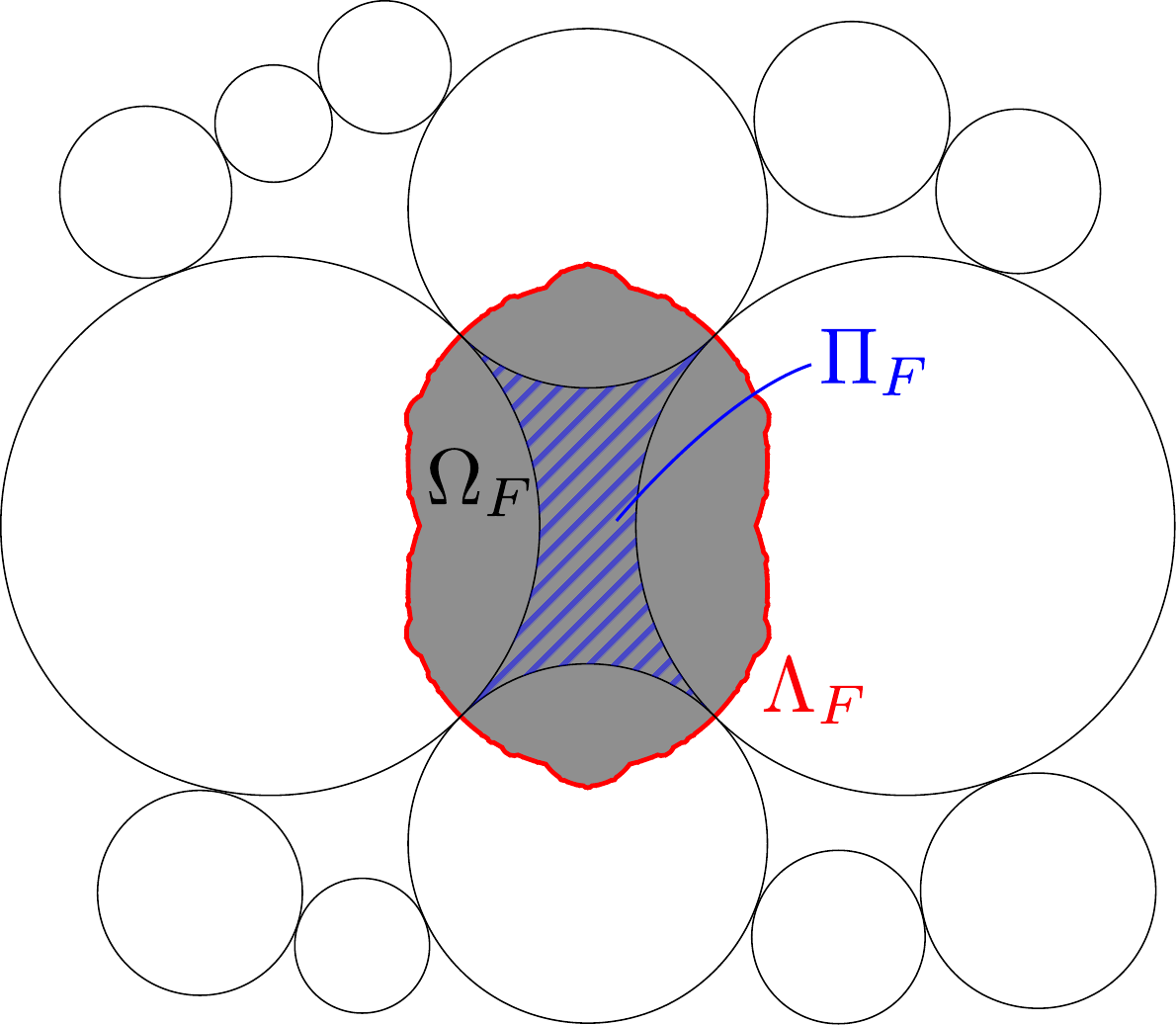}
    \caption{A face $F$, its interstice $\Pi_F$, the corresponding component of the domain of discontinuity $\Omega_F$ for the group $\Gamma_F$ generated by reflections in circles bounding $\Pi_F$, whose boundary is the limit set $\Lambda_F$.}
    \label{fig:reflection}
\end{figure}

\begin{defn}[Teichm\"uller mapping]\label{defn:tm}
    Let $\mathcal{P}, \mathcal{P}'$ be two infinite circle packings with nerve $\mathcal{G}_i$.
    A quasiconformal homeomorphism $\Psi$ between $\mathcal{P}$ and $\mathcal{P}'$ is called a {\em Teichm\"uller mapping} between $\mathcal{P}$ and $\mathcal{P}'$ if for each non-external face $F$ of $\mathcal{G}^n_i$ with $n\geq 0$, we have
    \begin{enumerate}
        \item $\frac{1}{2}\log(K(\Psi|_{\Omega_F})) = d(\mathcal{P}_{F}, \mathcal{P}'_{F})$, and
        \item $\Psi|_{\Lambda_F}: \Lambda_F \longrightarrow \Lambda'_F$ conjugates the dynamics of $\Gamma_F$ and $\Gamma'_F$.
    \end{enumerate}
    In particular, a Teichm\"uller mapping minimizes the dilatation in $\Omega(F)$ for {\em every} non-external face $F$ of $\mathcal{G}^n_i$.
\end{defn}

By Proposition \ref{prop:adf}, if $F$ is a non-external face of $\mathcal{G}^n_i$ with $n \geq 1$, then
$$
d(\mathcal{P}_{F}, \mathcal{P}'_{F}) \leq \delta^{n-1}\min\{d(\mathcal{P}, \mathcal{P}'), C\}.
$$
Therefore, as an immediate corollary, we have
\begin{cor}
     Let $\mathcal{P}, \mathcal{P}'$ be two infinite circle packings with nerve $\mathcal{G}_i$, and $\Psi:\widehat\C\longrightarrow\widehat\C$ be a Teichm\"uller mapping between them.
     Then 
     \begin{enumerate}
         \item $\frac{1}{2}K(\Psi) = d(\mathcal{P}, \mathcal{P}')$; and
         \item for each non-external face $F$ of $\mathcal{G}^n_i$ with $n \geq 1$, we have
          $$
         \frac{1}{2}K(\Psi|_{\Omega_F})\leq \delta^{n-1}\min\{d(\mathcal{P}, \mathcal{P}'), C\}.
         $$
         
     \end{enumerate}
\end{cor}

{We now show Teichm\"uller mappings exist. We first show there exist quasiconformal maps that satisfy the second condition of Definition~\ref{defn:tm}
\begin{lem}\label{lem:quasiconformalconjugacyinfinitegen}
    Let $\mathcal{P}, \mathcal{P}'$ be two infinite circle packings with nerve $\mathcal{G}_i$.
    Then there exists a quasiconformal homeomorphism $\Psi$ between $\mathcal{P}$ and $\mathcal{P}'$ so that
    \begin{enumerate}
         \item $\frac{1}{2}\log(K(\Psi)) = d(\mathcal{P}, \mathcal{P}')$, and
        \item $\Psi|_{\Lambda(F)}: \Lambda(F) \longrightarrow \Lambda'(F)$ conjugates the dynamics of $\Gamma(F)$ and $\Gamma'(F)$ for each non-external face $F$ of $\mathcal{G}^n_i$ with $n \geq 0$.
    \end{enumerate}
\end{lem}
\begin{proof}
    Let $\Gamma^n, (\Gamma')^n$ be the reflection group generated by circles in $\mathcal{P}^n$ and $(\mathcal{P}')^n$ respectively.
    Then there exists a unique Teichm\"uller mapping $\psi_n: \Gamma^n\backslash\Omega^n \longrightarrow (\Gamma')^n\backslash(\Omega')^n$ that lifts to a quasiconformal homeomorphism $\Psi^n:\widehat\C \longrightarrow \widehat\C$ between the two finite circle packings $\mathcal{P}^n$ and $(\mathcal{P}')^n$.
    By Proposition \ref{prop:isom}, we have
    $$
    \frac{1}{2}\log(K(\Psi^n)) = d(\sigma, \sigma') = d(\mathcal{P}, \mathcal{P}'),
    $$
    where $\sigma, \sigma'$ are the external classes for $\mathcal{P}$ and $\mathcal{P}'$.
    By compactness of $K$-quasiconformal maps, let $\Psi:= \lim_{n_k} \Psi_{n_k}$ for some convergent subsequence, which is a quasiconformal map satisfying the required conditions.
\end{proof}
}

{
\begin{rmk}\label{rmk:nt}
    We remark that the map $\Psi$ constructed in the proof of Lemma~\ref{lem:quasiconformalconjugacyinfinitegen} is never a Teichm\"uller mapping in the non-trivial cases.
    In fact, this map $\Psi_\infty$ gives a quasiconformal conjugacy between the infinitely generated reflection groups $\Gamma^\infty, (\Gamma')^\infty$, which is too restrictive to be a Teichm\"uller mapping: the quasiconformal map between the external classes is reflected into circles of every level, and so the restriction of $\Psi_\infty$ to $\Omega_F$ for any non-external face $F$ always has dilatation $e^{2d(\sigma,\sigma')}$. 
\end{rmk}
}

\begin{theorem}[Existence of Teichm\"uller mappings]\label{thm:tm}
    Let $\mathcal{P}, \mathcal{P}'$ be two infinite circle packings with nerve $\mathcal{G}_i$.
    Then there exists a Teichm\"uller mapping between $\mathcal{P}, \mathcal{P}'$.
\end{theorem}

\begin{proof}
    We construct the map by taking the limit of a sequence of quasiconformal homeomorphism between the infinite circle packings $\mathcal{P}$ and $\mathcal{P}'$.

    To start, let $\Psi_0$ be a quasiconformal homeomorphism between $\mathcal{P}$ and $\mathcal{P}'$ constructed in Lemma~\ref{lem:quasiconformalconjugacyinfinitegen}.
    Let $F$ be a non-external face of $\mathcal{G}_i^1$.
    Then by construction, it is easy to see that 
    $$
    \Psi_0|_{\Lambda_F}: \Lambda_F \longrightarrow \Lambda'_F
    $$ 
    conjugates the dynamics of $\Gamma_F$ and $\Gamma'_F$.
    Let $\Psi_F$ be a quasiconformal homeomorphism between $\mathcal{P}_F$ and $\mathcal{P}'_F$ as constructed in Lemma~\ref{lem:quasiconformalconjugacyinfinitegen}.
    Then $\frac{1}{2}K(\Psi_F) = d(\mathcal{P}_F, \mathcal{P}'_F)$.
    Since $\Psi_F|_{\Lambda_F}: \Lambda_F \longrightarrow \Lambda'_F$ also conjugates the dynamics of $\Gamma_F$ and $\Gamma'_F$, $\Psi_F|_{\Lambda_F} = \Psi_0|_{\Lambda_F}$.
    
    We define $\Psi_1$ by replacing $\Psi_0|_{\Omega_F}$ with $\Psi_F|_{\Omega_F}$ for each non-external face $F$ of $\mathcal{G}^1_n$.
    Since $\Lambda_F$ is a quasicircle, it is quasiconformally removable.
    Therefore $\Psi_1$ is quasiconformal, and satisfies
    $$
    \frac{1}{2}\log(K(\Psi|_{\Omega_F})) = d(\mathcal{P}_{F}, \mathcal{P}'_{F})
    $$
    for each non-external face $F$ of $\mathcal{G}^1_i$.

    Inductively, we construct quasiconformal homeomorphism $\Psi_n$ between $\mathcal{P}$ and $\mathcal{P}'$ so that 
    $$
    \frac{1}{2}\log(K(\Psi|_{\Omega_F})) = d(\mathcal{P}_{F}, \mathcal{P}'_{F})
    $$
    for each non-external face $F$ of $\mathcal{G}^j_i$ with $j \leq n$.
    {By compactness of $K$-quasiconformal maps,} we can take a sub-sequential limit and conclude that a Teichm\"uller mapping exists.
\end{proof}

\begin{rmk}
It is easy to see that Teichm\"uller mappings between $\mathcal{P}, \mathcal{P}'$ are not unique, as we can introduce some small perturbations in the interiors of the disks in the circle packing.
\end{rmk}

\section{Universality and regularity of local symmetries}\label{sec:url}
In this section, we apply the results of exponential convergence for circle packings to prove the universality and regularity of local symmetries.

{This section is organized as follows. In \S~\ref{subsec:asympconf}, we prove that any homeomorphism between two infinite circle packings is $C^{1+\alpha}$-conformal (Theorem~\ref{thm:ac}). We prove a stronger form of asymptotic conformality for combinatorially deep points in \S~\ref{subsec:cdp}. We study renormalization periodic points in \S~\ref{subsec:renormalizationPeriodicPoints} and finish the proof of Theorem~\ref{thm:LS}. Finally, we give an explicit construction of infinite circle packings which are periodic under the renormalization in \S~\ref{subsec:explicitcon}.}

\subsection{Asymptotic conformality}\label{subsec:asympconf}
In this subsection, we fix $\mathcal{R}=\{P_i\}_{i=1}^k$ to be a simple irreducible acylindrical finite subdivision rule.
The main objective for this subsection is the following theorem.
\begin{theorem}\label{thm:ac}
    Let $\mathcal{P}$ and $\mathcal{P}'$ be two circle packings with nerve $\mathcal{G}_i$, and $\Psi$ a homeomorphism between them.
    Then $\Psi|_{\Lambda(\mathcal{P}) \cap \overline{\Pi(\mathcal{P})}}$ is $C^{1+\alpha}$-conformal.
\end{theorem}



{We first normalize by M\"obius transformations so that $\overline{\Pi(\mathcal{P})}$ and $\overline{\Pi(\mathcal{P}')}$ are bounded subsets of $\C$.}
Recall definition of $C^{1+\alpha}$-conformality in \S~\ref{subsec:univeralregular}. 
Let $x' = \Psi(x)$ be the corresponding points in $\Lambda(\mathcal{P}')$.

The proof of asymptotic conformality of $\Psi$ at $x$ consists of two steps.
{We normalize by translation so that $x = x' = 0$. Note that this normalization by translations does not change the exponent $\alpha$ nor the constant in $O(|t|^{1+\alpha})$ in Equation~\eqref{eqn:asymp} for the definition of $C^{1+\alpha}$-conformality.}
\begin{enumerate}
    \item We will first construct a sequence of pinched neighborhoods $\Delta^n$ and $(\Delta')^n$ of $0$ so that $\Lambda(\mathcal{P}) \cap \overline{\Delta^n}$ and $\Lambda(\mathcal{P}') \cap \overline{(\Delta')^n}$ are quasiconformally homeomorphic with dilatation converges to $1$ exponentially fast.
    This allows us to approximate the homeomorphism by M\"obius transformation with error $O(\delta^n|z|\log(\frac{1}{|z|}))$.
    \item We show that these pinched neighborhoods do not shrink to $0$ too fast so that $O(\delta^n|z|\log(\frac{1}{|z|})) = O(|z|^{1+\alpha})$ on $(\Delta^n - \overline{\Delta^{n+1}}) \cap \Lambda(\mathcal{P})$.
\end{enumerate}
This allows us to prove $\Psi(z) = \Psi'(0)z+O(|z|^{1+\alpha})$ on the limit set.

\subsection*{A sequence of pinched neighborhoods $\Delta^n$}
Let $F^n$ be a nested sequence of faces in $P_i$ associated to $x\in \Lambda(\mathcal{P}) \cap \overline{\Pi(\mathcal{P})}$. {We normalize by some translation so that $x = 0$.}
We remark that there are potentially $2$ different nested sequences associated to $0$.
This happens when $0$ is the tangent point of two circles in $\mathcal{P}$ and is in the interior $\Pi(\mathcal{P})$ (see \S~\ref{ss:renom}).

Let $D^n$ be a finite union of faces of $\mathcal{G}^n_i$ that share common boundary edges with $F^n$, and $A^n = D^n - \Int(F^n)$.
We denote the faces appeared in $A^n$ as $F^n(1),..., F^n(k_n)$.
Intuitively, one should think of $D^n$ is a combinatorial neighborhood of $F^n$ and provides some buffer for $F^n$ from the rest of the faces of $\mathcal{G}^n_i$ {(see Figure~\ref{fig:pinched_nbhd})}.
    


Let $\mathcal{P}$ be a circle packing with nerve $\mathcal{G}_i$, and let $\mathcal{P}^n$ be the finite sub-circle packing associated to $\mathcal{G}^n_i$.
Let $\Pi^n = \Pi_{F^n}$ be the interstice of $\mathcal{P}^n$ associated to the face $F^n$.
Let $\Gamma^n$ be the reflection group generated by circles in $\mathcal{P}$ associated to $\partial F^n$, and let $\Omega^n$ be the component of the domain of discontinuity of $\Gamma^n$ that contains $\Pi^n$.

We define $\Pi^n(j)$ and $\Omega^n(j)$ similarly for the face $F^n(j)$.
We define
$$\Delta^n = \Omega^n \cup \bigcup_j\Omega^n(j)$$
and $\overline{\Delta^n}$ is its closure.
We remark that similar as $D^n$ gives combinatorial protection to $F^n$, $\Delta^n$ gives some buffer zone for $\Pi^n$ from the other interstices of the circle packing $\mathcal{P}^n$ {(see Figure~\ref{fig:pinched_nbhd})}.

Note that by construction, it is clear that $\Pi^n \subseteq \Omega^n \subseteq \Delta^n$ and
$$
...\subseteq \Delta^n \subseteq \Delta^{n-1} \subseteq... \subseteq \Delta^0.
$$

\subsection*{Controlling the sizes of $\Delta^n$}
Let $F$ be a non-external face of $\mathcal{G}^n_i$, and $\Pi_F$ be the corresponding interstice in $\mathcal{P}^n$.
The boundary consists of finitely many circular arcs
$$
\partial \Pi_F = \bigcup_{v \text{ vertex in } \partial F} I_v.
$$

\begin{lem}
    Let $\mathcal{P}$ be a circle packing with nerve $\mathcal{G}_i$ normalized so that $\Lambda(\mathcal{P})$ is a bounded set in $\C$.
    There exists a constant $\alpha>0$ so that for any non-external face $F$ of $\mathcal{G}^n_i$,
    every circular arc $I_v$ on the boundary of the corresponding interstice $\Pi_F$ has length $|I_v| \geq \alpha^n$.
\end{lem}
\begin{proof}
    Let $F^n$ is a nested sequence of faces.
    Note that the subdivision rule $\mathcal{R}$ is acylindrical, so by the Bounded Image Theorem (Theorem \ref{thm:bitr}), the circle packings associated to $\partial F^n$ for each $n\geq 1$ belong to a fixed collection of compact sets of the corresponding Teichm\"uller spaces. {Thus the interstices $\Pi^n = \Pi_{F^n}$ are ideal hyperbolic polygons with uniformly bounded extremal width between any pairs of non-adjacent arcs of $\Pi^n$. 
    This means either all arcs of $\partial \Pi^n$ have comparable length or there is a pair of adjacent arcs $I_a, I_b$ of $\partial \Pi^n$ with much bigger length than the rest of the arcs (see Figure~\ref{fig:parabolicnorm}). In the first case, we use an affine map $\Phi$ to normalize all circles associated to $\partial \Pi^n$ of comparable and definite size. In the second case, let $\{p\}=I_a \cap I_b$ and we first use a parabolic M\"obius map $\Phi$ with fixed point $p$ to normalize so that all circles associated to $\partial \Pi^n$ of comparable size, and then apply some affine map to make them of comparable and definite size (see Figure~\ref{fig:parabolicnorm}).
    By Theorem \ref{thm:bitr}, the finite sub-circle packing $\mathcal{P}_{\mathcal{G}_{F^n}}$ for $\mathcal{G}_{F^n}:=\mathcal{G}^{n+1}_i \cap F^n$ is contained in a bounded set up to M\"obius maps.
    Thus after the normalization, $\mathcal{P}_{\mathcal{G}_{F^n}}$ is contained in a bounded set. In either case of the normalization, one sees
    $$
    \min_{v \in \partial F^{n+1}} |I_v| / \min_{w \in \partial F^{n}} |I_v| \asymp \min_{v \in \partial F^{n+1}} |\Phi(I_v)| / \min_{w \in \partial F^{n}} |\Phi(I_v)|.
    $$
    Therefore, the ratio $\min_{v \in \partial F^{n+1}} |I_v| / \min_{w \in \partial F^{n}} |I_w|$ is in a bounded set of $(0, \infty)$.}
    The lemma now follows.
\end{proof}
    \begin{figure}[htp]
    \centering
    \includegraphics[width=0.5\textwidth]{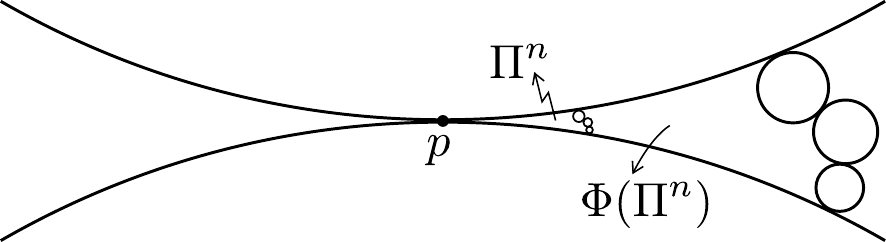}
    \caption{An illustration of the parabolic normalization.}
    \label{fig:parabolicnorm}
    \end{figure}

The previous lemma allows us to control the size of the neighborhood $\Delta^n$.
\begin{lem}\label{lem:ep}
    Let $\mathcal{P}$ be a circle packing with nerve $\mathcal{G}_i$ normalized so that $\Lambda(\mathcal{P})$ is a bounded set in $\C$..
    There exists a constant $\epsilon>0$ so that if $x\in \Lambda(\mathcal{P}) - \overline{\Delta^n}$, then $|x| \geq \epsilon^{n}$.
    Equivalently, we have
    $$
    B(0,\epsilon^{n}) \cap \Lambda(\mathcal{P}) \subseteq \overline{\Delta^n}.
    $$
\end{lem}
\begin{proof}
    By the previous lemma, in the pinched neighborhood $\Delta^n$ (see Figure~\ref{fig:pinched_nbhd}), all circular arcs in the shaded region has length $\ge\alpha^n$. Any point $x\in\Lambda(\mathcal{P})-\overline{\Delta^n}$, lies outside the shaded region, as well as all the disks forming the interstices in the region. Since $0$ is in the center component (the darker shaded region), by choosing $\epsilon$ small enough (and smaller than $\alpha$), we can make sure that the distance from $x$ to $0$ is $\ge\epsilon^n$, as desired.
\end{proof}

\begin{figure}[htp]
    \centering
    \includegraphics[width=0.45\textwidth]{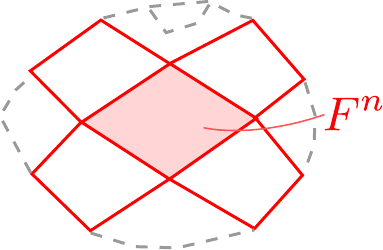}
    \includegraphics[width=0.4\textwidth]{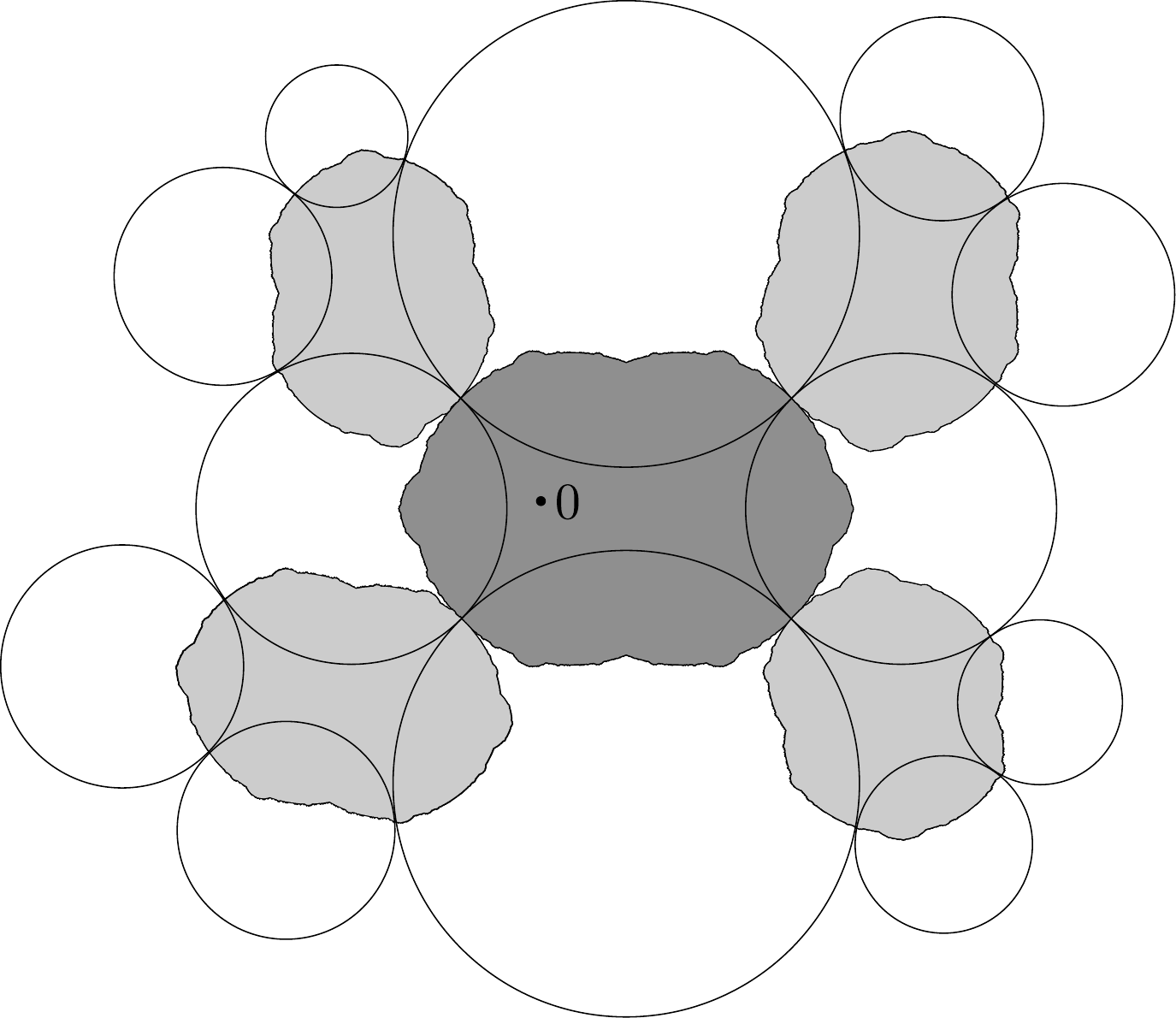}
    \caption{A pinched neighborhood in the graph with the corresponding pinched neighborhood of $0$.}
    \label{fig:pinched_nbhd}
\end{figure}

\subsection*{Local quasiconformal homeomorphism}
Let $\mathcal{P}'$ be another circle packing with nerve $\mathcal{G}_i$.
By Proposition \ref{prop:adf}, there exists a quasiconformal map $f_n: \widehat \C \longrightarrow \widehat \C$ with 
$$
f_n(\Lambda(\mathcal{P}) \cap \overline{\Omega^n}) = \Lambda(\mathcal{P}') \cap \overline{(\Omega')^n}
$$
so that $\frac{1}{2}\log K(f_n) = O(\delta^n)$.
Since $D^n$ is a finite union of adjacent faces of $\mathcal{G}^n_i$, the argument {in \S~\ref{ss:renom}} applies to $\Delta^n$ and we have
\begin{lem}
    Let $\mathcal{P}, \mathcal{P}'$ be two circle packings with nerve $\mathcal{G}_i$.
    There exists a quasiconformal map $f_n: \widehat \C \longrightarrow \widehat \C$ with 
$$
f_n(\Lambda(\mathcal{P}) \cap \overline{\Delta^n}) = \Lambda(\mathcal{P}') \cap \overline{(\Delta')^n}
$$
so that $\frac{1}{2}\log K(f_n) = O(\delta^n)$.
\end{lem}

We remark that by definition, these quasiconformal homeomorphisms respect the marking, i.e., $f_n$ sends a circle $C_v$ of $\mathcal{P}$ to the corresponding circle $C'_v$ of $\mathcal{P}'$.
Thus, we have the following lemma.
\begin{lem}\label{lem:ag}
    Suppose that $z \in \Lambda(\mathcal{P}) \cap \overline{\Delta^n}$.
    Then $f_n(z) = \Psi(z)$.
\end{lem}

\subsection*{M\"obius approximation}
Note that we have $f_n(0) = 0$.
Let $q_n = f_n^{-1}(1)$ and $p_n= f_n^{-1}(\infty)$.
We define an approximating M\"obius transformation $M_n \in \PSL_2(\C)$ so that
$$
M_n(0) = 0, M_n(q_n) = f_n(q_n) = 1, \text{ and } M_n(p_n) = f_n(p_n) = \infty.
$$
Denote $a_n = M_n'(0)$, and $X:= \widetilde \C - \{0, 1, \infty\}$.
Let $\epsilon$ be the constant in Lemma \ref{lem:ep}.
By shrinking $\epsilon$ if necessary, we assume that $\epsilon \leq \delta$.

\begin{lem}\label{lem:R}
    We have $p_n^{-1}, q_n^{-1}, a_n, a_n^{-1} = O(\epsilon^{-n})$.
    Thus, for all $|z| < \epsilon^{3n}$,
    $$
    M_n(z) = a_n z + O(\epsilon^{2n} |a_nz|) = a_n z + O(\epsilon^{n} |z|).
    $$ 
\end{lem}
\begin{proof}
    Let $n \geq m$, and consider the map $f_{n}^{-1}\circ f_m$.
    By Lemma \ref{lem:ep}, we can choose points $\alpha_{n}, \beta_{n} \in \Lambda(\mathcal{P}) \cap \overline{\Delta^{n}}$ with $|\alpha_n| = \frac{\epsilon^n}{2}$ and $|\beta_n| = \epsilon^n$.
    By Lemma \ref{lem:ag}, we have that $f_{n}^{-1}\circ f_m$ fixes $0, \alpha_{n}, \beta_{n}$.
    Consider the cross-ratio map $R(z)=( \alpha_n, \beta_n; 0, z) = \frac{\alpha_n}{\beta_n} \cdot \frac{z-\beta_n}{z-\alpha_n}$.  
    By applying Lemma \ref{lem:qc} to $R\circ f_{n}^{-1}\circ f_m$ and $R$, we get
    $$
     d_{X}(\frac{\alpha_n}{\beta_n} \cdot \frac{p_m-\beta_n}{p_m-\alpha_n}, \frac{\alpha_n}{\beta_n} \cdot \frac{p_n-\beta_n}{p_n-\alpha_n}) = d_{X}(R(p_m), R\circ f_{n}^{-1}\circ f_m(p_m)) = O(\delta^m).
    $$
    Note that $\frac{p_m-\beta_n}{p_m-\alpha_n} \to 1$ as $n\to\infty$, and $\frac{|\alpha_n|}{|\beta_n|} = \frac{1}{2}$.
    By fixing $m$ large enough, we conclude that $\frac{\alpha_n}{\beta_n} \cdot \frac{p_n-\beta_n}{p_n-\alpha_n}$ is close to $\frac{\alpha_n}{\beta_n}$ for all large $n$.
    Thus, $\frac{p_n-\beta_n}{p_n-\alpha_n}$ is close to $1$ for all sufficiently large $n$.
    Therefore, $|p_n| \geq \epsilon^n$ for all sufficiently large $n$.
    Similarly, we have $|q_n| \geq \epsilon^n$ for all sufficiently large $n$.

    A simple computation shows that $a_n = \frac{1}{q_n}-\frac{1}{p_n}$. Thus, $a_n = O(\epsilon^{-n})$.
    Note that the map $f_n^{-1}\circ f_m$ sends $0, q_m, \alpha_n, p_m$ to $0, q_n, \alpha_n, p_n$.
    Thus, 
    $$
    d_X((0, q_m; \alpha_n, p_m), (0, q_n; \alpha_n, p_n)) = d_X(\frac{-\alpha_n}{q_m-\alpha_n}\cdot \frac{p_m-q_m}{p_m}, \frac{-\alpha_n}{q_n-\alpha_n}\cdot \frac{p_n-q_n}{p_n}) =  O(\delta^m).
    $$
    We claim that $\frac{|p_n-q_n|}{|p_nq_n|} \geq A\epsilon^n$ for some $A>0$ and all large $n$.
    Otherwise, after passing to a subsequence, we have $|\frac{-\alpha_n}{q_m-\alpha_n}\cdot \frac{p_m-q_m}{p_m}|$ behaves like $\epsilon^n$, and $|\frac{-\alpha_n}{q_n-\alpha_n}\cdot \frac{p_n-q_n}{p_n}| = O(\epsilon^{2n})$.
    Since the hyperbolic metric on $X$ near $0$ is bounded below by $\frac{|dx|}{|x||\log|Bx||}$ for some $B> 0$ (see e.g.~\cite[Lemma~2.5]{Hem79}), we have $d_X(\epsilon^n, \epsilon^{2n})$ is bounded below by a constant.
    This is a contradiction.
    Thus, we have $a_n^{-1} = \frac{p_nq_n}{p_n-q_n} = O(\epsilon^{-n})$.
    
    Note $M_n(z) = \frac{a_n z}{1-\frac{z}{p_n}}$.
    Thus, if $|z|< \epsilon^{3n}$, we have 
    $$
    |M_n(z)-a_n z| = \frac{|\frac{a_n z^2}{p_n}|}{|1-\frac{z}{p_n}|} = |a_nz| \frac{|\frac{z}{p_n}|}{|1-\frac{z}{p_n}|} = O(\epsilon^{2n}\cdot |a_nz|).
    $$
    Since $|a_n| = O(\epsilon^{-n})$, $\epsilon^{2n} |a_nz| = O(\epsilon^{n} |z|)$ and the lemma follows.
\end{proof}

\begin{lem}\label{lem:est}
    For all $\epsilon^{3n+3} \leq |z|\leq \epsilon^{3n}$, we have
    $$
    f_n(z) = M_n(z) + O(\delta^{n/2}|z|) = a_n z + O(\delta^{n/2}|z|).
    $$
\end{lem}
\begin{proof}
    We first claim that $a_n = O(1)$.
    \begin{proof}[Proof of the claim]
    Let $z_n \in \Lambda(\mathcal{P})$ with $|z_n|= \epsilon^{3n}$.
    Then by Lemma \ref{lem:ag}, $f_{n-1}(z_n) = f_n(z_n)$.
    By Lemma \ref{lem:R}, $M_n(z_n) = a_n z_n + O(\epsilon^{2n} |a_nz_n|)$, so 
    $$
    d_X(M_n(z_n), a_nz_n) = O(\epsilon^{2n}) = O(\delta^n).
    $$
    Similarly, we have that $d_X(M_{n-1}(z_n), a_{n-1}z_n) = O(\delta^n)$.
    By Lemma \ref{lem:qc},
    $d_X(f_n(z_n), M_n(z_n)) = O(\delta^n)$.
    Similarly, we have $d_X(f_{n-1}(z_n), M_{n-1}(z_n)) = O(\delta^n)$.
    Hence, we have 
    $$
    d_X(a_nz_n, a_{n-1}z_n) = O(\delta^n).
    $$

    Note that near $0$, the hyperbolic metric on $X$ is bounded below by $\frac{|dx|}{|x||\log|Ax||}$ for some $A> 0$.
    By Lemma \ref{lem:R}, $|\log|a_nz_n|| = O(n)$ and $|\log|a_{n-1}z_n|| = O(n)$, so the hyperbolic metric between the interval bounded by $|a_nz_n|$ and $|a_{n-1}z_n|$ is bounded below by $\frac{|dx|}{B n|x|}$ for some $B>0$.
    Thus
    $$
    O(\delta^n) = d_X(a_nz_n, a_{n-1}z_n) \geq d_X(|a_nz_n|, |a_{n-1}z_n|) \geq \frac{|\log\frac{|a_n|}{|a_{n-1}|}|}{Bn}.
    $$
    Hence, $|\log\frac{|a_n|}{|a_{n-1}|}| = O(n\delta^n) = O(\delta^{n/2})$.
    Therefore, $a_n = O(1)$.
    \end{proof}   
    
    By Lemma \ref{lem:R}, $|M_n(z)| \sim |a_nz|$, so $|\log |M_n(z)|| = O(n)$ for all $\epsilon^{3n+3} \leq |z|\leq \epsilon^{3n}$.
    Thus the hyperbolic metric on the interval between $|M_n(z)|/2$ and $2|M_n(z)|$ is bounded below by $\frac{|dx|}{B n|x|}$ for some $B>0$.

    Let $\gamma(t), t\in I$ be the hyperbolic geodesic connecting $M_n(z)$ and $f_n(z)$.
    We claim that $1/2\leq |\gamma(t)|/|M_n(z)|\leq 2$ for all large $n$ and $t\in I$.
    Without loss of generality, suppose for contradiction that $|\gamma(t)| \geq 2 |M_n(z)|$.
    Then 
    $$
    d_X(\gamma(t), M_n(z)) \geq d_X(|\gamma(t)|, |M_n(z)|) \geq d_X(2|M_n(z)|, |M_n(z)|) \geq \frac{\log 2}{Bn}.
    $$
    Thus, $\log 2 = O(n \delta^n)$ which is a contradiction.

    Let $|\gamma|$ denote the Euclidean length of $\gamma$.
    Then $|\gamma| \geq |f_n(z)-M_n(z)|$.
    Let $\rho_X(x)|dx|$ be the hyperbolic metric on $X$.
    Note that $\rho_X(x)|dx| \geq \frac{|dx|}{2Bn|M_n(z)|}$ for all $x\in \gamma$.
    Therefore, we have
    $$
    d_X(f_n(z), M_n(z)) = \int_{\gamma} \rho_X(x)|dx| \geq |\gamma| \cdot \frac{1}{2Bn|M_n(z)|} \geq \frac{|f_n(z) - M_n(z)|}{2Bn|M_n(z)|}.
    $$
    By Lemma \ref{lem:qc}, $d_X(f_n(z), M_n(z)) = O(\delta^n)$.
    Thus, 
    $$|f_n(z) - M_n(z)| = O(n\delta^n|M_n(z)|) = O(\delta^{n/2}|z|),$$
    where the last equality holds as $a_n = O(1)$.
    By Lemma \ref{lem:R}, $M_n(z) = a_nz +O(\epsilon^{n} |z|)$.
    Since $\epsilon \leq \delta$, we have $f_n(z) = a_n z + O(\delta^{n/2}|z|)$.
\end{proof}

\begin{lem}\label{lem:der}
    The derivative $a_n = M_n'(0)$ satisfies $|a_n - a_{n-1}| = O(\delta^{n/2})$.
    In particular, the limit $a = \lim_n a_n$ exists and $|a - a_n| = O(\delta^{n/2})$.
\end{lem}
\begin{proof}
    Choose a sequence $z_n \in \Lambda(\mathcal{P})$ so that $|z_n| = \epsilon^{3n}$.
    By Lemma \ref{lem:ep}, we have $z_n \in \overline{\Delta^n}$.
    Therefore, by Lemma \ref{lem:ag}, $\Psi(z_n)=f_n(z_n) = f_{n-1}(z_n)$.
    Hence, by Lemma \ref{lem:est},
    $$
    0 = \frac{f_n(z_n)-f_{n-1}(z_n)}{z_n} = (a_n-a_{n-1}) + O(\delta^{n/2}).
    $$
    Therefore, $|a_n - a_{n-1}|  = O(\delta^{n/2})$.
\end{proof}

\subsection*{Proof of asymptotic conformality}
We are ready to prove the regularity of local symmetries.
\begin{proof}[Proof of Theorem \ref{thm:ac}]
    Let $z \in \Lambda(\mathcal{P})$ with $\epsilon^{3n+3} \leq |z|\leq \epsilon^{3n}$.
    By Lemma \ref{lem:ep}, $z\in \overline{\Delta^{3n}} \subseteq \overline{\Delta^{n}}$.

    Let $\alpha = \frac{\log \delta}{6\log \epsilon} > 0$.
    Then $\delta^{n/2} = \epsilon^{3n\alpha} = O(|z|^\alpha)$.
    By Lemma \ref{lem:ag}, $\Psi(z) = f_n(z)$.
    By Lemma \ref{lem:R}, Lemma \ref{lem:est} and Lemma \ref{lem:der}, we have
    \begin{align*}
        \Psi(z) = f_n(z) &= M_n(z) + O(\delta^{n/2}|z|)\\
        &= a_nz + O(\delta^{n/2}|z|)\\
        &= a z + O(\delta^{n/2}|z|)\\
        &= a z+O(|z|^{1+\alpha}).
    \end{align*}
    This holds for any $n$, we conclude that $\Psi|_{\Lambda(\mathcal{P})}$ is $C^{1+\alpha}$-conformal at $0$.
\end{proof}

\subsection{Combinatorially deep points}\label{subsec:cdp}
Let $\mathcal{P}$ be a circle packing with nerve $\mathcal{G}_i$.
Let $x\in \Lambda(\mathcal{P}) \cap \overline{\Pi(\mathcal{P})}$, and let $F^n$ be a nested sequence of faces associated to $x$.
We say $x$ is a {\em combinatorially deep point} if for each $n$, $F^{n+1} \cap \partial F^n$ is either empty or a single vertex.
We remark that for combinatorially deep points, the corresponding nested sequence is unique.
For combinatorially deep points, we can improve the result of asymptotic conformality.
\begin{theorem}\label{thm:tmac}
     Let $\mathcal{P}$ and $\mathcal{P}'$ be two circle packings with nerve $\mathcal{G}_i$, and $\Psi$ be a Teichm\"uller mapping between them.
    Then $\Psi$ is $C^{1+\alpha}$-conformal at all combinatorially deep points.
\end{theorem}
\begin{rmk}\label{rmk:cm}
    We remark that the assertion of Theorem \ref{thm:tmac} is much stronger than the restriction $\Psi_{\Lambda(\mathcal{P}) \cap \overline{\Pi(\mathcal{P})}}$ is $C^{1+\alpha}$-conformal.
    It asserts that for a combinatorially deep point $x$,
    $$
    \Psi(x+t) = \Psi(x)+\Psi'(x)t + O(|t|^{1+\alpha})
    $$
    for all small $t$, while Theorem \ref{thm:ac} only asserts that for small $t$ with $x+t \in \Lambda(\mathcal{P}) \cap \overline{\Pi(\mathcal{P})}$.
    We remark that in order to control the behavior of points outside the limit set, it is crucial that $\Psi$ is a Teichm\"uller mapping.
\end{rmk}

\subsection*{Nested annuli with bounded modulus}
Let $\mathcal{P}$ be a circle packing with nerve $\mathcal{G}_i$.
Let $x$ be a combinatorially deep point, with nested sequence of faces $F^n$.
In contrast to the pinched neighborhoods that we construct for general points, we can construct {nested neighborhoods} of $x$ that are simply connected with good geometric control.

Let $\Omega^n= \Omega_{F^n}$ be the component of domain of discontinuity associated to $F^n$.
Since $F^n$ and $F^{n+1}$ {do} not share a common edge, it is easy to see that $\Omega^{n+1}$ is compactly contained in $\Omega^n$.
Let $A^n:= \Omega^n-\overline{\Omega^{n+1}}$ be the corresponding annulus.
By the Bounded Image Theorem (Theorem \ref{thm:bitr}), we have
\begin{lem}\label{lem:uni}
    The domain $\Omega^n$ is a uniform quasi-disk and the annulus $A^n$ has uniformly bounded modulus.
    More precisely, there exists a constant $K, M > 1$ so that for each $n \geq 1$,
    \begin{itemize}
        \item $\Omega^n$ is a $K$-quasi-disk, and 
        \item $\Mod(A_n) \in [1/M, M]$.
    \end{itemize}
    In particular, by compactness, there exists a constant $\epsilon>0$ so that
    $$
    B(0,\epsilon^{n})\subseteq \overline{\Omega^n}.
    $$
\end{lem}

\subsection*{Proof of asymptotic conformality at combinatorially deep points}
\begin{proof}[Proof of Theorem \ref{thm:tmac}]
Similar as before, by Proposition \ref{prop:adf}, there exists a quasiconformal map $f_n: \widehat \C \longrightarrow \widehat \C$ with 
$$
f_n(\Lambda(\mathcal{P}) \cap \overline{\Omega^n}) = \Lambda(\mathcal{P}') \cap \overline{\Omega^n}
$$
so that $\frac{1}{2}\log K(f_n) = O(\delta^n)$. 
Since $\Psi$ is a Teichm\"uller mapping, we may choose $f_n$ so that
$$
\Psi|_{\Omega^n} = f_n|_{\Omega^n}.
$$
By replacing the pinched neighborhoods $\Delta^n$ and Lemma \ref{lem:ep} in the proof of Theorem \ref{thm:ac} by $\Omega^n$ and Lemma \ref{lem:uni}, the same arguement proves that $\Psi$ is $C^{1+\alpha}$ at all deep points.
\end{proof}

\subsection{Renormalization periodic points}\label{subsec:renormalizationPeriodicPoints}
In this section, we study renormalization periodic points.
Without loss of generality and for simplicity of the presentation, let us assume the period is $1$.

First, let $(\mathcal{P}_0, x_0)$ be a renormalization fixed point.
Let $F^n$ be a corresponding nested sequence of faces.
We remark that the corresponding nested sequence of faces $F^n$ is unique (see Remark \ref{rmk:uniq} below).
To describe the combinatorics, we give the following definition.
\begin{defn}
    A nested sequence of faces $(F^n)$ is called a subdivision fixed point if
    \begin{enumerate}
        \item each $F^n$ is identified with $F^0$, by a cellular homeomorphism 
        $$
        \psi_n: F^0 \longrightarrow F^n;
        $$
        \item $\psi_n \circ \psi_1 = \psi_{n+1}$ for all $n$.
    \end{enumerate}
\end{defn}
\begin{prop}
    Renormalization fixed points are in one to one correspondence with subdivision fixed points.
\end{prop}
\begin{proof}
    Suppose $(\mathcal{P}_0, x_0)$ is a renormalization fixed point.
    Let $(F^n)$ be a corresponding nested sequence of faces.
    Then it is clear that $(F^n)$ is a subdivision fixed point.

    Conversely, let $(F^n)$ be a subdivision fixed point.
    By Proposition \ref{prop:adf}, we can find the unique renormalization fixed point by iterating the renormalization operator $\mathfrak{R}$ given by the combinatorial data $(F^n)$.
\end{proof}

\begin{rmk}\label{rmk:uniq}
    Let $(F^n)$ be a subdivision fixed point.
    If $F^1$ does not share an edge with $F^0$. Then $(F^n)$ gives a combinatorially deep point, and $x_0$ is not the tangent point of two circles of $\mathcal{P}_0$.
    If $F^1$ shares an edge with $F^0$. Then $x_0$ is the tangent point of two circles of $\mathcal{P}_0$ and is on the boundary $\partial \Pi(\mathcal{P}_0)$.
    Thus, in either case, the preimage $\pi^{-1}((\mathcal{P}_0, x_0)) \subseteq \widetilde\Sigma^\infty$ is a singleton.
\end{rmk}

We are ready to assemble the proof of the universality and regularity theorem.
\begin{proof}[Proof of Theorem \ref{thm:LS}]
    The statement (1) follows from Theorem \ref{thm:ac}, and the statement (2) follows from the above discussion.
    
    By Remark \ref{rmk:uniq}, $x_0$ is either a combinatorially deep point, or a cusp point on the boundary $\partial \Pi(\mathcal{P}_0)$.
    In the first case, we have that the scaling factor $|\mu| > 1$ by Lemma \ref{lem:uni}.
    In the second case, the M\"obius transformation $\phi$ preserves the two circles in $\mathcal{P}_0$ that passes through $x_0$.
    Thus $\phi$ is a parabolic map with a fixed point at $x_0$.
    This proves (3).

    The last statement follows from the fact that the map $\Psi|_{\Lambda(\mathcal{P}_0) \cap \overline{\Pi(\mathcal{P}_0)}}$ is $C^{1+\alpha}$-conformal.
\end{proof}

\subsection{Explicit constructions}\label{subsec:explicitcon}
{In the following, we give a more descriptive construction of the renormalization fixed point, which is in the same spirit of the {\em tower construction} in \cite{McM96}.
The construction also relates the renormalization fixed point and infinite circle packings for spherical subdivision graphs.}

\subsection*{Enlargement of $\mathcal{R}$}
    Let $(F^n)$ be a subdivision fixed point.
    We may assume that $F^0 = P_i$.
    To construct the renormalization fixed point, we will first construct a new finite subdivision rule $\mathcal{R}'$ which enlarges $\mathcal{R}$.
    This new $\mathcal{R}'$ is constructed by adding a new polygon that is {\em inverted} copy of $F^1$.

    More precisely, we define
    $Q:= \overline{\widehat \C - F^1}$, and $Q':=  \overline{\widehat \C - P_i}$.
    Let us denote the subdivision $\mathcal{R}(P_i)$ as 
    $$
    P_i = F^1 \cup \bigcup_{j} P_{i,j}.
    $$
    Then $\bigcup_{j} P_{i,j} \cup Q'$ gives a subdivision of $Q$.
    Since $F^0 = P_i$ and $F^1$ are identified by $\psi_1$, it induces an identification between $Q$ and $Q'$.
    We remark that this also induces a natural orientation reversing identification of $\partial P_i$ with $\partial Q$.
    
    We extend $\mathcal{R}$ to $\mathcal{R}'$ so that it consists of $\{P_1,..., P_k, Q\}$, where the subdivision $\mathcal{R}'(P_i) = \mathcal{R}(P_i)$ and $\mathcal{R}'(Q)$ as constructed as above.
    It is easy to see that we can identify each face of the subdivision with a polygon in the list $\{P_1,..., P_k, Q\}$, and thus $\mathcal{R}'$ is a finite subdivision rule.
    Since $\mathcal{R}$ is simple, irreducible and acylindrical, it is easy to see that so is $\mathcal{R}'$.

\subsection*{$\mathcal{R}'$-complex and renormalization fixed point}
    Let $X$ be the $\mathcal{R}'$-complex by obtained gluing $P_i$ with $Q$ using the natural orientation reversing identification between $\partial P_i$ and $\partial Q$.
    Then $X$ is homeomorphic to a $2$-sphere.

    Let $\mathcal{G}^n$ be the 1-skeleton of $\mathcal{R}^n(X)$ and $\mathcal{G} = \bigcup_n \mathcal{G}^n$.
    Then it is easy to verify that $\mathcal{G}$ is a simple graph.
    By Theorem \ref{thm:A}, there exists a unique circle packing $\mathcal{P}$ with nerve $\mathcal{G}$.

    Let $\mathcal{G}_{F^n} = \mathcal{G} \cap F^n$, where $F^n$ is naturally identified as a subset of $P_i$.
    By construction, it is easy to see that there exists a homeomorphism $\tilde \phi$ of $\mathcal{G}$ extending the local homeomorphism $\mathcal{G}_{F^1} \longrightarrow \mathcal{G}_{F^0}$.
    Thus, by Theorem \ref{thm:A}, there exists a M\"obius transformation $\phi$ that fixes $\mathcal{P}$.

    Let $\mathcal{P}_0$ be the sub-circle packing of $\mathcal{P}$ associated to $\mathcal{G}_{F^0}$, and $x_0$ is the repelling fixed point of $\phi$.
    Then it is easy to see that $(\mathcal{P}_0, x_0)$ is the corresponding renormalization fixed point.

\section{Kleinian circle packings}\label{sec:kcp}

Recall that a circle packing $\mathcal{P}$ is called \emph{Kleinian} if its limit set $\Lambda(\mathcal{P})$ is the same as the limit set of some Kleinian group $\Gamma$. We emphasize that by convention $\mathcal{P}$ has connected nerve. Our main result is the existence of subdivision rules for geometrically finite Kleinian circle packings.
\begin{theorem}\label{thm:ksdr}
Let $\mathcal{P}$ be a Kleinian circle packing, with the associated Kleinian group $\Gamma$. Suppose $\Gamma$ is finitely generated. Then the nerve $\mathcal{G}$ of $\mathcal{P}$ is a spherical subdivision graph
\begin{itemize}
    \item for a finite subdivision rule if and only if $\Gamma$ is geometrically finite and contains no rank-2 parabolic subgroup;
    \item for a $\mathbb{Z}^2$-subdivision rule if and only if $\Gamma$ is geometrically finite and contains at least one rank-2 parabolic subgroup.
\end{itemize}
\end{theorem}
Here, a \emph{$\mathbb{Z}^2$-subdivision rule} is a generalization of the theory we have developed so far, modeled on the $\mathbb{Z}^2$-tiling under a rank-2 parabolic subgroup. Our theory is well-suited for this generalization, and we have an existence and uniqueness theorem in this setting (see Theorem~\ref{thm:Z2EU}). To avoid repetition and streamline our arguments, this generalization and others are included in Appendix~\ref{sec:z2sr}.

Combined with Theorem~\ref{thm:A} (without rank-2 cusps) or Theorem~\ref{thm:Z2EU} (with rank-2 cusps), we immediately obtain Theorem~\ref{thm:C}. Note that we need not check acylindricity of the subdivision rule, since a circle packing already exists.

{The remainder of the section is organized as follows. In \S\ref{subsec:g&t} we recall some basic 3-dimensional geometry and topology, and state Theorem~\ref{thm:GFAC1} that characterizes hyperbolic 3-manifolds with circle packing limit sets. In \S\ref{subsec:ksr} and \S\ref{subsec:coae}, we construct subdivision rules of such limit sets in two steps: we first reduce it to a topological problem in \S\ref{subsec:ksr}, and then solve the topological problem in our setting in \S\ref{subsec:coae}. We finally provide a topological point of view of our subdivision rule in \S\ref{subsec:atm}.}

\subsection{Background on topology and geometry of 3-manifolds}\label{subsec:g&t}
Recall that a Kleinian group $\Gamma$ (and the corresponding hyperbolic 3-manifold $M$) is called geometrically finite if there exists a finite-sided fundamental polyhedron for the action of $\Gamma$ on $\mathbb{H}^3$. Equivalently, the unit neighborhood of the \emph{convex core} of $M$, {defined by}
$$\core(M):=\Gamma\backslash\chull(\Lambda),$$
has finite volume, where $\chull(\Lambda)$ is the convex hull of the limit set $\Lambda$ of $\Gamma$ in $\mathbb{H}^3$.

We have discussed incompressibility and acylindricity in the context of kissing reflection groups, cf.\ \S\ref{subsec:td}. In general, let $(N,P)$ be a \emph{pared 3-manifold}, where $N$ is a compact oriented 3-manifold with boundary, and $P\subseteq\partial N$ is a submanifold consisting of incompressible tori and annuli. See \cite{hyperbolization1} for a precise definition in arbitrary dimension. Set $\partial_0N=N-P$. Then we say $(N,P)$ is \emph{acylindrical} if each component of $\partial_0N$ is incompressible, and $N$ contains no essential cylinder with both ends in $\partial_0N$.

For a geometrically finite $M$, let $\core_\epsilon(M)$ be the convex core of $M$ minus $\epsilon$-thin cuspidal neighborhoods for all cusps. Here $\epsilon$ is always chosen small enough, say smaller than the Margulis constant in dimension 3. Let $P\subseteq\partial\core_\epsilon(M)$ be the union of boundaries of all cuspidal neighborhoods, then $(\core_\epsilon(M),P)$ is a pared 3-manifold, and we say $M$ (and the corresponding Kleinian group $\Gamma$) is acylindrical if $(\core_\epsilon(M),P)$ is. Note that the connected components of $\partial\core_\epsilon(M)-P$ are in one-to-one correspondence with the boundary components of the Kleinian manifold $\overline{M}$.

A compact, orientable, irreducible manifold with boundary $(N,\partial N)$ is called a \emph{compression body} if it is homeomorphic to the boundary connected sum of a solid 3-ball with a collection of solid tori and a collection of trivial interval bundles over closed surfaces, so that the other summands are all attached to the 3-ball along disjoint disks. Either collection may be empty.

For one boundary component $\partial_eN$ of $N$, the inclusion induces a surjection on $\pi_1$; we refer to this component as the \emph{exterior} boundary. As a matter of fact, this characterizes compression bodies: $(N,\partial N)$ is a compression body if and only if the inclusion of one component of $\partial N$ induces a surjection on $\pi_1$ (see e.g.\ \cite[Lem.~4.1]{BJM13}). All other boundary components of $N$ are incompressible; we refer to them as \emph{interior} boundary components.

A result of Otal, called ``disk-busting" \cite[Rem.~1.5]{Ota88}, gives a criterion for $(N,P)$ to be a pared acylindrical manifold, where $P$ is the union of all toroidal boundary components and a collection of simple curves on $\partial_eN$.

We have the following topological characterization of geometrically finite Kleinian circle packings.
\begin{theorem}\label{thm:GFAC1}
Let $\mathcal{P}$ be a Kleinian circle packing with nerve $\mathcal{G}$. Let $\Gamma$ be the associated Kleinian group, and $M=\Gamma\backslash\mathbb{H}^3$ the corresponding hyperbolic 3-manifold. Suppose $\Gamma$ is geometrically finite. Then $M$ is acylindrical, and homeomorphic to the interior of a compression body with empty or only toroidal interior boundary components.
\end{theorem}
The proof of this theorem is contained in Appendix~\ref{appx:topology}.

\subsection{Kleinian subdivision rules}\label{subsec:ksr}
In this subsection and the next, we associate a finite or $\mathbb{Z}^2$-subdivision rule $\mathcal{R}$ to the nerve $\mathcal{G}$ of the circle packing limit set of a geometrically finite Kleinian group $\Gamma$. This gives the ``if" part of Theorem~\ref{thm:ksdr}. The ``only if" part is given by Proposition~\ref{prop:topology} in Appendix~\ref{appx:topology}.

The proof comes in two steps: in this section we first reduce the question to a topological problem, and then in the next subsection \S\ref{subsec:coae} we show that this topological problem can always be solved in our setting.

By Theorem~\ref{thm:GFAC1}, the hyperbolic 3-manifold $M=\Gamma\backslash\mathbb{H}^3$ is homeomorphic to the interior of a compact compression body $N$ whose interior boundary is either empty or consists of tori. Moreover, a collection of disjoint simple closed curves $\{\gamma_i\}$ on $\partial_e N$ gives the parabolic locus $P=\sqcup P_i$, where each $P_i\subseteq\partial N$ is an annulus with core curve $\gamma_i$ (see Figure~\ref{fig:handlebody} for an example). {Moreover, since $M$ is acylindrical, $\partial _eN-\cup_i\gamma_i$ consists of incompressible surfaces, and $N$ contains no essential cylinders with both ends in $\partial _eN-\cup_i\gamma_i$.}

\begin{figure}[htp]
    \centering
    \includegraphics[width=0.5\linewidth]{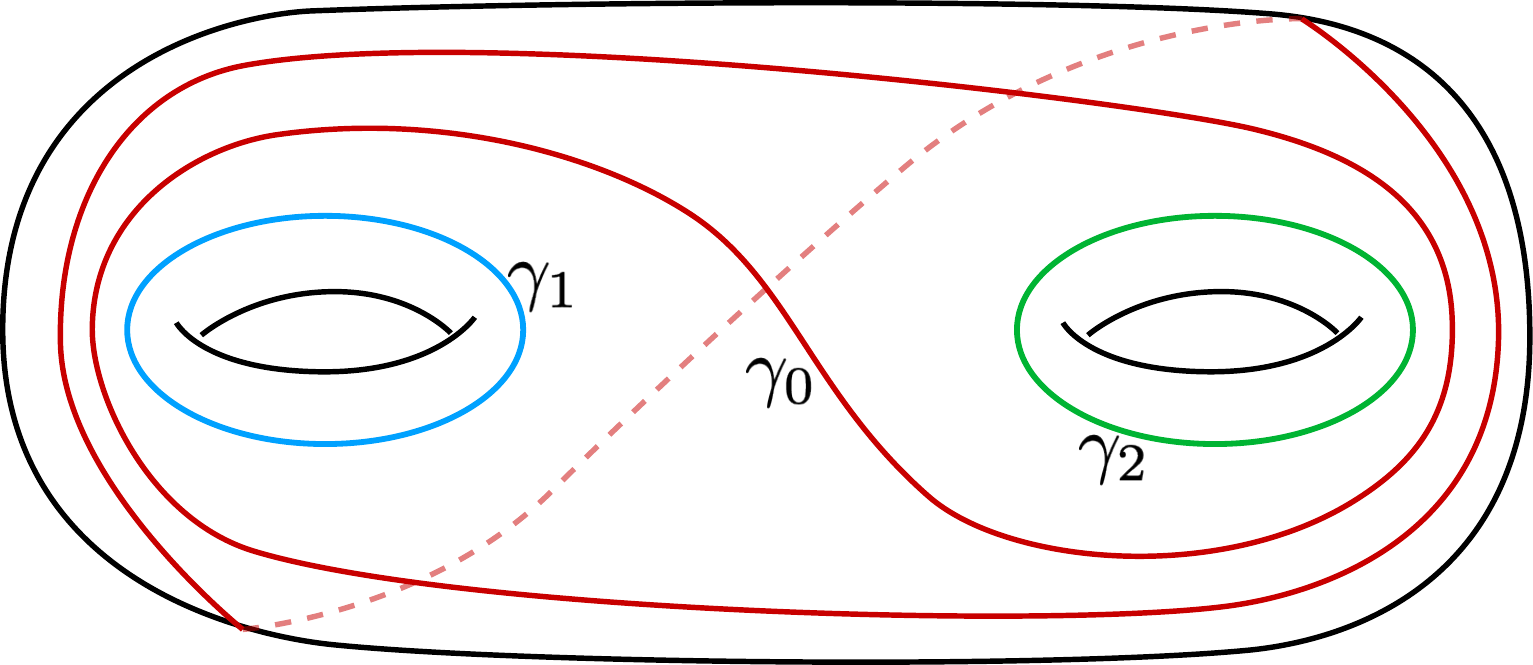}
    \caption{A Kleinian manifold whose limit set is the Apollonian gasket, homeomorphic to a genus two handlebody with parabolic locus marked. {See \cite[Appx.~A]{Zha22} for a detailed discussion of this picture.}}
    \label{fig:handlebody}
\end{figure}

{Suppose $\partial_eN$ is a surface of genus $g$, and there are $s$ toroidal interior boundary components (clearly $s\le g$).} Let $\mathcal{E}=\{D_j\}$ be a collection of disjoint compressing disks in $N$. Note that $\eta_j=\partial D_j$ forms a collection of disjoint simple closed curves on $\partial_e N$. {We always assume curves in this collection are not homotopic to each other. The collection $\mathcal{E}$ is said to be \emph{optimal} if cutting $N$ along all of $D_j$ produces a topological ball $B$ and a collection of thickened tori $\{T_l\}$. The only exception is $g=s=2$, in which case an optimal collection $\mathcal{E}$ consists of one compressing disk dividing $N$ into two thickened tori.}

From now on we always assume $\mathcal{E}$ is {optimal}.
After cutting along $\{D_j\}$, each $D_j$ becomes two disjoint disks $D_j^+$ and $D_j^-$. For the topological ball $B$, on its boundary lies the disjoint union of several of the disks $D_j^\pm$. Denote the portion of $\partial B$ coming from $\partial_eN$ by $\partial'B$. {Note that $\partial'B$ is an $n$-holed sphere for some $n\ge3$}. For each thickened torus $T_l$, one of its boundary components $\partial _eT_l$ contains one among the pair $D_j^\pm$ for some $j$. As a convention, we will always assume $D_j^+$ bounds the torus (and so $D_j^-$ bounds the ball). Denote the portion of $\partial _eT_l$ from $\partial_eN$ by $\partial'T_l$. {Note that $\partial'T_l$ is a one-holed torus}.

Since each $\gamma_i$ is homotopically nontrivial in $N$, it must intersect $\cup_j\eta_j$. Each segment of $\cup_i\gamma_i-\cup_j\eta_j$ lies on some $\partial'T_l$ or $\partial' B$, and connects two curves (possibly identical) among $\{\eta_j^\pm\}$, where $\eta_j^\pm=\partial D_j^\pm$. {For an example, see Figure~\ref{fig:handlebody} with Figure~\ref{fig:apollonian_ball}(A)}.
We have the following simple lemma.
\begin{lem}\label{lem:connects}
    The collection of curves $\{\gamma_i\}\cup\{\eta_j\}$ fills the surface $\partial_eN$, i.e. $\partial_eN-\cup_i\gamma_i-\cup_j\eta_j$ consists of disks.
\end{lem}
\begin{proof}
    Suppose by contradiction that $\partial_eN-\cup_i\gamma_i-\cup_j\eta_j$ contains a nontrivial curve $\gamma$. Then in particular it lies on $\partial'B$ or some $\partial'T_l$. In the former case, $\gamma$ is homotopically trivial in $N$, and in the latter case, it is homotopic to a curve on an interior toroidal boundary component. Moreover, since $\gamma$ does not intersect any $\gamma_i$, it is contained in a component $S'$ of $\partial_eN-\cup_i\gamma_i$. But $S'$ is incompressible in $N$, and no curve on $S'$ represents a cuspidal curve around a rank-2 cusp.
\end{proof}

The collection $\mathcal{E}$ is called \emph{admissible} if it is optimal, and satisfies the following conditions.
\begin{itemize}
    \item For each segment of $\cup_i\gamma_i-\cup_j\eta_j$ contained in $\partial B$, its endpoints do not lie on the same $\eta_j^\pm$.
    \item Any essential simple path on $\partial'B$ connecting some $\eta_j^\pm$ to itself intersects $\cup_i\gamma_i$ at least twice.
\end{itemize}
Here an essential path connecting two boundary components of a $n$-holed sphere is one that cannot be homotoped rel boundary into the boundary.

The following result reduces finding a subdivision rule to finding an admissible collection of compression disks.
\begin{prop}\label{prop:curve_and_subdivision}
    Any admissible collection $\mathcal{E}$ of disjoint compressing disks gives an irreducible simple finite or $\mathbb{Z}^2$-subdivision rule $\mathcal{R}$ on the graph $\mathcal{G}$ so that each face of the subdivision is an induced Jordan domain.
\end{prop}
\begin{remark}
    We note that because of the arguments in \S \ref{subsec:sjd}, to obtain rigidity it is not necessary to make sure every face of the subdivision is a Jordan domain nor require irreducibility.
    However, the geometrically meaningful modification given here ensures that the first level of subdivision can be described in a single fundamental domain. {This turns out to have dynamical and number-theoretic applications in an upcoming joint work of S.~Lim, K.~Park and the second-named author.}
\end{remark}

\subsection*{Subdivisions and compressing disks}\label{subsec:sacd}
To give a precise statement and prepare the proof of Proposition~\ref{prop:curve_and_subdivision}, we first discuss how to translate the topological data of compressing disks to combinatorial information about subdivision of the graph $\mathcal{G}$.

Via the homeomorphism between the Kleinian manifold $(\overline{M},\partial\overline{M})$ and the compression body $(N,\partial N-P)$, objects discussed above (compressing disks, fragments of their boundary curves, etc.) are also present in $\overline{M}$, which we denote by the same notations. 

In the universal cover $\mathbb{H}^3\cup\Omega$, choose a lift $\tilde B$ of $B$. Then
\begin{itemize}
    \item $\widetilde B$ is bounded by lifts $\widetilde D_{j}^\pm$ of compression disks;
    \item The boundary at infinity of each $\widetilde D_{j}^\pm$, denoted by $\widetilde\eta_j^\pm$, is a closed curve in $\widehat\C$ whose only intersections with circles in the packing $\mathcal{P}$ are points of tangency;
    
    \item The outward-pointing normals along each $\widetilde D_j^\pm$ points to a region $E_j^\pm$ bounded by $\widetilde\eta_j^\pm$, and $\widehat{\mathbb{C}}-\bigcup_{D^\pm_j \text{ bounds }B}\overline{E_j^\pm}$ is contained entirely in the domain of discontinuity -- it is a lift of $\partial'B-\cup_i\gamma_i$.

    \item If $D_j^{+}$ bounds a thickened torus, then $D_j^{-}$ bounds $B$, with associated lift $\widetilde D_j^{-}$ and region $E_j^{-}$. We will define $E_j^+$ the complement of $E_j^{-}$ and $\tilde\eta^{+}=\tilde\eta^{-}$. 
    
    \item For each pair $D_j^{\pm}$ not bounding a thickened torus, there exists a pair of gluing maps $g_j^{\pm}\in\Gamma$ that glues back together $\widetilde D_j^{\pm}$. More precisely, $g_j^{+}=(g_j^{-})^{-1}$, and $g_j^{\pm}$ maps the complement of $E_j^{\mp}$ to $E_j^{\pm}$.
    \item If $D_j^{+}$ bounds a thickened torus $T_l$, there exists a rank-2 parabolic subgroup $\Gamma_l\subseteq\Gamma$ such that $\overline{E_j^{-}}=\overline{\bigcup_{\gamma\in\Gamma_l-\{id\}}\gamma\cdot E_j^+}$.
\end{itemize}
Note that each $\widetilde\eta_j^\pm$ determines a cycle $\partial P_j^{\pm}$ in $\mathcal{G}$, which bounds a polygon whose face corresponds to the region $E_j^{\pm}$. The third point above implies that the union of these polygons is the sphere. We claim
\begin{lem}
    If the collection $\mathcal{E}$ is admissible, then the cycles $\partial P_j^\pm$ are induced subgraphs in $\mathcal{G}$ bounding a Jordan domain.
\end{lem}
\begin{proof}
    If a cycle fails to bound a Jordan domain, then some vertices or edges are repeated. If an edge is repeated, the the corresponding curve $\widetilde\eta_j^\pm$ passes through a point of tangency twice. Since points of tangecy are parabolic fixed points, this only happens if $\eta_j^\pm$ is connected to itself by a segment of $\cup_i\gamma_i-\cup_j\eta_j$, contradicting admissibility. Technically, there might be several nested curves (which are orbits of some of $\widetilde\eta_t^\pm$ under the Kleinian group) passing through the same tangency twice (see Figure~\ref{fig:nested}), but we can apply the argument above to the innermost curve.
    \begin{figure}[htp]
        \centering
        \includegraphics[width=0.35\textwidth]{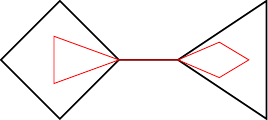}
        \caption{Nested cycle repeating the same edge}
        \label{fig:nested}
    \end{figure}
    
    Similarly, if some vertices (but none of the edges) are repeated, then there exist two distinct segments of $\widetilde\eta_j^\pm$ that lie in the same disk of the packing, and we can connect them via a path in that disk. If this path is contained in $E^{\pm}_j$ and $D_j^{\pm}$ both bound $B$, then we can use $g^{\mp}_j$ to obtain a similar path in the complement of $E^{\mp}_j$, and it can be modified to remain in the domain of discontinuity. If $D_j^+$ bounds a thickened torus, then this path cannot be in $E_j^{-}$ (since it contains a $\mathbb{Z}^2$-tiling of $E_j^{+}$), and as above, we can modify this path to remain in the domain of discontinuity. Thus the path projects down to a simple path on $\partial' B$ violating admissibility. See the picture on the left of Figure~\ref{fig:cycles}.

    Finally, if $\partial P_j^{\pm}$ is not an induced subgraph, then there exists two vertices $v,w$ nonadjacent in $\partial P_j^{\pm}$ but the corresponding disks $D_v$ and $D_w$ are tangent. Then two distinct segments of $\widetilde\eta_j^\pm$ lie in these two disks, and we can connect them via a path that passes through only one tangency point. As above, this contradicts admissibility. See the picture on the right of Figure~\ref{fig:cycles}.
    \begin{figure}[htp]
        \centering
        \includegraphics[width=0.35\textwidth]{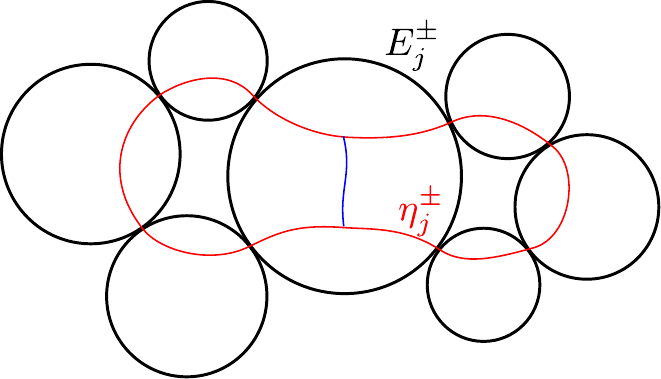}
        \includegraphics[width=0.35\textwidth]{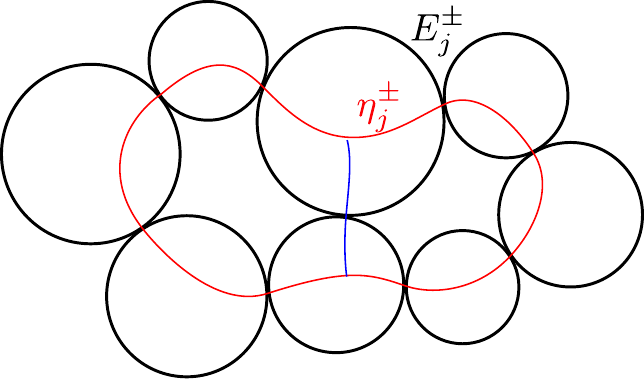}
        \caption{In the case where $\partial P_j^{\pm}$ does not bound Jordan domain (left), or is not an induced graph (right), we can find paths contradicting admissibility.}
        \label{fig:cycles}
    \end{figure}
\end{proof}

We now proceed to describe how to associate a subdivision rule to an admissible collection $\mathcal{E}$ of compressing disks. For simplicity, first assume that the Kleinian group $\Gamma$ contains no rank-2 parabolic subgroups.
Then we can define a finite subdivsion rule $\mathcal{R}$ as follows.

\subsubsection*{The polygons}
The collection of polygons is $\{P_j^\pm\}$.

\subsubsection*{The $\mathcal{R}$-complex}
We first identify a spherical CW-complex whose faces are identified with the given polygons. Recall that each $P_j^{\pm}$ comes from a cycle in the nerve $\mathcal{G}$, together with one of the Jordan domains bounded by the cycle. The union of these cycles gives a spherical graph, whose faces are precisely the polygons $P_j^{\pm}$.

\subsubsection*{The subdivision and the cellular maps}
For each $P_j^\pm$, the gluing maps $g_{j}^{\mp}$ between $\widetilde D_j^{\pm}$ induce homeomorphisms between $P_j^{\pm}$ and the complement of $P_j^{\mp}$. Note that the complement of $P_j^{\mp}$ is the union of all the polygons except $P_j^{\mp}$, so this gives a subdivision of $(P_j^{\mp})^c$. Pulling this subdivision back via $g_j^{\mp}$, we get a subdivision of $P_j^{\pm}$. Moreover, the restriction of $g_j^{\mp}$ to each polygon in the subdivision gives the required cellular maps.

\subsubsection*{Rank-2 cusps}
We briefly discuss the case where $\Gamma$ contains at least one rank-2 parabolic subgroup. Without loss of generality, suppose $D_1^+$ bounds a thickened torus $T_1$.
Then the corresponding rank-2 parabolic subgroup $\Gamma_1$ induces a tiling of $P_1^{-}$ by translates of $P_1^{+}$.
Using the notation of \S\ref{sec:z2sr}, we set $Q_1=P_1^-$. Then
$$Q_1=\overline{\bigcup_{\gamma\in\Gamma_1-\{id\}}\gamma\cdot P_1^+}.$$
Now suppose $P_1^+,Q_1,\ldots,P_s^+,Q_s$ are all the polygons obtained as above for thickened tori $\{T_l\}$, and $P_{s+1}^\pm,\ldots,P_{t}^\pm$ the remaining polygons obtained from other compressing disks, then one can define subdivisions of each polygon similarly as above. This gives a $\Z^2$-subdivision rule, as desired.

\subsection{Construction of admissible $\mathcal{E}$}\label{subsec:coae}
It remains to explicitly choose an admissible collection $\mathcal{E}$ of compressing disks on $N$. Suppose $\partial_eN$ is a surface of genus $g$, and there are $s$ toroidal boundary components (clearly $s\le g$).

It is easy to see that we can find $g$ nontrivial compressing disks $D_1,\ldots, D_g$ (with boundary $\eta_1,\ldots,\eta_g$), so that $D^+_l$ bounds thickened tori $T_l$ for $1\le l\le s$, and $N-\bigcup_{l=1}^s T_l-\bigcup _{j=s+1}^gD_j$ is a connected ball $B$. (The only exception is $g=2$ and $s=2$, where we only need one compressing disk $D_1$ to decompose $N$ into two thickened tori.) If the collection $\mathcal{E}:=\{D_1,\ldots, D_s\}$ is already admissible, then we are done. Otherwise, we can modify the collection as follows.

First consider $D_1$; the other disks bounding tori can be modified similarly. First suppose on $B$, there is a segment $L$ of $\cup\gamma_i-\cup_j\eta_j$ on $B$ connecting $\eta_1$ to itself. The end points of $L$ divides $\eta_1$ into two parts $\eta_1'$ and $\eta_1''$. We first claim
\begin{lem}
    There exists a segment $L'$ of $\cup\gamma_i-\cup_j\eta_j$ on $\partial'T_1$ connecting the interior of $\eta_1'$ to that of $\eta_1''$.
\end{lem}
\begin{proof}
    We can connect the two endpoints of $L$ in $\partial'T_1$ via a simple segment $J$ not homotopic to either $\eta_1'$ nor $\eta_1''$. Then $L\cup J$ is homotopic in $N$ to a simple close curve on the torus boundary. On the other hand, if $J$ does not intersect any $\gamma_i$, the curve $L\cup J$ can be homotoped in $\partial_eN$ to be disjoint from $\cup\gamma_i$. But such a curve cannot be homotopic to a cuspidal curve around the rank-2 cusp, a contradiction. So $J$ must intersect some segment $L'$ of $\gamma_i$ in $\partial'T_1$, and it is then easy to see that $L'$ connects $\eta_1'$ to $\eta_1''$.
\end{proof}
\begin{figure}[htp]
    \centering
    \includegraphics[width=0.5\textwidth]{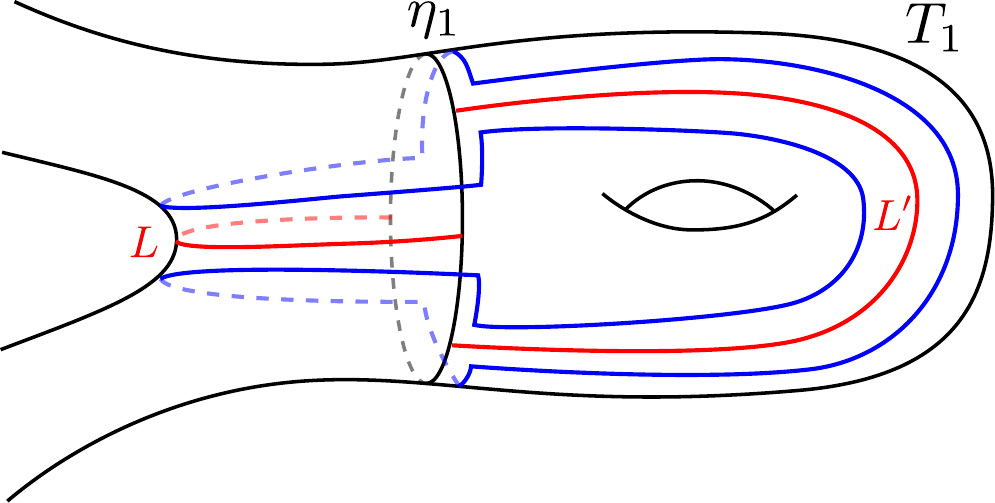}
    \caption{Modifying the disk bounding a thickened torus. To avoid crowding the figure, the torus interior boundary component is not drawn.}
    \label{fig:disk-pushing-torus}
\end{figure}
Now as in Figure~\ref{fig:disk-pushing-torus}, tracing two parallel copies of $L$, one copy of $\eta_1$, and two parallel copies of $L'$, we obtain a curve $\eta_1'$ and it is easy to see that it bounds a compressing disk $D_1'$. Clearly there are fewer segments of $\cup_i\gamma_i-\cup_j\eta_j$ connecting $\eta_1'$ to itself, and this does not affect the disks bounding other thickened tori. Inductively, we can thus modify $D_1,\ldots,D_s$ so that no segment of $\cup_i\gamma_i-\cup_j\eta_j$ connects $\eta_j^-$ to itself for $1\le j\le s$. An entirely analogous argument also ensures that the second condition in the definition of admissibility is satisfied for these disks.

Next we consider $D_g$ assuming it does not bounds a thickened torus. Suppose there is a segment $L$ of $\cup_i\gamma_i-\cup_j\eta_j$ connecting $D_g^+$ to itself in $B_1$. Tracing two parallel copies of $L$ and a copy of $\eta_g$, we obtain two compressing disks $D_g^1$ and $D_g^2$; see Figure~\ref{fig:disk_pushing_ball}.
\begin{figure}[htp]
    \centering
    \includegraphics[width=0.45\textwidth]{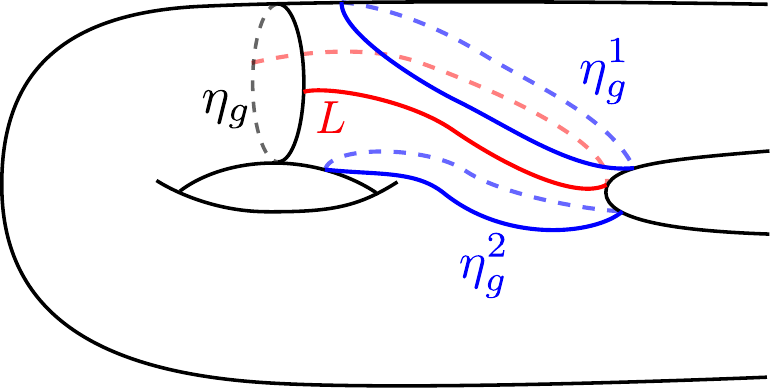}
    \caption{Modifying a disk not bounding a thickened torus}
    \label{fig:disk_pushing_ball}
\end{figure}
One of $D_g^1$ and $D_g^2$, say $D_g^1$, is either homotopic to one of the other disks, or a separating disk (for otherwise, the genus of $\partial_eN$ would be larger than $g$).

Replacing $D_g$ with $D_g^2$, note that the number of segments violating the first condition in admissibility either reduces by 1, or remain the same, if the two segments in $\cup_i\gamma_i-\cup_j\eta_j$ extending $L$ beyond its two endpoints terminate at the same $\eta_j^\pm$. We can then replace $D_j$ similarly by one of the two disks $D_j^1$ and $D_j^2$. This process can be repeated, and terminates when one of the following happens:
\begin{itemize}
    \item The two endpoints of extended $L$ are connected by some segment of $\cup_k\gamma_k-\cup_j\eta_j$. This is impossible since we now have a curve among $\gamma_i$ that is homotopically trivial.
    \item The two segments extending $L$ terminate at two different $\eta_j^\pm$. In this case, the number of segments violating the first condition in admissibility decreases by 1.
    \item The two endpoints of $L$ lie on some $\eta_j$ bounding a thickened torus. In this case, modify the disk as in Figure~\ref{fig:disk-pushing-torus}, we again reduce the number of segments violating the first condition by 1.
\end{itemize}
Inductively, we can thus guarantee the first condition is satisfied. An entirely analogous argument also guarantees the second condition is satisfied.

In this way, we obtain an admissible collection of compressing disks, as desired.

\subsection{A topological model}\label{subsec:atm}
We give a purely topological way of visualizing the subdivision rule. For simplicity, we focus on one simple example given in Figure~\ref{fig:handlebody}.

It is easy to find an admissible collection of compressing disks. In fact, for one such collection, after cutting the handlebody along the disks, we obtain a topological ball $B$ as in Figure~\ref{subfig:ball_model}. The subdivisions are also illustrated on the Apollonian circle packing in Figure~\ref{subfig:triadic}.
\begin{figure}[htp]
    \centering
    \begin{subfigure}[b]{0.35\textwidth}
        \centering
        \includegraphics[width=\textwidth]{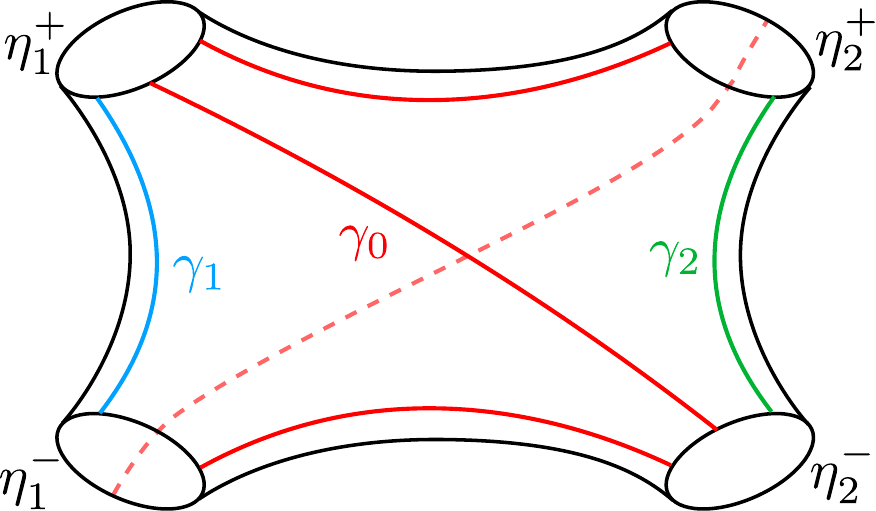}\vspace{15pt}
        \caption{}
        \label{subfig:ball_model}
    \end{subfigure}
    \begin{subfigure}[b]{0.55\textwidth}
        \centering
        \includegraphics[width=\textwidth]{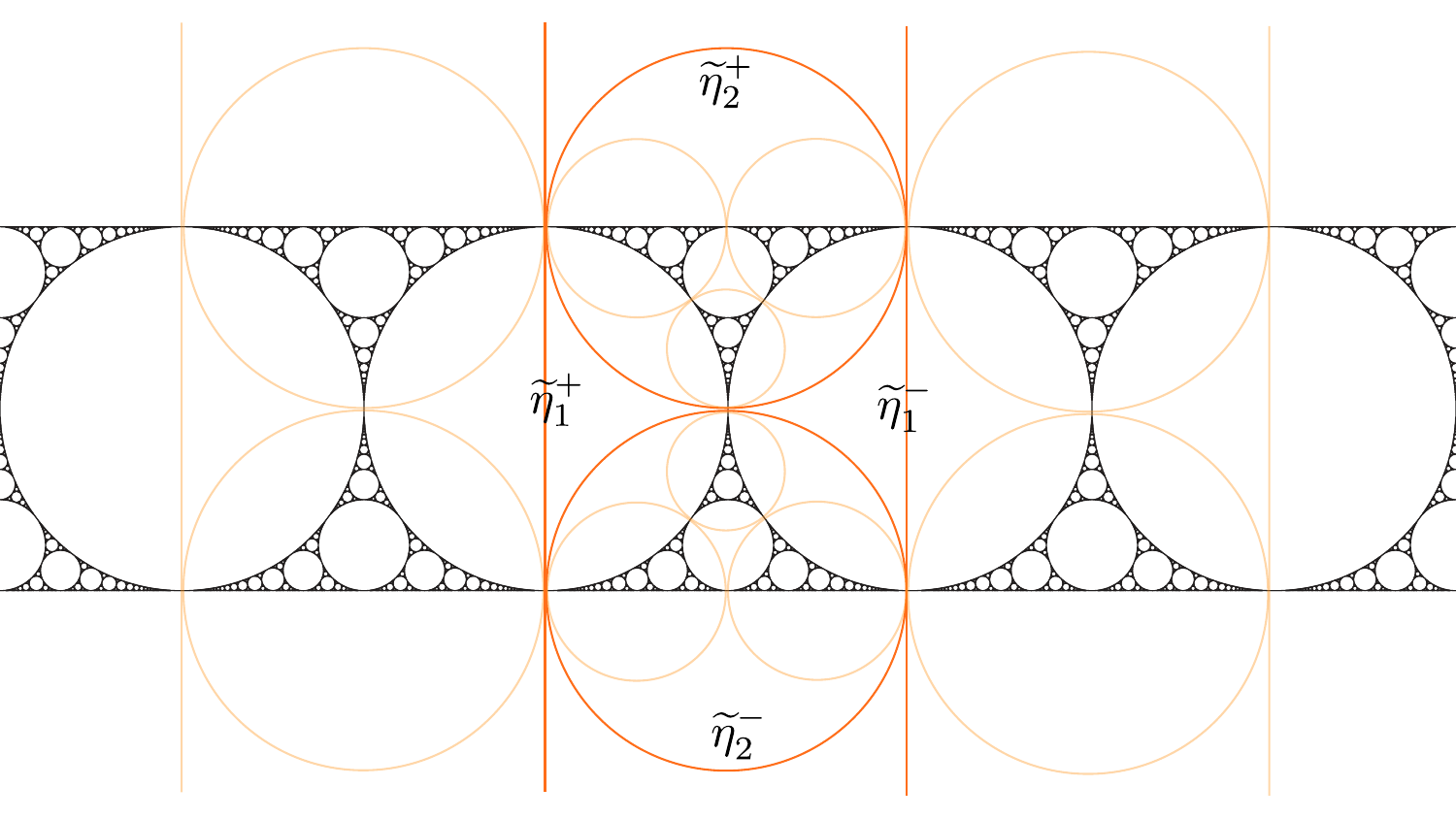}
        \caption{}
        \label{subfig:triadic}
    \end{subfigure}
    \caption{Cutting along an admissible collection of compressing disks, topologically (A) or on the limit set (B)}
    \label{fig:apollonian_ball}
\end{figure}

We can also use $B$ as a fundamental domain to build up a universal cover for the handlebody, which is homotopic to a infinite 4-valent tree (note that $\gamma_1$ and $\gamma_2$ generates the fundamental group); see Figure~\ref{fig:infinite_tree}.
\begin{figure}[htp]
    \centering
    \includegraphics[width=0.5\textwidth]{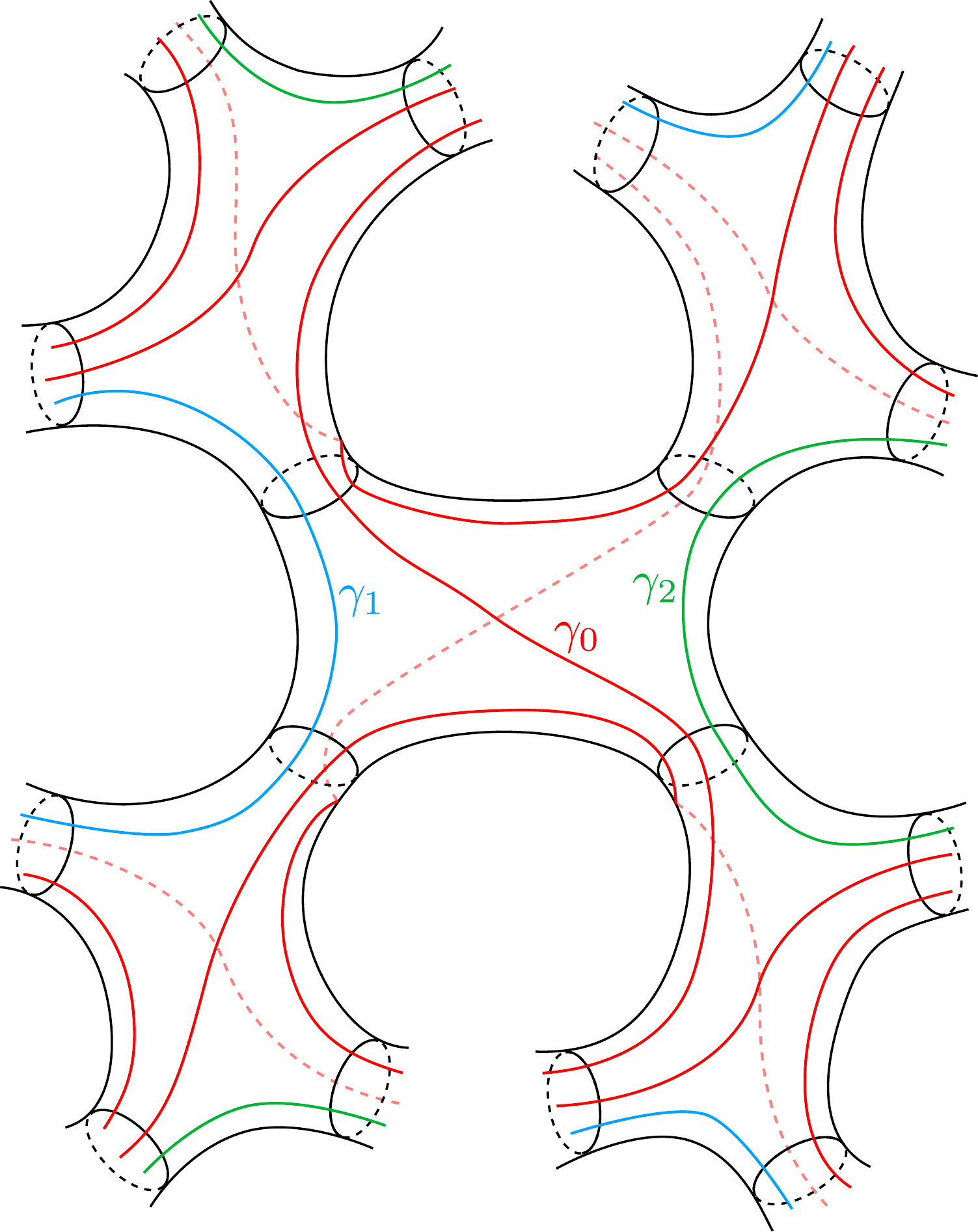}
    \caption{An infinite tree model for the universal cover}
    \label{fig:infinite_tree}
\end{figure}
In this picture, the parabolic locus is unraveled to be a collection of infinite lines. Each component of the complement of these lines on the boundary corresponds to a disk in the packing. Entering a new copy of the ball $B$ is tantamount to entering a piece of the subdivision at the next level. The end space of the tree is naturally identified with the limit set with each parabolic fixed point blown up to be two points, which are also the ends of the line associated to the corresponding parabolic element.

\appendix

\section{Generalizations of subdivision rules}\label{sec:z2sr}
In this appendix, we will explain how our methods can be generalized for graphs with other types of subdivision rules.

\subsection{Interpolations of finite subdivision rules}
In this section, we generalize the finite subdivision rules.
This generalization allows us to study subdivisions that are not periodic, and can be regarded as the interpolation of multiple finite subdivision rules.

\begin{defn}\label{defn:ifsr}
An {\em interpolation of finite subdivision rules} $\mathcal{R}$ consists of
\begin{enumerate}
\item a finite collection of {oriented} polygons $\{P_i: i=1,..., k\}$;
\item a subdivision $\mathcal{R}(P_i)$ that decomposes each polygon $P_i$ into $m_i \geq 2$ closed 2-cells
$$P_i = \bigcup_{j=1}^{m_i} P_{i, j}$$
so that each edge of $\partial P_i$ contains no vertices of $\mathcal{R}(P_i)$ in its interior {and each 2-cell $P_{i,j}$ inherits the orientation of $P_i$}; and
\item {a map $\sigma:\bigcup_{i=1}^k\bigcup_{j=1}^{m_i}\left(\{(i,j)\}\times\{1,\ldots,k_{i,j}\}\right)\to\{1,\ldots,k\}$ and a collection of orientation preserving cellular maps for each $i\in\{1,\ldots,k\}, j\in\{1,\ldots,m_i\}, l\in\{1,\ldots,k_{i,j}\}$, denoted by
$$
\psi_{i,j,l}: P_{\sigma(i,j,l)} \longrightarrow P_{i,j},
$$
that are homeomorphisms between the open 2-cells.}
\end{enumerate}
\end{defn}
\begin{rmk}
    We remark that the difference between an interpolation of finite subdivision rules and a finite subdivision rule is in the condition $(3)$.
    For a finite subdivision rule, there is a {\em unique} identification between a face of the subdivision with a polygon in the original list, while an interpolation of finite subdivision rules allows multiple different ways of identifications.
\end{rmk}

Let $\mathcal{R}$ be an interpolation of finite subdivision rules.
Let $\mathcal{R}(P_i)$ be the subdivision of $P_i$ and $\mathcal{G}^1_i$ be the corresponding 1-skeleton.
\begin{defn}\label{defn:icyl}
Let $\mathcal{R}$ be an interpolation of finite subdivision rules. 
We say it is 
\begin{itemize}
    \item {\em simple} if for each $i$, $\mathcal{G}^1_i$ is a simple graph;
    \item {\em irreducible} if for each $i$, $\partial P_i$ is an induced subgraph of $\mathcal{G}^1_i$;
    \item {\em acylindrical} if for each $i$, each face of $\mathcal{R}(P_i)$ is an induced Jordan domain, and the face $P_i$ is acylindrical in $\mathcal{G}^1_i$.
\end{itemize}
\end{defn}

\begin{figure}[htp]
    \centering
    \includegraphics[width=0.25\textwidth]{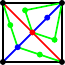}    \caption{The first few levels of an interpolation of the two finite subdivision rules shown in Figure~\ref{fig:subd}. Note that this example is not acylindrical.}
    \label{fig:interpolation}
\end{figure}

Let $P\in \{P_1,..., P_k\}$ be a polygon.
To iterate $\mathcal{R}$ on $P$, we need to make a choice of the identification of the face of $\mathcal{R}(P)$ with a polygon in $\{P_i: i=1,..., k\}$.
Suppose that such a choice is made at every iteration.
Abusing the notations, we denote the sequence of subdivisions by $\mathcal{R}^n(P)$. 
Let $\mathcal{G}^n$ be the 1-skeleton of $\mathcal{R}^n(P)$, and $\mathcal{G}$ be the direct limit of $\mathcal{G}^n$.
Let $F$ be the external face for $\mathcal{G}$.

We call the infinite graph $\mathcal{G}$ an {\em admissible subdivision graph} for $\mathcal{R}$.
We remark that there can be uncountably many different admissible subdivision graph for $\mathcal{R}$.

The same proof as Theorem \ref{thm:LR} gives the following.
\begin{theorem}\label{thm:ILR}
Let $\mathcal{R}$ be a simple, irreducible, acylindrical interpolation of finite subdivision rules.
Let $\mathcal{G}$ be an admissible subdivision graph for $\mathcal{R}$.
Then the Teichm\"uller space $\Teich(\mathcal{G})$ is isometrically homeomorphic to $\Teich(\Pi_{F})$.
In particular, any two circle packings with nerve $\mathcal{G}$ are quasiconformally homeomorphic.
\end{theorem}

Given a sequence of nested faces for $\mathcal{G}$, we can define the corresponding iteration of renormalizations $\mathfrak{R}^n$.
The same proofs of Theorems~\ref{thm:LB} and \ref{thm:EC} also give
\begin{theorem}\label{thm:IEC}
Let $\mathcal{R}$ be an interpolation of finite subdivision rules.
There exist constants $\delta < 1$ and $C>0$ {depending only on $\mathcal{R}$} such that the following holds.

Let $\mathcal{G}$ be an admissible subdivision graph for $\mathcal{R}$.
\begin{itemize}
\item(Finite approximation) Suppose $\mathcal{P}, \mathcal{P}' \in \Teich(\mathcal{G}^{n+k})$ have the same external class.
Let $\mathcal{P}^k, (\mathcal{P}')^k$ be the sub-circle packing associated to $\mathcal{G}^k$. Then
$$
d(\mathcal{P}^k, (\mathcal{P}')^k) \leq \delta^{n-1} \min\{d(\mathcal{P}, \mathcal{P}'), C\} = O(\delta^n).
$$
    \item(Renormalization) Let $\mathcal{P}, \mathcal{P}' \in \Teich(\mathcal{G})$. Then given a sequence of nested faces for $\mathcal{G}$, we have
\begin{align*}
 d(\mathfrak{R}^n((\mathcal{P}, x)), \mathfrak{R}^n((\mathcal{P}', x')) &\leq \delta^{n-1} d(\mathfrak{R}((\mathcal{P}, x)), \mathfrak{R}((\mathcal{P}', x')) \\
 &\leq \delta^{n-1} \min\{d((\mathcal{P}, x), (\mathcal{P}', x')), C\} \\&= O(\delta^n). 
\end{align*}
\end{itemize}
\end{theorem}

\subsection{$\Z^2$-subdivision rules}
In this subsection, we discuss what happens if the subdivision rule is not finite.
We illustrate the ideas using a special type of subdivision rules, called {\em $\Z^2$-subdivision rule}.
Such subdivision rules arise naturally from rank 2 cusps of a Kleinian groups, and will be applied in \S \ref{sec:kcp}.

\subsection*{$\Z^2$-tilings}
A CW-complex $X$ that is homeomorphic to $\C$ is called a {\em $\Z^2$-tiling} if there exists a properly discontinuous action of $\Z^2$.
Note that each element of $\Z^2$ is a cellular homeomorphism of $X$, thus $X$ necessarily contains infinitely many 2-cells.

Let $P$ be a fundamental domain for this $\Z^2$-action, which we assume to be a Jordan domain.
Let $Q = \overline{\widehat\C - P}$ be the complementary polygon associated to $P$.
Then the $\Z^2$-tiling generates a subdivision of $Q$. More precisely, we can write
$$
Q = \overline{\bigcup_{g \in \Z^2-\{id\}}g\cdot P},
$$
where $(g\cdot \Int(P)) \cap (g'\cdot \Int(P)) = \emptyset$ if $g\neq g'$.
We remark that technically, this is a subdivision of the punctured polygon $Q^*:= Q - \{\infty\}$.
To avoid complicated notations, we will abuse the terminology and simply call it a subdivision of $Q$.

\subsection*{$\Z^2$-subdivisions}
A $\Z^2$-subdivision rule is a generalization of a finite subdivision rule, that allows infinitely many faces in the subdivision.
\begin{defn}
    A $\Z^2$-subdivision rule $\mathcal{R}$ consists
\begin{enumerate}
\item a finite collection of {oriented} polygons $\{P_1,..., P_k, Q_1, ..., Q_l\}$ with $l \leq k$ and $Q_i = \overline{\widehat\C - P_i}$ is the complementary polygon of $P_i$;
\item a subdivision $\mathcal{R}(P_i)$ that decomposes each polygon $P_i$ into at least two closed 2-cells
$$P_i = \bigcup_{j=1}^{m_i} P_{i, j} \cup \bigcup_{j=1}^{n_i} Q_{i,j}$$
so that each edge of $\partial P_i$ contains no vertices of $\mathcal{R}(P_i)$ in its interior {and each 2-cell $P_{i,j},Q_{i,j}$ inherits the orientation of $P_i$}; 
\item {two maps $\sigma$, $\tau$ specifying types and a collection of {orientation-preserving} cellular maps, denoted by}
$$
\psi_{i,j}: P_{\sigma(i,j)} \longrightarrow P_{i,j}, \;\; \phi_{i,j}: Q_{\tau(i,j)} \longrightarrow Q_{i,j}
$$
that are homeomorphisms between the open 2-cells;
\item a $\Z^2$-tiling $X_i$ with fundamental domain $P_i$ which induces a subdivision $\mathcal{R}(Q_i)$ by
$$
Q_i = \overline{\bigcup_{g \in \Z^2-\{id\}}g\cdot P_i},
$$
where each closed 2-cell of $\mathcal{R}(Q_i)$ is naturally identified with $P_i$.
\end{enumerate}
\end{defn}

Similar as a finite subdivision rule, we can iterate this procedure.
We define $\mathcal{G}^n_i$ and $\mathcal{H}^n_i$ as the 1-skeleton of $\mathcal{R}^n(P_i)$ and $\mathcal{R}^n(Q_i)$ respectively.
Let $\mathcal{G}_i = \bigcup_n \mathcal{G}^n_i$ and $\mathcal{H}_i = \bigcup_n \mathcal{H}^n_i$.
We say $\mathcal{R}$ is {\em simple} if $\mathcal{G}_i$ and $\mathcal{H}_i$ are simple.

Following the discussion in \S \ref{subsec:sjd}, we say $\mathcal{R}$ is {\em acylindrical} if
\begin{itemize}
    \item each face of $\mathcal{G}^1_i$ is an induced Jordan domain;
    \item the subgraph $\mathcal{G}^0_i$ is acylindrical in $\mathcal{G}^1_i$.
\end{itemize}
We remark that by our definition, the above two properties hold automatically for $\mathcal{H}^1_i$ and $\mathcal{H}^0_i$.

We can define an $\mathcal{R}$-complex $X$ similar as in the case of finite subdivision rule.
If $X$ is homeomorphic to the Riemann sphere $\widehat\C$, we say the corresponding graph $\mathcal{G} = \bigcup_n \mathcal{G}^n$ a spherical subdivision graph with $\Z^2$-subdivision rule.

\begin{figure}[htp]
    \centering
    \includegraphics[width=0.75\textwidth]{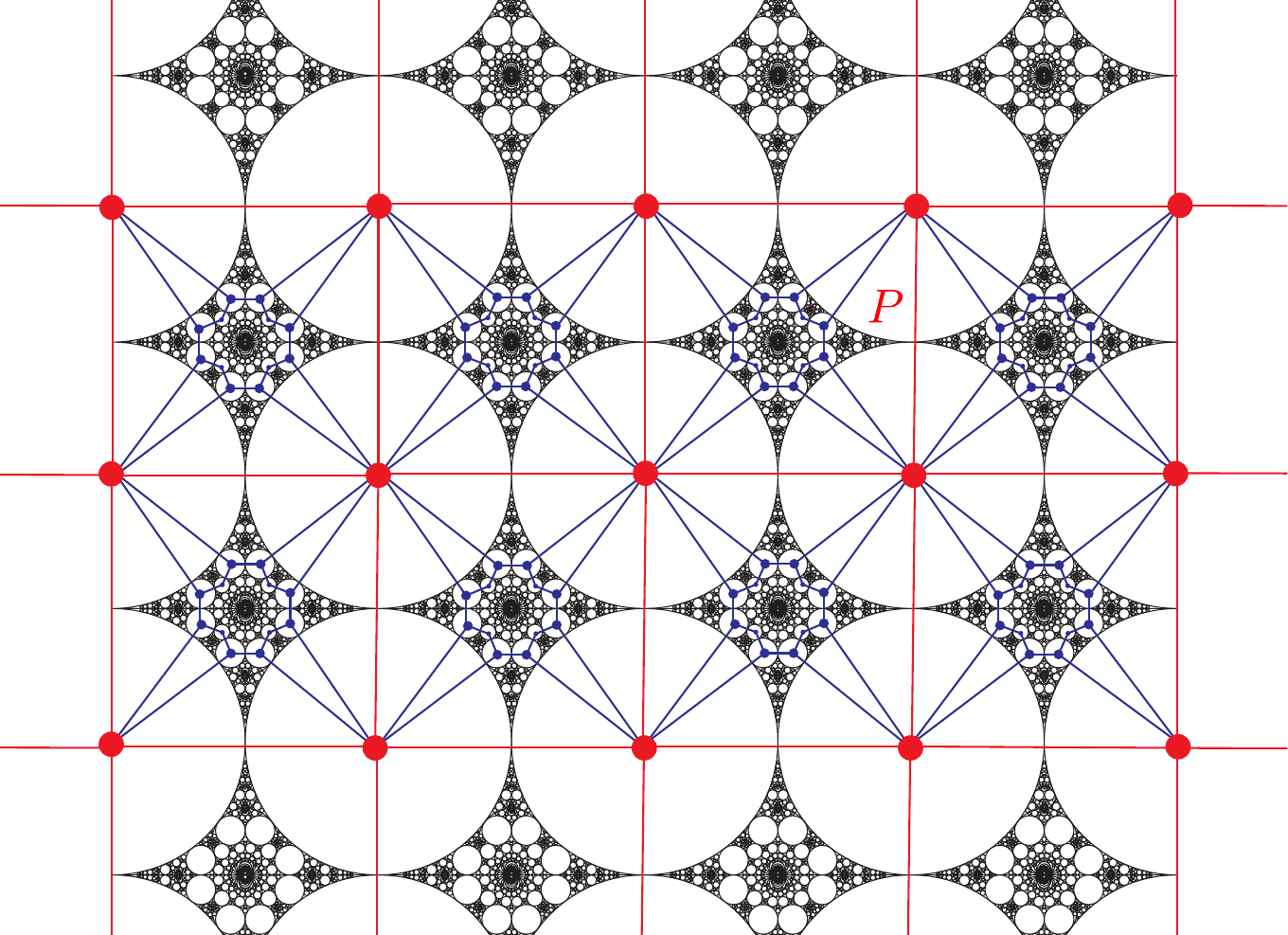}
    \caption{An example of a subdivision graph for a $\Z^2$-subdivision rule with the associated circle packing, given by a Kleinian group}
    \label{fig:z2sreg}
\end{figure}

\subsection*{Existence and rigidity of $\Z^2$-subdivision rule}
\begin{theorem}\label{thm:Z2EU}
    Let $\mathcal{R}$ be a simple acylindrical $\Z^2$-subdivision rule.
    \begin{itemize}
        \item The Teichm\"uller space
        $\Teich(\mathcal{G}_i)$ is isometrically homeomorphic to $\Teich(\Pi_{F_i^{\ext}})$, where $F_i^{\ext}$ is the external face of $\mathcal{G}_i$.
        \item Suppose $\mathcal{G}$ is a simple spherical subdivision graph with $\Z^2$-subdivision rule $\mathcal{R}$.
        Then there exists a unique circle packing $\mathcal{P}$ whose nerve is $\mathcal{G}$.
    \end{itemize}
\end{theorem}

We remark that with some minor adjustments, the proof is almost identical to the case of a finite subdivision rule.
We will only sketch the proof and emphasize the one subtle difference here.

Consider the graph $\mathcal{H}^1_i$.
Note that it is simply the 1-skeleton of a $\Z^2$-tiling, thus it is easy to construct circle packings with nerve $\mathcal{H}^1_i$.
Each face of $\mathcal{H}^1_i$ can be identified with $P_i$.
Hence, for each face $F$, we can construct a map
$$
\pi_F: \Teich(\mathcal{H}^1_i) \longrightarrow \Teich(\Pi_{P_i}).
$$
Similar as in \S \ref{subsec:su}, we can construct a compact set $K_i \subseteq \Teich(\Pi_{P_i})$.
Consider the set
$$
\Teich_{K_i}(\mathcal{H}^1_i) = \{\mathcal{P} \in\Teich(\mathcal{H}^1_i): \pi_F(\mathcal{P}) \in K_i \text{ for all faces } F \}.
$$
Let 
$$
\rho_{Q_i} = \pi_{Q_i}\circ \tau_{\mathcal{H}^0_i, \mathcal{H}^1_i}: \Teich(\mathcal{H}^1_i) \longrightarrow \Teich(\Pi_{Q_i})
$$
be the composition of a projection map with the skinning map.

\begin{lem}\label{lem:uci}
    Let $\mathcal{P}_1, \mathcal{P}_2 \in \Teich_{K_i}(\mathcal{H}^1_i)$.
    Then there exists a constant $\delta$ so that
    $$
    d(\rho_{Q_i}(\mathcal{P}_1), \rho_{Q_i}(\mathcal{P}_2)) \leq \delta \cdot d(\mathcal{P}_1, \mathcal{P}_2) = \delta \cdot \sup_{F}d(\pi_F(\mathcal{P}_1)), \pi_F(\mathcal{P}_2)).
    $$
\end{lem}
\begin{proof}
    Since $\pi_F(\mathcal{P}_1)$ (and $\pi_F(\mathcal{P}_2)$) is contained in a compact set, 
    $\mathcal{P}_1$ and $\mathcal{P}_2$ are quasiconformally homeomorphic.
    It follows from the discussion in \cite[\S 3]{He91} that the limit set of the group generated by reflection along circles in $\mathcal{P}_1$ (and $\mathcal{P}_2$) has zero measure.
    Therefore 
    $d(\mathcal{P}_1, \mathcal{P}_2) = \sup_{F}d(\pi_F(\mathcal{P}_1)), \pi_F(\mathcal{P}_2)) \leq C_i$, where $C_i$ is the diameter of $K_i$.

    Let $F^{\ext}_i$ be the external face of $\mathcal{H}^1_i$. Then $\mathcal{H}^1_i = \bigcup_{g\in \Z^2} g\cdot \partial F^{\ext}_i$.
    To get the uniform contraction, we can consider the subgraph 
    $$
    \mathcal{H}^1_i(3) = \bigcup_{g=(a,b), |a|, |b| \leq 1} g\cdot \partial F^{\ext}_i.
    $$
    Let $H$ be the unique face of $\mathcal{H}^1_i(3)$ that is not a face of $\mathcal{H}^1_i$.

    Since $d(\mathcal{P}_1, \mathcal{P}_2) \leq C_i$, the diameter of $\rho_H(\Teich_{K_i}(\mathcal{H}^1_i)) = \pi_H \circ \tau_{\mathcal{H}^1_i(3), \mathcal{H}^1_i}(\Teich_{K_i}(\mathcal{H}^1_i))$ is bounded by $C_i$, and hence it is contained in some compact set $K_H \subseteq \Teich(\Pi_{H})$.
    Consider the map $\pi_{Q_i} \circ \tau_{\mathcal{H}^0_i, \mathcal{H}^1_i(3)}: \Teich(\mathcal{H}^1_i(3)) = \Teich(\Pi_{P_i})^9 \times \Teich(\Pi_{H}) \longrightarrow \Teich(\Pi_{P_i})$.
    Thus, there exists $\delta<1$ so that 
$$|d(\pi_{Q_i} \circ \tau_{\mathcal{H}^0_i, \mathcal{H}^1_i(3)})| \leq \delta$$
    on the compact set $K_i^9 \times K_H \subseteq \Teich(\Pi_{P_i})^9 \times \Teich(\Pi_{H})$.
    The lemma now follows from a similar argument as in Lemma \ref{lem:cf}.
\end{proof}

\begin{proof}[Proof of Theorem \ref{thm:Z2EU}]
    The existence part follows by a similar argument as in Theorem \ref{thm:exc} by taking geometric limit of finite circle packings.

    To prove the uniqueness, let $\mathcal{P}_1, \mathcal{P}_2\in \Teich(\mathcal{G}_i)$ be two circle packings with the same external class.
    Let $\mathcal{P}^n_i$ be the sub-circle packings of $\mathcal{P}_i$ associated to $\mathcal{G}^n_i$.
    By applying Theorem \ref{thm:bitr} to subdivisions of $P_i$, and combining Corollary \ref{cor:cf} and Lemma \ref{lem:uci}, the same argument as in Theorem \ref{thm:rig} gives
    $$
    d(\mathcal{P}^n_1, \mathcal{P}^n_2) = O(\delta^n).
    $$
    Since this is true for any $n$, we conclude that $\mathcal{P}_1, \mathcal{P}_2$ are conformally homeomorphic.

    A similar argument also allows us to prove the case for graph $\mathcal{G}$.
\end{proof}

\section{Hyperbolic 3-manifolds with gasket limit sets}\label{appx:topology}
In this appendix, we characterize topologically hyperbolic 3-manifolds whose limit sets are homeomorphic to circle packings. Our main results here are Proposition~\ref{prop:comp_body} and Corollary~\ref{cor:geom_finite}.


\subsection{Acylindricity from limit sets}\label{subsec:acy}
In this subsection, we recall a well-known result on recognizing acylindricity from the limit set in the geometrically finite case, in the spirit of Theorems~\ref{thm:jordan} and \ref{thm:acy}. We start with the following lemma:
\begin{lem}\label{lem:pfp}
Let $\Gamma$ be any nonelementary Kleinian group, and $\Omega_1$, $\Omega_2$ two connected components of its domain of discontinuity $\Omega$. Let $\Gamma_i$ be the stabilizer of $\Omega_i$ in $\Gamma$, and assume $X_i:=\Gamma_i\backslash\Omega_i$ is a Riemann surface of finite type. Suppose $\overline{\Omega_1}\cap\overline{\Omega_2}$ consists of one point $p$. Then
\begin{enumerate}[label=\normalfont{(\arabic*)}]
    \item $p$ is a parabolic fixed point;
    \item Let $\sigma_i$ be a curve on $X_i$ that is not null-homotopic in $\overline{M}:=\Gamma\backslash\{\mathbb{H}^3\cup\Omega\}$. Then $\sigma_1$ and $\sigma_2$ are not homotopic in $\overline{M}$.
\end{enumerate}
\end{lem}
\begin{proof}
By \cite[Theorem~3]{Mas74}, $p=\overline{\Omega_1}\cap\overline{\Omega_2}=\Lambda(\Gamma_1\cap\Gamma_2)$, so it must be a parablic fixed point. Moreover, if two curves $\sigma_1$ and $\sigma_2$ are homotopic, then $\Gamma_1\cap\Gamma_2$ contains a hyperbolic element, and thus $\overline{\Omega_1}\cap\overline{\Omega_2}$ contains at least two points.
\end{proof}
\begin{prop}[Characterization of geom.\ finite acy.\ Kleinian groups]\label{prop:acy_geom_fin}
Suppose $\Gamma$ is nonelementary and geometrically finite of infinite volume. Then $\Gamma$ is acylindrical if and only if any connected component of the domain of discontinuity $\Omega$ is a Jordan domain, and the closures of any pair of connected components share at most one point. Moreover, any common point of the closures of two connected components is a parabolic fixed point, and any rank-1 parabolic fixed point arises this way.
\end{prop}
\begin{proof}
For the ``only if" and the ``moreover" part, see \cite[Lemma~11.2]{BO22}. For the ``if" part, note that each boundary component of $\overline{M}$ is incompressible since each component of $\Omega$ is simply connected. Moreover, $\overline{M}$ is acylindrical, for any essential cylinder gives curves on $\partial_0N$ violating Part (2) of the previous lemma. 
\end{proof}

Recall that the quasiconformal deformation space $\mathcal{QC}(\Gamma)$ can be identified with the Teichm\"uller space $\Teich(\partial \overline{M})$ if every boundary component is incompressible. Moreover, when $\Gamma$ is acylindrical, there exists a unique geometrically finite $\Gamma'\in \mathcal{QC}(\Gamma)$ so that the convex core of the corresponding 3-manifold has totally geodesic boundary \cite[Corollary 4.3]{McM90}. This implies the limit set $\Lambda$ of $\Gamma$ is quasiconformally homeomorphic to the complement of countably many open round disks.

\subsection{Compression bodies and gasket limit sets}\label{subsec:comp_gasket}
In this section, we describe what hyperbolic 3-manifolds associated to Kleinian circle packings look like. In fact, we work in a more general setting and consider Kleinian groups with gasket limit sets. Here, a {\em gasket} is a closed subset of $\widehat\C$ so that
\begin{enumerate}
    \item each complementary component is a Jordan domain;
    \item any two complementary components touch at most at 1 point;
    \item no three complementary components share a boundary point;
    \item the {\em nerve}, obtained by assigning a vertex to each complementary component and an edge if two touch, is connected.
\end{enumerate}
Clearly the limit set of a circle packing is a gasket.

For the remainder of this subsection, we assume $\Gamma$ is finitely generated, whose limit set $\Lambda$ is a gasket with nerve $\mathcal{G}$. Let $\core'_\epsilon(M)$ be the convex core of the corresponding hyperbolic 3-manifold $M$ minus $\epsilon$-thin cuspidal neighborhoods for all cusps of rank $1$. We will show
\begin{prop}\label{prop:comp_body}
The manifold $M$ is homeomorphic to the interior of a compression body $(N,\partial N)$, with exterior boundary $\partial_eN$ homeomorphic to $S=\partial \core'_\epsilon(M)$.
\end{prop}
This proposition is closely related to the characterization in \cite{HPW23}, cf. Theorems B and C there.

\subsection*{(Punctured) compression bodies}
Recall that a compression body is a compact, orientable, irreducible manifold with boundary $(N,\partial N)$ homeomorphic to the boundary connected sum of a solid 3-ball with a collection of solid tori and a collection of trivial interval bundles over closed surfaces.

More generally, a \emph{punctured} compression body also allows in the definition above gluing trivial interval bundles over punctured surfaces. In this case, the exterior boundary is a punctured surface, so are some of the interior boundary components. From now on for simplicity, when we use ``compression body", the punctured case is also included. It is easy to see that in this general setting, compression bodies are still characterized by the fact that the inclusion of the exterior boundary induces a surjection on $\pi_1$.

\subsection*{Gasket limit sets}
We are now ready to prove Proposition~\ref{prop:comp_body}. We start with the following lemma.
\begin{lem}\label{lem:comp_body}
The boundary $S:=\partial\core'_\epsilon(M)$ is connected, and the map $\pi_1(S)\to\pi_1(M)$ induced by inclusion is a surjection.
\end{lem}
\begin{proof}
Note that $X:=\Gamma\backslash\Omega$ is a union of finite area hyperbolic surfaces by the Ahlfors finiteness theorem. Components of $X$ are in one-to-one correspondence with components of $\partial\core(M)$. Moreover, $\partial\core(M)\cap\partial\core_\epsilon'(M)$ consists of $\epsilon$-thick parts of components of $\partial\core(M)$, and $\partial\core_\epsilon'(M)-\partial\core(M)$ consists of boundaries of cuspidal neighborhoods.

Each component $\widetilde{S'}$ of $\partial\chull(\Omega)$ faces a component $\Omega'$ of $\Omega$, and covers a component $S'$ of $\partial\core(M)$. The $\epsilon$-part of $S'$ is covered by $\widetilde{S'}$ minus disjoint horoballs based at all the parabolic fixed points on $\partial\Omega'$. Note that by Lemma~\ref{lem:pfp}, these in particular include all common boundary points between $\Omega'$ and another component of $\Omega$. In fact by Proposition~\ref{prop:acy_geom_fin}, in the geometrically finite case, parabolic fixed points on $\partial\Omega'$ are precisely common boundary points, but in general the former also include singly cusped parabolic points.

Furthermore, lifts of each component of $\partial\core_\epsilon'(M)-\partial\core(M)$ are either flat strips bounded by two horocycles based at a parabolic fixed point at infinity (when the parabolic fixed point is a common boundary point), or half planes bounded by a single horocycle (when the parabolic fixed point is singly cusped).

Thus the preimage of $S$ under the projection $\pi:\mathbb{H}^3\to M$ can be described as follows. Take $\partial\chull(\Omega)$, minus a horoball based at each parabolic fixed point, and then add back the flat strips or half planes on those horoballs contained in $\chull(\Omega)$. Note that for any pair of components $\Omega',\Omega''$ of $\Omega$ sharing a common boundary point, the corresponding components $\widetilde{S'},\widetilde{S''}$ of $\partial\chull(\Omega)$ are connected by a strip. As the nerve $\mathcal{G}$ is connected, this implies that $\pi^{-1}(S)$ is connected, and hence $S$ is connected as well.

Finally note that $\Gamma$ acts on $\pi^{-1}(S)$ freely and properly discontinuously. Hence $\Gamma\cong\pi_1(M)$ is a quotient of $\pi_1(S)$, as desired.
\end{proof}

\begin{proof}[Proof of Proposition~\ref{prop:comp_body}]
By the tameness theorem of Agol \cite{Ago04} and Calegari-Gabai \cite{CG06}, $M$ is homeomorphic to the interior of a compact 3-manifold with boundary $N$. If there are any singly cusped parabolic points, we modify $N$ by deleting an annular neighborhood on $\partial N$ of the corresponding simple curve representing the parabolic element.

By the previous lemma, one boundary component of $N$ induces a surjection on $\pi_1$, so $N$ must be a compression body, as desired.
\end{proof}
Moreover, we note that any non-toroidal interior boundary component of $N$ must correspond to a geometrically infinite end of $M$, for otherwise the limit set is not a gasket. Also, as singly cusped parabolic points only exist when $M$ is geometrically infinite, $N$ is compact (i.e.\ not punctured) when $M$ is geometrically finite. Thus we have the following corollary:
\begin{cor}\label{cor:geom_finite}
    The 3-manifold $M$ is geometrically finite if and only if the interior boundary of $N$ is empty or consists entirely of tori. Moreover, in this case
    \begin{itemize}
        \item the compression body $N$ is compact;
        \item the 3-manifold $M$ is acylindrical; and hence
        \item the limit set $\Lambda$ is quasiconformally homeomorphic to the limit set of a Kleinian circle packing.
    \end{itemize}
\end{cor}
\begin{proof}
    The fact that $M$ is acylindrical follows from Proposition~\ref{prop:acy_geom_fin}, and the last statement then follows from \cite[Cor.~4.3]{McM90}.
\end{proof}
So as soon as we have a geometrically finite gasket limit set, we can quasiconformally deform it to obtain a circle packing. In practice, it is much easier to construct geometrically finite Kleinian circle packings directly, e.g. the Apollonian gasket, and many examples studied in \cite{LLM22}.

\subsection{Topology from the nerve.}
We want to extract more topological information about the 3-manifold from the nerve. For example, is it possible to see whether the interior boundary of the compression body $N$ is empty from $\mathcal{G}$? Since the topology of $(N,\partial N)$ determines many geometric properties of $M$ (e.g.\ geometric finiteness by Corollary~\ref{cor:geom_finite}), this is the same as asking if these geometric properties can be read from $\mathcal{G}$.

By the discussion above, four distinct cases arise:
\begin{enumerate}
    \item The interior boundary of $N$ is empty, i.e.~$N$ is a handlebody;
    \item The interior boundary of $N$ consists only of tori;
    \item $N$ is homeomorphic to $\partial_eN\times [0,1]$;
    \item The interior boundary of $N$ has at least one non-toroidal component, and $N$ is not homeomorphic to $\partial_eN\times [0,1]$.
\end{enumerate}
The first two cases have been discussed in \S\ref{sec:kcp}. We now discuss the remaining two cases.

\subsection*{Thickened surfaces and trees}
We can distinguish the third case easily from the graph $\mathcal{G}$.
\begin{prop}\label{prop:tree}
The compression body $N$ is homeomorphic to $\partial_eN\times [0,1]$ if and only if $\mathcal{G}$ is a tree. In this case, $M$ has a geometrically finite end and an geometrically infinite end.
\end{prop}
\begin{proof}
Note that in the proof of Lemma~\ref{lem:comp_body}, $\pi^{-1}(S)$ is simply connected if and only if $\mathcal{G}$ is a tree, and thus $\Gamma\cong\pi_1(S)$ if and only if $\mathcal{G}$ is a tree.

Moreover, $\Gamma\backslash\Omega$ corresponds to a geometrically finite end of $M$. The other end must be geometrically infinite, for otherwise some component of $\Omega$ meets others in infinitely many points.
\end{proof}
It is easy to see that for the last case, we can find a collection of compressing disks to cut off non-toroidal interior boundary components. The structure of the nerve graph is thus a hybrid of the geometrically finite case and the tree case. We will remark on this case in greater detail later.

We remark that we can construct examples of Kleinian circle packings whose nerve is a tree. Indeed, Let $S$ be a surface of genus $g\ge2$, and $\mu$ a maximal collection of disjoint simple closed curve on $S$ (in particular $S-\mu$ is a collection of pairs of pants), and $\lambda$ a minimal lamination so that $S-\lambda$ consists of ideal polygons. By \cite[Theorem~1.3]{NS12}, there exists a Kleinian surface group with ending lamination $\lambda$, and parabolic locus $\mu$. The convex core boundary consists of three-punctured spheres, so must be totally geodesic. Therefore each domain of discontinuity is a round disk. As mentioned in the introduction, varying $\lambda$ while fixing $\mu$ gives conformally different packings with the same nerve, so no rigidity can be obtained in general for this case.

An entirely analogous construction also provides many examples for the last case, again applying \cite[Theorem~1.3]{NS12}.

\subsection*{With non-toroidal interior boundary components}
When $N$ has at least one non-toroidal interior boundary component, and is not homoemorphic to $\partial_eN\times[0,1]$, we can still cut off each interior boundary component with a compressing disk. However, it might not be possible to modify these disks so that the corresponding cycles in the nerve $\mathcal{G}$ bounds a Jordan domain.

Moreover, choose a compressing disk $D$ so that cutting $N$ along $D$ produces a component homeomorphic to the trivial interval bundle over a non-toroidal interior boundary component, and let $P$ be the polygon in $\mathcal{G}$ bounded by the cycle determined by $D$. Then $\mathcal{G}-\bigcup_{\gamma\in\Gamma}\gamma\cdot P$ contains no loops, due to a similar argument as the tree case.

\subsection*{Determining topology from the nerve}
We come back to the question of extracting topological information about the 3-manifold from the nerve. We have the converse to the construction in \S\ref{subsec:ksr}, giving the ``only if" part of Theorem~\ref{thm:ksdr}.
\begin{prop}\label{prop:topology}
    If $\mathcal{G}$ is a spherical subdivision graph for a finite subdivision rule, then $N$ is a handlebody. Similarly, if $\mathcal{G}$ is a spherical subdivision graph for a $\mathbb{Z}^2$-subdivision rule, then the interior boundary of $N$ consists of tori.
\end{prop}
\begin{proof}
    We have the topology of $N$ has one of the types (1)-(4) listed just above Proposition~\ref{prop:tree}. When $N$ is not a handlebody, say of type (2), then the corresponding Kleinian group has a rank-2 parabolic subgroup. Smaller and smaller polygons containing the corresponding parabolic fixed point has more and more sides, so no finite subdivision rule exists. Similar arguments can be made for all other cases.

    Suppose now $\mathcal{G}$ is a spherical subdivision graph for a $\mathbb{Z}^2$-subdivision rule. Then we have a $\mathbb{Z}^2$-action on $\mathcal{G}$ coming from the subdivision rule. If $N$ is a handlebody, by combinatorial rigidity in this case the $\mathbb{Z}^2$-action realizes as M\"obius transformations fixing the circle packing. But this is only possible with rank-2 cusps.
    
    On the other hand, clearly $\mathcal{G}$ cannot be a tree.
    Finally, if $N$ is of type (4), the limit set of any subgroup $\Gamma'$ of $\Gamma$ corresponding to a geometrically infinite end is surrounded by cycles of disks in the packing. The length of the cycle grows exponentially as we approach the limit set (since $\Gamma'$ contains free groups); but this means $\mathcal{G}$ does not come from a $\mathbb{Z}^2$-subdivision rule, since the length of the cycles can only grow linearly in that case.
\end{proof}

\end{document}